\documentclass[11pt,twoside,letterpaper]{article}
\usepackage[english]{babel}
\usepackage[T1]{fontenc}
\usepackage{amsmath}
\usepackage{amssymb}
\usepackage{amsthm}
\usepackage{graphicx}
\usepackage{mathrsfs}

\newtheorem{theorem}{Theorem}
\newtheorem{proposition}[theorem]{Proposition}
\newtheorem{definition}{Definition}
\newtheorem{lemma}{Lemma}
\newtheorem{remark}{Remark}  
\textheight = 8.0in \def\footnote{\@ifnextchar[{\@xfootnote}{\stepcounter{\@mpfn}\xdef\@thefnmark{\thempfn}\@footnotemark\@footnotetext}}
\newcommand{\ud}{\mathrm{d}}

\newcommand{\1}{\mathbf{1}}
\newcommand{\tb}{t_{\mathbf{b}}}

\newcommand{\xb}{x_{\mathbf{b}}}
\newcommand{\e}{\varepsilon}

\newcommand{\R}{\mathbb{R}}
\newcommand{\p}{\partial}
\usepackage{ulem}
\usepackage{color}
\begin{document}

\title{BV-regularity of the Boltzmann equation\\ in Non-Convex Domains} \author{Y. Guo, C. Kim, D. Tonon, A. Trescases
}
\date{}
\maketitle

\footnotetext{
Last update : \today}
 \begin{abstract}
Consider the Boltzmann equation in a general non-convex domain with the diffuse boundary condition. We establish optimal BV estimates for such solutions. Our method consists of a new $W^{1,1}-$trace estimate for the diffuse boundary condition and a delicate construction of and an $\varepsilon-$tubular neighborhood of the singular set. %We also illustrate that such $BV-$regularity is rather optimal in the case of strictly non-convex domains.  
\end{abstract} 

\section{Introduction}

Boundary effects play an important role in the dynamics of Boltzmann solutions for 
\begin{equation}\label{Boltzmann}
\partial_{t} F + v\cdot \nabla_{x} F = Q(F,F),
\end{equation}
where $F(t,x,v) \geq 0$ denotes the particle distribution in the phase space $\Omega\times\mathbb{R}^{3}$. Throughout this paper, the collision operator takes the form 
\begin{equation}\label{collision_Q}
\begin{split}
Q(F_{1},F_{2})
&:=Q_{\text{gain}}(F_{1},F_{2})-Q_{\text{loss%
}}(F_{1},F_{2})\\
&=\int_{\mathbb{R}^{3}}\int_{\mathbb{S}^{2}}|v-u|^{\kappa
}q_{0}(\theta )\Big[F_{1}(u^\prime)F_{2}(v^\prime)-F_{1}(u)F_{2}(v)\Big]%
\mathrm{d}\omega \mathrm{d}u,
\end{split}
\end{equation}
where $u^\prime=u+[(v-u)\cdot \omega ]\omega ,\ v^\prime=v-[(v-u)\cdot
\omega ]\omega $ and $0\leq \kappa \leq 1$ (hard potential) and $0\leq
q_{0}(\theta )\leq C|\cos \theta |$ (angular cutoff) with $\cos \theta =%
\frac{v-u}{|v-u|}\cdot \omega $. We denote the global Maxwellian
\[
\mu(v) = \exp \Big( - \frac{|v|^{2}}{2}\Big).
\]

Throughout this paper we assume that $\Omega$ is a bounded open subset of $\mathbb{R}^{3}$. The boundary $\partial\Omega$ is locally a graph of a given $C^{2}$ function: for each point $x_{0} \in \partial\Omega$ there exist $r>0$ and a $C^{2}$ function $\eta: \mathbb{R}^{2} \rightarrow \mathbb{R}$ such that, upto a rotation and relabeling, we have
\begin{equation}\label{boundary_graph}
\begin{split}
\partial \Omega \cap B(x_{0}; r) = \big\{ x \in B(x_{0}; r) : x_{3} = \eta (x_{1},x_{2}) \big\},\\
  \Omega \cap B(x_{0}; r) = \big\{ x \in B(x_{0}; r) : x_{3} > \eta (x_{1},x_{2}) \big\}.
 \end{split}
\end{equation}

The boundary of the phase space $\Omega\times\mathbb{R}^{3}$ is
\begin{equation}\label{phase_bdry}
\gamma:= \{(x,v) \in\partial\Omega \times\mathbb{R}^{3}\}.
\end{equation}
We denote $n=n(x)$ the outward normal direction at $x\in \partial \Omega$. %: $n = \frac{1}{\sqrt{1+ |\nabla \eta|^{2} }} \big(  \nabla \eta ,  -1  \big)^{T}$ upto rotation and relabeling. 
 We decompose $\gamma$ as 
\begin{equation}\notag
\begin{split}
\gamma
_{-}&=\{(x,v)\in \partial \Omega \times \mathbb{R}^{3}: n({x})\cdot v<0\}, \ \ \  \ \ \ \ \ \ (\textit{the incoming set}),\\
\gamma
_{+}&=\{(x,v)\in \partial \Omega \times \mathbb{R}^{3}: n({x})\cdot v>0\}, \ \ \  \ \ \ \ \ \ (\textit{the outgoing set}),\\
\gamma _{0}&=\{(x,v)\in \partial \Omega \times \mathbb{R}^{3}:  n({x}%
) \cdot v=0\} , \ \ \  \ \ \ \ \ \ (\textit{the grazing set}).
\end{split}
\end{equation}

In general the boundary condition is imposed only for the incoming set $\gamma_{-}$ for general kinetic PDEs. We consider \textit{the diffuse boundary condition} in this paper: {for $(x,v)\in\gamma_{-}$}
\begin{equation}\label{diffuseBC_F}
F(t,x,v)=c_{\mu }\mu (v)\int_{n(x)\cdot  u>0}F(t,x,u)\{n(x)\cdot u\}\mathrm{d}u,
\end{equation}%
with $c_{\mu }\int_{n(x)\cdot u>0}\mu (u)\{n(x)\cdot
u\}\mathrm{d}u=1$.

Despite extensive developments in the study of the Boltzmann equation, many basic questions regarding solutions in a physical bounded domain, such as their regularity, have remained largely open. This is partly due to the characteristic nature of boundary conditions in {kinetic theory}: Consider the simple transport equation $v\cdot \nabla_{x} f(x,v)=0$ with the given boundary condition $f|_{\gamma_{-}}=g.$ Then we solve $f(x,v) = g(x_{\mathbf{b}}(x,v),v)= g(x-t_{\mathbf{b}}(x,v)v,v)$ where $\tb(x,v)$ is the \textit{backward exit time} defined as
\begin{equation}
\begin{split}
\tb(x,v)  &:=\sup (  \{0\} \cup   \{\tau >0:x-sv\in \Omega \text{ for all }0<
s < \tau \}),
\\ \xb(v)&:=x-\tb(x,v)v.   \label{backexit}
\end{split}
\end{equation} 
{Similarly the \textit{forward exit time} $t_{\mathbf{f}}$ is defined as}
\begin{equation}\label{forwardexit}
\begin{split}
t_{\mathbf{f}}(x,v)&:=\sup (  \{0\} \cup   \{\tau >0:x+sv\in \Omega \text{ for all }0<
s < \tau \}),\\
x_{ \mathbf{f}}(v)&:=x+t_{\mathbf{f}}(x,v)v.  
\end{split}
\end{equation} 
Since $x_{\mathbf{b}}(x,v)$ has singular behavior (even not continuous) if $n(x_{\mathbf{b}}(x,v))\cdot v=0${, we expect $f$} might be singular on \textit{the singular set}: 
\begin{equation}\label{singular_set}
\mathfrak{S}_{\mathrm{B}} : = \{  (x,v) \in  \bar{\Omega} \times \mathbb{R}^{3} :  n(x_{\mathbf{b}} (x,v)) \cdot v=0  \}, 
\end{equation}
which is the collection of all the characteristics emanating from the grazing set $\gamma_{0}$. 

In \cite{Guo10}, it is shown that in convex domains, Boltzmann solutions are continuous away from the grazing set $\gamma_{0}$. On the other hand, in \cite{Kim11}, it is shown that the singularity
(discontinuity) does occur for Boltzmann solutions in a non-convex domain,
and such singularity propagates along the singular set $\mathfrak{S}_{\mathrm{B}}$. Very recently the authors were able to establish weighted $C^{1}$ estimates in convex domains for all basic boundary conditions. The main purpose of this paper is to establish the first BV regularity estimate for the Boltzmann solution in non-convex {domains}.

\vspace{4pt}

{We denote $||\cdot||_{\infty}$ the $L^{\infty}(\bar{\Omega} \times \mathbb{R}^{3})$ norm, while $|| \cdot ||_{p}$ is the $L^{p}(\Omega\times\mathbb{R}^{3})$ norm. We denote $|\cdot|_{p}$ the $L^{p} (\partial\Omega\times \mathbb{R}^{3},\mathrm{d}S_{x} \mathrm{d}v)$ and $|\cdot|_{\gamma,p}$ the $L^{p} (\partial\Omega\times \mathbb{R}^{3} )=L^{p} (\partial\Omega\times \mathbb{R}^{3},\mathrm{d}\gamma)$ norm where $\mathrm{d}\gamma = |n(x)\cdot v| \mathrm{d}S_{x} \mathrm{d}v$ with the surface measure $\mathrm{d}S_{x}$ on $\partial\Omega.$ We write $|\cdot|_{\gamma_\pm,p} = |\cdot \mathbf{1}_{\gamma_{\pm}}|_{\gamma,p}$. For a function $f$ on $\Omega\times\mathbb{R}^{3}$, we denote $f_\gamma$ to be its trace on $\gamma$ whenever it exists.}

A function $f \in L^1(\Omega\times \mathbb{R}^{3})$ has \textit{bounded variation} in $\Omega\times\mathbb{R}^{3}$ if
\[
\sup\Big\{ \iint_{\Omega\times\mathbb{R}^{3}} f \mathrm{div} \varphi \ud x \ud v : \varphi \in C^1_{c}(\Omega\times \mathbb{R}^{3};\mathbb{R}^3 \times\mathbb{R}^{3}), \ |\varphi| \leq 1
\Big\} < \infty.
\]
We define
\[
|| f ||_{BV} := || f||_{L^1(\Omega)} + ||  f ||_{\tilde{BV}},
\]
where
\[
||  f ||_{\tilde{BV}} := \sup\Big\{ \iint_{\Omega \times \mathbb{R}^{3}}f \mathrm{div} \varphi \ud x \ud v: \varphi \in C^1_{c}(\Omega\times \mathbb{R}^{3};\mathbb{R}^3 \times\mathbb{R}^{3}), \ |\varphi| \leq 1
\Big\} < \infty.
\]
%
%We define $f \in BV(\gamma)$ if and only if 
%\begin{equation}\label{BV_gamma}
%f\in L^{1}(\partial\Omega \times \mathbb{R}^{3}, \mathrm{d}\gamma) \ \ \text{and} \ \ 
%\mathrm{sup} \Big\{\int_{\partial\Omega \times \mathbb{R}^{3}} f \mathrm{div}_{\gamma} \varphi \mathrm{d}\gamma :  \varphi \in C^{1}_{c}(\partial\Omega \times \mathbb{R}^{3}; \mathbb{R}^{5}), \ |\varphi|  \leq 1  \Big\}< \infty,
%\end{equation}
%where $\mathrm{div}_{\gamma} \varphi$ is defined locally as following: Fix $(x_{0},v_{0}) \in \partial\Omega \times \mathbb{R}^{3}$. Then for $(x,v) \in \partial\Omega \times \mathbb{R}^{3} \cap B((x_{0},v_{0}), \delta)$ for $0< \delta \ll1 $ we define
%\[
%\mathrm{div}_{\gamma} f(x,v) = \Big( \frac{\partial}{\partial x_{\tau }^{1}}( f\circ \Phi) , \frac{\partial}{\partial x_{\tau}^{2}} ( f\circ \Phi),  \frac{\partial}{\partial v_{\tau}^{1}}   ( f\circ \Phi),   \frac{\partial}{\partial v_{\tau}^{2}}   ( f\circ \Phi), \frac{\partial}{\partial v_{n}} ( f\circ \Phi) 
% \Big)^{T} \in \mathbb{R}^{5},
%\] 
%where $\Phi$....

\begin{theorem}\label{main_theorem}
Let $\Omega$ be a bounded open subset of $\mathbb{R}^{3}$ with $C^{2}$ boundary $\partial\Omega$ as in (\ref{boundary_graph}). Assume that $0 \leq \kappa \leq 1$ in (\ref{collision_Q}), $F_{0} = \sqrt{\mu}f_{0} \geq 0,$ $f_{0}\in BV(\Omega\times\mathbb{R}^{3}),$ and $|| e^{\theta |v|^{2}} f_{0} ||_{\infty} < + \infty$ for $0< \theta < \frac{1}{4}$. Then there exists $T=T(|| e^{\theta |v|^{2}} f_{0} ||_{\infty})>0$ such that $ F =\sqrt{\mu} f$ solves the Boltzmann equation (\ref{Boltzmann}) with the diffuse boundary condition (\ref{diffuseBC_F}) and $f \in L^{\infty} ([0,T]; BV(\Omega\times \mathbb{R}^{3}))$ and $\nabla_{x,v}f  \mathrm{d}\gamma$ is a Radon measure on $\partial\Omega \times \mathbb{R}^{3}$. 

Moreover, for all $0 \leq t \leq T$, 
\begin{equation}\label{est_main}
\begin{split}
||  f(t) ||_{BV} 
 \ \lesssim_{t,\Omega} \  ||   f_{0} ||_{BV} + P (|| e^{\theta |v|^{2}}  f_{0} ||_{\infty}),
\end{split}
\end{equation}
for some polynomial $P$ and $\nabla_{x,v}f_{\gamma}$ is a Radon measure $\sigma$ on $\partial\Omega \times \mathbb{R}^{3}$ such that $|\sigma(\partial\Omega\times \mathbb{R}^{3})| \lesssim ||   f_{0} ||_{BV} + P (|| e^{\theta |v|^{2}}  f_{0} ||_{\infty}). $
\end{theorem}

We remark that {the result holds even without} any size restriction for the initial datum. On the other hand, if $|| e^{\theta |v|^{2}} g_{0} ||_{\infty}\ll1$ for $F_{0 }= \mu + \sqrt{\mu} g_{0}\geq 0$, then Theorem \ref{main_theorem} holds for $g(t)$ for all $t\geq 0.$ Moreover the $BV$ regularity (even in the bulk) is the best regularity we can expect. The reason is that in general the singular set $\mathfrak{S}_{\mathrm{B}}$ is a co-dimension 1 subset in the phase space $\Omega\times\mathbb{R}^{3}$.

\begin{remark} Assume that the domain $\Omega$ is non-convex: there exists at least one point $x_{0}\in\partial\Omega$ and $u \in \mathbb{R}^{3}$ and $(u_{1},u_{2}) \neq 0$ such that (\ref{boundary_graph}) and 
\begin{equation}\label{nonconvex}
\sum_{i,j=1,2}  u_{i} u_{j} \partial_{i} \partial_{j} \eta(x_{0}) < 0,   \ \ \  \ \ \ \ \ \ (\textit{strictly non-convex point}).
\end{equation}
Then the singular set $\mathfrak{S}_{\mathrm{B}}$ is a co-dimension 1 subset of $\Omega \times \mathbb{R}^{3}$. Moreover if we restrict the singular set to the characteristics emanating from the strictly non-convex points
\[
\big\{ (x,v) \in \mathfrak{S}_{\mathrm{B}}:  (x_{\mathbf{b}}(x,v),v) \text{ is a strictly non-convex point} \big\},
\]
then this set is a co-dimension 1 submanifold of $\Omega\times\mathbb{R}^{3}.$
\end{remark}
We put the proof of Remark 1 in the appendix. Since discontinuous solutions were constructed for non-convex {domains} in \cite{Kim11}, this remark shows that the Boltzmann solutions are singular on the co-dimensional 1 subset $\mathfrak{S}_{\mathrm{B}}$. Then it is standard to conclude that the best possible regularity space is indeed the $BV$ space. Hence Theorem 1 is optimal.

\vspace{4pt}

The equation for $f= F/ \sqrt{\mu}$ where $F$ solves (\ref{Boltzmann}) is
\begin{equation}
 \p_{t} f +v\cdot \nabla_{x} f + \nu(\sqrt{\mu}f)f  =  \Gamma_{\mathrm{gain}}(f,f), \ \ \ \ \text{ in } \Omega\times\mathbb{R}^{3}, \label{eqnt_F}\\
 \end{equation}
 where  
\begin{equation}\label{Gamma}
\Gamma_{\mathrm{gain}}(f_{1},f_{2}) := \frac{1}{\sqrt{\mu}} Q_{\mathrm{gain}}(\sqrt{\mu}f_{1},\sqrt{\mu}f_{2}), \ \ \ \nu(\sqrt{\mu}f_{1})f_{2}= \frac{1}{\sqrt{\mu}} Q_{\mathrm{loss}}(\sqrt{\mu}f_{1},\sqrt{\mu}f_{2})   
. 
\end{equation}
The boundary condition for $f= F/ \sqrt{\mu}$ where $F$ satisfies (\ref{diffuseBC_F}) is
 \begin{equation}
f(t,x,v)   =  c_{\mu} \sqrt{\mu(v)}  \int_{n(x)\cdot u>0} f(t,x,u) \sqrt{\mu(u)} \{n(x)\cdot u\} \mathrm{d}u, \ \ \ \ \text{ on } (x,v)\in \gamma_{-}.\label{bdry_f}
\end{equation} 
The local-in-time existence {of the solution $f$} with $\sup_{0 \leq t \leq T}||e^{\theta |v|^{2}} f(t) ||_{\infty} \lesssim ||e^{\theta^{\prime} |v|^{2}} f_{0} ||_{\infty}$ for $0< \theta < \theta^{\prime} < \frac{1}{4}$ is standard (e.g. Lemma 6 in \cite{GKTT}). 

We now illustrate the main ideas of {the proof }of Theorem 1. For simplicity we assume that $f$ satisfies (\ref{bdry_f}) but solves the following simpler problem 
\begin{equation}\label{transport}
\partial_{t} f + v\cdot \nabla_{x} f  + \nu f = H, \ \ \ \ f|_{t=0} =f_{0},
\end{equation}
with the boundary condition (\ref{bdry_f}){, and} where $\nu= \nu(t,x,v)\geq 0,  \ H,$ and $\nu$ are smooth enough. In general solutions $f$ of (\ref{transport}) are discontinuous on $\mathfrak{S}_{\mathrm{B}}$ and (distributional) derivatives do not exist \cite{Kim11}. 

{To take} (distributional) derivatives we consider the following problem with
%, the boundary condition, and the initial datum by multiplying {equations \eqref{bdry_f}, \eqref{transport} by }
some smooth cut-off function $\chi_{\varepsilon}(x,v)$ vanishing on an open neighborhood of $\mathfrak{S}_{\mathrm{B}}$:  
\begin{equation}\label{e-transport}
\begin{split}
 \partial_{t} f^{\varepsilon} + v\cdot \nabla_{x} f ^{\varepsilon} + \nu f^{\varepsilon}  = \chi_{\varepsilon}H \ \ \ &\text{in} \ (x,v) \in \Omega \times \mathbb{R}^{3},\\
f^{\varepsilon}|_{t=0} =\chi_{\varepsilon}f_{0} \ \ \ &\text{in} \ (x,v) \in \Omega \times \mathbb{R}^{3},\\
 f^{\varepsilon}(t,x,v)    = \chi_{\varepsilon} c_{\mu} \sqrt{\mu(v)}  \int_{n(x)\cdot u>0} f^{\varepsilon}(t,x,u) \sqrt{\mu(u)} \{n(x)\cdot u\} \mathrm{d}u, \ \ \ & \text{on} \  (x,v)\in \gamma_{-}.
\end{split}
\end{equation} 
Once we can {show} that $f^{\e}$ is uniformly bounded in $L^{\infty}$ and $\partial f^{\e}$ is uniformly bounded {in $L^1(\Omega\times\R^3)$} then we conclude that $f^{\e}$ converges to $f$ weak$-*$ in $L^{\infty}$ and $f\in BV$ solves (\ref{transport}) with (\ref{bdry_f}). 

Due to the cut-off $\chi_{\varepsilon}$, the solution of (\ref{e-transport}) $f^{\varepsilon}$ vanishes on some open subset of $\bar{\Omega} \times \mathbb{R}^{3}$ containing the singular set $\mathfrak{S}_{\mathrm{B}}$ {defined in} (\ref{singular_set}). Therefore $f^{\varepsilon}$ is smooth. We {apply} (distributional) derivatives $\partial \in \{ \nabla_{x}, \nabla_{v} \}$ to the equation and have {
\[
|\partial_{t} \partial f^{\varepsilon} + v\cdot \nabla_{x} \partial f^{\varepsilon} + \nu \partial f^{\varepsilon}| \leq |\partial f^{\varepsilon}| + |\partial \nu f^{\varepsilon}| + |\partial \chi_{\varepsilon} H| + |\chi_{\varepsilon} \partial H|.
 \]
 }
On the other hand at the boundary we use an orthonormal transformation $\mathcal{T}$ in order to remove a $x-$dependence of the integration range: $\{n(x)\cdot u>0\} \mapsto \{  (\mathcal{T}^{-1}u)_{3}>0 \}$ (see $(17)\sim(21)$ in \cite{GKTT}). Then derivatives of the boundary terms are bounded as {in \cite{GKTT}: }
\[
|\partial f^{\varepsilon}|\sim |\partial \chi_{\varepsilon}|     + \frac{1}{|n\cdot v|} \int_{n\cdot u>0} |\partial f^{\varepsilon}|   \{n\cdot u\} \mathrm{d}u+ \cdots, \ \ \ \  \text{ on }{\gamma_{-}}.
\]

We then apply the energy-type estimate (Green's identity, Lemma \ref{Green}) and the above boundary control to have
\begin{eqnarray}\notag 
&&|| \partial f^{\varepsilon}(t) ||_{1}  + \int_{0}^{t} | \partial f^{\varepsilon}  |_{\gamma_+,1}   
 + \int_{0}^{t} || \nu \partial f^{\varepsilon}(s) || \mathrm{d}s 
\\
  &  \lesssim& \ ||  \partial \chi_{\varepsilon}  f({0}) ||_{1} + \int_{0}^{t} |\partial f^{\varepsilon} |_{\gamma_-,1}  + \int_{0}^{t} ||  \partial \chi_{\varepsilon} H  ||_{1}  + \  \text{``good terms''} \notag\\
  &  \lesssim_{t} & \   \underbrace{|| \partial \chi_{\varepsilon}  ||_{1} +  | \partial \chi_{\varepsilon}   |_{\gamma_{-},1}}_{(A)}  \ + \ \underbrace{C\int_{0}^{t}  
    |\partial f^{\varepsilon} |_{\gamma_{+},1}}_{(B)}  \  + \ \text{``good terms''}. \notag
 \end{eqnarray}

The first main difficulty is to construct a cut-off function $\chi_{\varepsilon}$ such that it vanishes on an open neighborhood of $\mathfrak{S}_{\mathrm{B}}$ and makes $(A)$ be finite at the same time.

We carefully construct, in Lemma \ref{open_cover}, an open neighborhood $\mathcal{O}_{\e}$ of $\mathfrak{S}_{\mathrm{B}}$. More precisely $\mathcal{O}_{\e}$ is a collection of $\e-$tubular neighborhood of forward trajectories emanating from the grazing set $\gamma_{0}$. Also we can show that $\mathcal{O}_{\e}$ contains all points whose distance from $\mathfrak{S}_{\mathrm{B}}$ is less than $\varepsilon$. 
Such $\varepsilon-$thickness is important for constructing cut-off functions. In fact we construct cut-off functions $\chi_{\varepsilon}$ by standard $\varepsilon-$mollifying function of the characteristic function $\mathbf{1}_{\bar{\Omega} \times \mathbb{R}^{3} \backslash   \mathcal{O}_{\e}}$. And the $\varepsilon-$thickness guarantees that the cut-off function vanishes around $\mathfrak{S}_{\mathrm{B}}$. 

Fortunately $\chi_{\varepsilon}$ satisfies the desired bound $(A) < \infty$, that is, $\chi_{\varepsilon}$ is uniformly bounded in $W^{1,1}$ (Lemma \ref{cut_off} and Lemma \ref{small_boundary_lemma}), whose proofs are delicate. Since $\chi_{\e}$ is a standard $\e-$mollification of $\mathbf{1}_{\bar{\Omega} \times \mathbb{R}^{3} \backslash \mathcal{O}_{\e}}$ we have $\partial \chi_{\e}\sim \partial [ 1- \chi_{\e} ] \sim \frac{1}{\e} (\frac{1}{\e^{6}} \mathbf{1}_{|x|+|v|<\e}) * \mathbf{1}_{\mathcal{O}_{\e}}$. For example a desired estimate for $|\partial \chi_{\e}|_{\gamma_{-},1}$ is
\[
\int_{(x,v) \in\gamma_{-}, \ |v|\lesssim1} \mathbf{1}_{\mathcal{O}_{ \varepsilon } } (x ,v ) |n(x) \cdot v|
 \mathrm{d}v \mathrm{d}S_{x} \sim \e.
\]
For fixed $x$, the velocity $v\in \mathcal{O}_{\e}$ is a collection of $\e-$tubular neighborhood of forward trajectories, which happens to pass near $x$ and emanates from $\gamma_{0}$. But there could be infinitely many grazing trajectories passing $x$, which might lead to
\begin{eqnarray*}
&&\int_{(x,v) \in\gamma_{-}, \ |v|\lesssim1} \mathbf{1}_{\mathcal{O}_{ \varepsilon } } (x ,v ) |n(x) \cdot v|
 \mathrm{d}v \mathrm{d}S_{x}\\
 & \sim&
 \{\text{number of grazing at $x$}\} \times 
 \int_{|v|\lesssim 1}  \mathbf{1}_{\e-\text{tubular neighborhood}}(v) |n(x)\cdot v|  \mathrm{d}v\\
 &\sim& \infty.
\end{eqnarray*}
 
%{\color{blue}  I DO NOT UNDERSTAND THE COMPUTATION ABOVE. I WOULD HAVE RATHER WRITTEN (for a fixed $x$)
  
 % \begin{eqnarray*}
%&&\int_{v:(x,v) \in\gamma_{-}, \ |v|\lesssim1} \mathbf{1}_{\mathcal{O}_{ \varepsilon } } (x ,v ) |n(x) \cdot v|
% \mathrm{d}v 
 %\sim O(\varepsilon) \int_{v:(x,v) \in\gamma_{-}, \ |v|\lesssim1} \mathbf{1}_{\mathcal{O}_{ \varepsilon } } (x ,v )
% \mathrm{d}v\\
% & \sim&
% O(\varepsilon) \times 
% \int_{|v|\lesssim 1}  \mathbf{1}_{\e-\text{tubular neighborhood}}(v) \mathrm{d}v\\
% & \sim&
% O(\varepsilon^3) \times \{\text{number of grazing at $x$}\} \\
 %&\sim& \infty.
%\end{eqnarray*}
%  
%}  
  
Instead we establish a geometric Lemma \ref{cone_lemma} to show that $|n(x ) \cdot v| \lesssim \sqrt{\e}$ if $(x,v) \in \mathcal{O}_{\e}$. For the proof, we decompose $\mathcal{O}_{\e}$ carefully in position and velocity with varying grazing trajectories. We remark that $|\partial \chi_{\varepsilon} |_{\gamma_{+},1}<\infty$ may not be true in general.

 \bigskip
 
The second main difficulty is to control the outgoing term $(B)$. We denote \textit{the (outgoing) almost grazing set }
\begin{equation}\label{n-grazing}
\gamma_{+}^{\delta} := \{ (x,v) \in \gamma_{+} : v\cdot n(x) < \delta \ \text{or} \ |v|> 1/\delta \},
\end{equation}
and \textit{the (outgoing) non-grazing set}
\begin{equation}\label{non-grazing}
\gamma_{+} \backslash \gamma_{+}^{\delta}  = \{ (x,v) \in \gamma_{+} : v\cdot n(x) \geq  \delta \ \text{and} \ |v| \leq 1/\delta \}.
\end{equation}
In fact the $\gamma_{+}\backslash \gamma_{+}^{\delta}$ contribution can be controlled by the bulk integration and the initial data by the trace theorem. However {the }$\gamma_{+}^{\delta}$ contribution cannot be bounded by the bulk integration nor $\int_{0}^{t} | \partial f^{\varepsilon}  |_{\gamma_+,1}$ in the LHS of the energy-type estimate since the constant $C>0$ of $(B)$ can be large in general.
%
%In fact for small $\delta>0$ we can control the near-grazing set $\gamma_{+}^{\varepsilon},$ the set of almost grazing velocities or large velocities 
%\begin{equation}\label{n-grazing}
%\gamma_{+}^{\varepsilon} := \{ (x,v) \in \gamma_{+} : v\cdot n(x) < \varepsilon \ \text{or} \ |v|> 1/\varepsilon \}.
%\end{equation}
%Even the $\gamma_{+}\backslash \gamma_{+}^{\varepsilon}$ contribution can be controlled by the integration in the bulk and the initial data by the trace theorem, there is no direct way to control $\gamma_{+}^{\varepsilon}$ in $L^{1}$. 

The new idea is to use the Duhamel formula along the trajectory once again (\textit{Double iteration scheme}) to extract an extra small constant to close the estimate. We evaluate $\partial f^{\e}$ along the characteristics and use the bound of $\partial f^{\e}$ {on $\gamma_{-}$} to have
\begin{eqnarray}
&& \int_{0}^{t} | \partial f^{\varepsilon}   |_{\gamma_+^{\delta},1} \notag \\
 &=& \int_{0}^{t} \int_{(x,v) \in \gamma_{+}^{\delta}} |\partial f^{\e} (s,x,v)|
  \{n(x)\cdot v\}  \mathrm{d}S_{x} \mathrm{d}v 
 \mathrm{d}s \notag
\\
 &\sim & \int_{0}^{t}\mathrm{d}s \int_{(x,v) \in \gamma_{+}^{\delta}}    |\partial f^{\varepsilon}(s-t_{\mathbf{b}}(x,v),x_{\mathbf{b}}(x,v),v)| \{n(x)\cdot v\} \mathrm{d}S_{x} \mathrm{d}v 
 \notag
 \\
&\sim  & 
\int_{0}^{t} \int_{(x,v) \in \gamma_{+}^{\delta}}   \{n(x)\cdot v\}   |\partial \chi_{\varepsilon}(x_{\mathbf{b}}(x,v),v)|\label{frs}\\
 &&+  \int_{0}^{t}  \int_{(x,v) \in \gamma_{+}^{\delta}} \frac{n(x)\cdot v}{n(x_{\mathbf{b}}(x,v))\cdot v}  \int_{n(x_{\mathbf{b}})\cdot u>0} |\partial f^{\varepsilon} (x_{\mathbf{b}}(x,v),u)| \{n(x_{\mathbf{b}}(x,v))\cdot u\} \mathrm{d}u \mathrm{d}S_{x} \mathrm{d}v \mathrm{d}s.\ \ \  \ \ \ \ \label{sec}
\end{eqnarray} 

In Lemma \ref{COV}, we establish a crucial change of variables $(x,v) \mapsto (x_{\mathbf{b}}(x,v),v)$ with $|n(x)\cdot v|\mathrm{d} S_{x} \mathrm{d}v \lesssim 
|n(x_{\mathbf{b}})\cdot v|\mathrm{d} S_{x_{\mathbf{b}}} \mathrm{d}v$. Clearly $(\ref{frs})$ is bounded by $| \partial\chi_{\e}|_{\gamma_{-},1}$. For (\ref{sec}) we use Lemma \ref{COV} to convert $x-$integration into $x_{\mathbf{b}}-$integration and remove {the singular} factor $\frac{n(x)\cdot v}{n(x_{\mathbf{b}}(x,v))\cdot v}$. Furthermore, since $x\in\partial\Omega$ we have $x= x_{\mathbf{b}}(x_{\mathbf{b}}, -v)$ and $(x,v) \in \gamma_{+}^{\delta}$ which implies $(x_{\mathbf{b}}(x_{\mathbf{b}}, -v), v ) \in \gamma_{+}^{\delta}$. Then we can bound the last term by  
\begin{equation}\notag
\begin{split}
\sup_{x_{\mathbf{b}} \in\partial\Omega}\int \mathbf{1}_{(x_{\mathbf{b}}(x_{\mathbf{b}},-v),v) \in \gamma_{+}^{\delta}} \mathrm{d}v \times \int_{0}^{t} |\partial f^{\varepsilon}  |_{\gamma_{+},1}.
\end{split}
\end{equation}
Using the covering lemma of \cite{Guo10} (Lemma \ref{guo_covering}), we are able to extract an extra small constant from $\sup_{x_{\mathbf{b}} \in\partial\Omega}\int \mathbf{1}_{(x_{\mathbf{b}}(x_{\mathbf{b}},-v),v) \in \gamma_{+}^{\delta}} \mathrm{d}v$.

\section{Preliminary}

\begin{lemma}[\cite{Guo10,EGKM}]\label{le:backward_derivatives}
If
\begin{equation}\label{eq:incoming} v\cdot n(x_{\mathbf{b}}(x,v))<0,
\end{equation}
then $(t_\mathbf{b}(x,v),x_\mathbf{b}(x,v))$ are smooth functions of $(x,v)$ such that
\begin{equation*}\begin{split}\nabla_{x} t_{\mathbf{b}}&=\frac{n(x_{\mathbf{b}})}{v\cdot n(x_{\mathbf{b}})}, \qquad \nabla_{v} t_{\mathbf{b}}= -\frac{t_{\mathbf{b}}n(x_{\mathbf{b}})}{v\cdot n(x_{\mathbf{b}})},\\
\nabla_{x} x_{\mathbf{b}}&=I-\frac{n(x_{\mathbf{b}})}{v\cdot n(x_{\mathbf{b}})} \otimes v, \qquad \nabla_{v} x_{\mathbf{b}}= - t_{\mathbf{b}}I + \frac{t_{\mathbf{b}}n(x_{\mathbf{b}})}{v\cdot n(x_{\mathbf{b}})} \otimes v.\end{split}
\end{equation*}
\end{lemma}

Recall the almost grazing set $\gamma _{+}^{\delta }$ defined in (\ref%
{n-grazing}). We first estimate the outgoing trace on $\gamma
_{+}\setminus \gamma _{+}^{\delta }$. We remark that for the outgoing part, our estimate is global in time without cut-off, in contrast to the general trace
theorem.
\begin{lemma}[Outgoing trace theorem, \cite{GKTT}]
\label{le:ukai} Assume that $\varphi \geq 0$. For any small parameter $\delta >0$, there exists a
constant $C_{\delta ,T,\Omega }>0$ such that for any $h$ in $%
L^{1}([0,T]\times\Omega \times \mathbb{R}^{3})$ with $\partial _{t}h+v\cdot
\nabla _{x}h+\varphi h$ {lying} in $L^{1}([0,T]\times\Omega \times \mathbb{R}%
^{3})$, we have for all $0\leq t\leq T,$
\begin{equation*}
\int_{0}^{t}\int_{\gamma _{+}\setminus \gamma _{+}^{\delta }}|h|\mathrm{%
d}\gamma \mathrm{d}s\leq C_{\delta ,T,\Omega }\left[ \
||h_{0}||_{1}+\int_{0}^{t}\big\{\Vert h(s)\Vert _{1}+\big{\Vert} \lbrack \partial
_{t}+v\cdot \nabla _{x}+\varphi ]h(s)\big{\Vert} _{1}\big\}\mathrm{d}s\ \right] .
\end{equation*}%
Furthermore, for any $(s,x,v)$ in $[0,T]\times \Omega \times \mathbb{R}^{3}$
the function $h(s+s^{\prime },x+s^{\prime }v,v)$ is absolutely continuous in
$s^{\prime }$ in the interval $[-\min \{t_{\mathbf{b}}(x,v),s\},\min \{t_{\mathbf{b}}(x,-v),T-s\}]$.
\end{lemma}

\begin{lemma}[Green's Identity, \cite{Guo10,EGKM}]
\label{Green} For $p\in[1,\infty)$ assume that $f,\partial_t f + v\cdot
\nabla_x f  + \varphi f \in L^p ([0,T]\times \Omega\times\mathbb{R}^3)$ with $\varphi \geq 0$ and $%
f_{\gamma_-} \in L^p ([0,T]\times  \partial\Omega \times \mathbb{R}^{3}; \mathrm{d}t\mathrm{d}\gamma)$. Then $f  \in C^0( [0,T];
L^p(\Omega\times\mathbb{R}^3))$ and $f_{\gamma_+} \in
L^p ([0,T]\times\partial\Omega\times\mathbb{R}^{3}; \mathrm{d}t \mathrm{d}\gamma)$ and for almost every $t\in [0,T]$ :
\begin{eqnarray*}
|| f(t) ||_{p}^p + \int_0^t |f |_{\gamma_+,p}^p = || f(0)||_p^p + \int_0^t
|f |_{\gamma_-,p}^p + \int_0^t \iint_{\Omega\times\mathbb{R}^3 }
\{\partial_t f + v\cdot \nabla_x f + \varphi f \} |f|^{p-2}f .
\end{eqnarray*}
\end{lemma}

\begin{lemma}[Lemma 17 and Lemma 18 of \cite{Guo10}]\label{guo_covering}
Let $\Omega\subset \mathbb{R}^3$ be
an open bounded set with a smooth boundary $\partial\Omega.$ Then, for all
$x\in\bar\Omega$, we have
\begin{equation}\label{dong}
\mathrm{m}_3 \{ v\in \mathbb{R}^3: \ n(x_{\mathbf{b}}(x,v))\cdot v =0\}   =  0, \end{equation}
Moreover, for any $\varepsilon>0$ and $N\gg 1$, there exist $ \ \delta_{\varepsilon,N}>0$ and $l=l_{\varepsilon,N ,\Omega}$ balls $ \ B(x_1;r_1), B(x_2;r_2),$
$\cdots, B(x_l;r_l)$ with $x_i \in \bar{\Omega}$ and covering $\bar{\Omega}$ (i.e. $\bar{\Omega} \subset \bigcup B(x_{i};r_{i})$), as well as $l$ open sets $\mathcal{O}_{x_1},
\mathcal{O}_{x_2}, \cdots, \mathcal{O}_{x_l}\subset B_{N} := \{v\in\mathbb{R}^3 :
|v|\leq N \}$, with $\mathrm{m}_3 (\mathcal{O}_{x_i}) < {\varepsilon}$ for all $1\leq i \leq l_{\varepsilon ,N,\Omega},$ such that for any
$x\in\bar\Omega$, there exists $i=1,2,\cdots, l_{\varepsilon, N,\Omega}$ such that $x\in B(x_i;r_i)$ and
\[ |v \cdot
n(x_{\mathbf{b}}(x,v))| \ > \delta_{\varepsilon,N} ,  \ \ \text{for
all} \ v\notin \mathcal{O}_{x_i}. \] 
In particular,
 \begin{equation}\label{guo_covering_inclusion}
\mathcal{O}_{x_i} \ \supset
\ \bigcup_{x\in B(x_i;r_i)}\{ v\in B_{N} : |v\cdot n(x_{\mathbf{b}}(x,v))|\leq
\delta_{\varepsilon,N}\}   .
\end{equation}
\end{lemma}

\begin{proof}
The details of the proof are recorded in \cite{Guo10}. The proof of (\ref{dong}) is due to the Sard's theorem: Fix $x\in\bar{\Omega}$. If $v\in \mathbb{R}^{3}$ satisfies $v\cdot n(x_{\mathbf{b}}(x,v))=0$ and $n(x)\cdot v\neq0 $ then $\frac{v}{|v|}$ is a critical value of the mapping
\[
\phi_{x} : \partial\Omega \rightarrow \mathbb{S}^{2}, \ \ \ \phi_{x} : y \in \partial\Omega 
\mapsto  - \frac{y-x}{|y-x|},
\] 
If $n(x)\cdot v =0$ then $\frac{v}{|v|}$ is a critical value of $\phi_{x}$ at $y= x_{\mathbf{b}}(x,v)$. Then by Sard's theorem the Lebesgue measure of such set on $\mathbb{S}^{2}$ is zero. 

Now we fix $0<\varepsilon \ll 1$ and $x\in\bar{\Omega}$. Due to (\ref{dong}) there exists an open set $\mathcal{O}_{x}\in \mathbb{R}^{3}$ such that $\mathrm{m}_{3} (\mathcal{O}_{x}) < \varepsilon$ and $|v\cdot n(x_{\mathbf{b}}(x,v))| \neq 0$ for $v \notin \mathcal{O}_{x}$. By Lemma \ref{le:backward_derivatives}, $v \mapsto v\cdot n(x_{\mathbf{b}}(x,v))$ is smooth on the compact set $\{\mathbb{R}^{3}  \backslash \mathcal{O}_{x}\} \cap B_{N}$. Then by the compactness we have a positive lower bound $2\delta_{\e, N ,x}>0$ of $|v\cdot n(x_{\mathbf{b}}(x,v) )|$. Then by Lemma \ref{le:backward_derivatives} again, there exists a ball $B(x;r_{x})$ such that for all $y$ in this ball and all $v \in \{ \mathbb{R}^{3} \backslash \mathcal{O}_{x}\}\cap B_{N}$ we have $|v \cdot n(x_{\mathbf{b}} (y,v))|\geq \delta_{\e, N , x}$. Then we use the compactness of $\bar{\Omega}$ to extract the finite covering which satisfies (\ref{guo_covering_inclusion}).
\end{proof}

\section{New Trace Theorem via the Double Iteration}

\begin{proposition}\label{double_iteration}
Let $h_0 \in L^1(\Omega\times\mathbb{R}^{3})$. Let $(h^{m})_{m\ge0}   \subset L^{\infty}  ([0,T] ; L^{1} (\Omega\times\mathbb{R}^{3})) \cap L^{1} ([0,T]; L^{1}(\gamma_{+}, \mathrm{d}\gamma))$ {solve}
\begin{equation}\label{eqtn_h} 
\{\partial_{t} + v\cdot \nabla_{x}  +  \nu\} h^{m+1}  = H^{m}, \ \ \  h^{m+1}|_{t=0} =h_{0},
\end{equation}
where $\nu= \nu(t,x,v)\geq 0$, and {such that the following inequality holds for all $(x,v)\in\gamma_{-}$}
\begin{equation}\label{bdry_h}
\begin{split}
|h^{m+1}(t,x,v)| \leq&   \ C_{1}  \sqrt{\mu(v)}  \Big(1 + \frac{\langle v\rangle }{ |n(x)\cdot v| }\Big)
\int_{n(x)\cdot u>0}  |h^{m}(t,x,u)|  {\mu(u)}^{\frac{1}{4}} \{ n(x)\cdot u\} \mathrm{d}u\\
 & +  \Big(1 + \frac{ e^{-C_{2} |v|^{2}} }{ |n(x)\cdot v| }\Big)R^{m},
  \end{split}
\end{equation}
where $H^{m} \in L^{1}  ([0,T]; L^{1} (\Omega \times \mathbb{R}^{3}))$ and {$R^{m} \in L^{1}  ([0,T]; L^{1} (\partial\Omega\times\mathbb{R}^{3} , \langle v\rangle \ud S_{x}  \ud v))$}.

Then for all $m\geq 1$, $h^{m+1}_{\gamma_{-}} \in L^{1}  ([0,T]; L^{1}(\gamma_{-}, \mathrm{d}\gamma  ))$ and satisfies, for $t\in [0,T]$ {and $0 < \delta \ll 1,$}
\begin{equation}\label{L1trace} 
\begin{split} 
  \int^t_0  |h^{m+1}(s)|_{\gamma_-,1}
&\leq   \ O( \delta) \int_0^t  | h^{m-1}(s)|_{\gamma_+,1}  + C_{ \delta}  || h_{0}||_1  \\
& \ \ +  C_{ \delta}   \max_{i=m,m-1}\Big\{  \int_0^t ||h^{i}(s)||_1
 + \int_{0}^{t} | \langle v\rangle  R^{i}(s)|_{1}
 +\int_0^t ||   H^{i}(s)||_1      \Big\} .   
 \end{split}
\end{equation}

\end{proposition}

Our proof requires the following lemma:

\begin{lemma}\label{COV}
Let $\Omega\subset \mathbb{R}^3$ be
an open bounded set with a smooth boundary $\partial\Omega.$ 

For $k\in\mathbb{N}$, consider the map 
\begin{eqnarray*}
\Phi_{k} : 
\left\{(x,v) \in \gamma_+:  \ n(\xb(x,v))\cdot v < - {1}/{k}
\right\}
&\rightarrow& 
\left\{ (\xb ,v) \in \gamma_{-} :  n( \xb )\cdot v < - {1}/{k}
\right\},\\
(x,v)&\mapsto& \Phi_{k}(x,v):= (\tilde{x},v):=(\xb(x,v),v).
\end{eqnarray*}
Then $\Phi_{k}$ is one-to-one and we have a change of variables formula for all $k \in \mathbb{N}:$
\[
\mathbf{1}_{\{ n(\tilde{x})\cdot v < - 1/k \}}
\left|{n(\tilde{x} )\cdot v}\right|
\mathrm{d}v\mathrm{d}S_{\tilde{x}} \  \geq \ 
 \mathbf{1}_{\{ n(  {\xb (x,v)})\cdot v < - 1/k \}}\left| {n(x)\cdot v}\right| \mathrm{d}v\mathrm{d}S_x.
\]
\end{lemma}
\begin{proof}[\textbf{Proof of Lemma \ref{COV}}]

Let $(x,v)$, $(x',v')\in \gamma_+$ such that $n(\xb(x,v))\cdot v$, $\ n(\xb(x',v'))\cdot v' < - {1}/{k}$. If $\Phi_{k}(x,v) = \Phi_{k}(x^{\prime}, v^{\prime})$ then $v=v^{\prime}$ and $\xb (x,v) = \xb(x^{\prime}, v).$ Since $x= x_{\mathbf{f}} (\xb (x,v),v)= x_{\mathbf{f}} (\xb (x^{\prime},v),v)=x^{\prime}$ the mapping $\Phi_{k}$ is one-to-one. 

Now we prove the change of variable formula. It suffices to consider the small neighborhood of $\partial\Omega$ around $x$. Without loss of generality we may assume $x_{3} = \eta(x_{1}, x_{2})$ for some $\eta : \mathbb{R}^{2} \rightarrow \mathbb{R}$. First we consider the case $n_{3} (\xb(x,v))\neq 0$ so that $$\tilde{x}= x_{\mathbf{b}}(x,v) = (\tilde{x}_{1}, \tilde{x}_{2}, \tilde{x}_{3}) = (\tilde{x}_{1}, \tilde{x}_{2},  \varphi(\tilde{x}_{1}, \tilde{x}_{2})) \in\partial\Omega,$$ for some function $\varphi: \mathbb{R}^2 \rightarrow \mathbb{R}$.
 
The change of variable is given by 
\begin{equation}\label{COV_lemma}
\begin{split}
 \mathrm{d}S_{\tilde{x}} \mathrm{d}v &= \sqrt{1+|\nabla \varphi|^2} \mathrm{d}\tilde{x}_1 \mathrm{d}\tilde{x}_2 \mathrm{d}v\\
& = \frac{\sqrt{1+|\nabla \varphi|^2}}{\sqrt{1+|\nabla \eta|^2}} J \sqrt{1+|\nabla \eta|^2} \mathrm{d}x_1 \mathrm{d}x_2 \mathrm{d}v\\ &=\frac{\sqrt{1+|\nabla \varphi|^2}}{\sqrt{1+|\nabla \eta|^2}} J \mathrm{d}S_x \mathrm{d}v.
\end{split}
\end{equation}
 where $J$ is the Jacobian,
\begin{eqnarray*} 
J=\left|\frac{\partial(\tilde{x}_1,\tilde{x}_2, v_1,v_2,v_3)}{\partial (x_1,x_2,v_1,v_2,v_3)}\right|
=  \left|\begin{array}{ccc}
\partial_{x_1} \tilde{x}_1 & \partial_{x_2}\tilde{x}_1\\
\partial_{x_1} \tilde{x}_2 & \partial_{x_2} \tilde{x}_2
\end{array}\right|.
\end{eqnarray*}
By the definition of $x_{\mathbf{b}}(x,v)$, we have the following identity: $v |x-\tilde{x}|= |v| (x-\tilde{x})$, i.e.
\begin{equation}\label{xtildex}
\begin{split}
&\{(x_1-\tilde{x}_1)^2 + (x_2-\tilde{x}_2)^2 + [\eta(x_1,x_2)-\varphi(\tilde{x}_1,\tilde{x}_2)]^2\}^\frac{1}{2}\left(\begin{array}{ccc}
v_1 \\ v_2 \\ v_3\end{array}\right)\\
& = |v| \left(\begin{array}{ccc}x_1 -\tilde{x}_{1} \\
x_2 -\tilde{x}_2 \\ \eta(x_1,x_2)-\varphi(\tilde{x}_1, \tilde{x}_2)
\end{array}\right).
\end{split}
\end{equation}
%The Jacobian can be computed as
%\begin{eqnarray*} 
%J&=&\left|\frac{\partial(\tilde{x}_1,\tilde{x}_2, v_1,v_2,v_3)}{\partial (x_1,x_2,v_1,v_2,v_3)}\right|= \left| \begin{array}{cc} \frac{\partial ( \tilde{x}_{1} , \tilde{x}_{2})}{\partial (x_{1}, x_{2})} &  \frac{\partial (\tilde{x}_{1} , \tilde{x}_{2})}{\partial (v_{1},v_{2},v_{3})} \\ 
% \frac{\partial (v_{1}, v_{2},v_{3})}{\partial (x_{1},x_{2})}  
%& \frac{\partial (v_{1},v_{2},v_{3})}{\partial (v_{1},v_{2},v_{3})}\end{array}  \right|
%= \left| \begin{array}{cc} \frac{\partial ( \tilde{x}_{1} , \tilde{x}_{2})}{\partial (x_{1}, x_{2})} &  \frac{\partial (\tilde{x}_{1} , \tilde{x}_{2})}{\partial (v_{1},v_{2},v_{3})} \\ 
%\mathrm{0}_{3,2}
%&\mathrm{Id}_{3,3}\end{array}  \right|
%\\
% &=&  \left|\begin{array}{ccc}
%\partial_{x_1} \tilde{x}_1 & \partial_{x_2}\tilde{x}_1\\
%\partial_{x_1} \tilde{x}_2 & \partial_{x_2} \tilde{x}_2
%\end{array}\right|,
%\end{eqnarray*}
%and

Denote $D=\{(x_1-\tilde{x}_1)^2 + (x_2-\tilde{x}_2)^2 + [\eta(x_1,x_2)-\varphi(\tilde{x}_1,\tilde{x}_2)]^2\}.$ Directly from (\ref{xtildex})
\begin{equation}\nonumber
\begin{split}
& \left[\begin{array}{ccc}
\Big[(x_1-\tilde{x}_1) + (\eta-\varphi) \partial_{\tilde{x}_1}\varphi\Big]D^{-\frac{1}{2}}v_1 -|v| & \Big[ (x_2-\tilde{x}_2) + (\eta-\varphi) \partial_{\tilde{x}_2}\varphi\Big] D^{-\frac{1}{2}} v_1 \\
\Big[(x_1-\tilde{x}_1)+(\eta-\varphi)\partial_{\tilde{x}_1}\varphi\Big]D^{-\frac{1}{2}}v_2
&
\Big[(x_2-\tilde{x}_2) + (\eta-\varphi) \partial_{\tilde{x}_2} \varphi\Big] D^{-\frac{1}{2}}v_2 -|v|
\end{array}\right]\\
& \times 
\left[\begin{array}{ccc}
\partial_{x_1}\tilde{x}_1 & \partial_{x_2} \tilde{x}_1\\
\partial_{x_1}\tilde{x}_2 & \partial_{x_2} \tilde{x}_2
\end{array}\right]\\
 =& \left[\begin{array}{ccc}
v_1 D^{-\frac{1}{2}} \big((x_1-\tilde{x}_1)+ (\eta-\varphi) \partial_{x_1}\eta\big)-|v| & v_1 D^{-\frac{1}{2}} \big( (x_2-\tilde{x}_2) + (\eta-\varphi) \partial_{x_2}\eta\big)\\
v_2 D^{-\frac{1}{2}} \big( (x_1-\tilde{x}_1)+ (\eta-\varphi)\partial_{x_1}\eta\big)&
v_2 D^{-\frac{1}{2}}\big((x_2-\tilde{x}_2)+ \partial_{x_2}\eta (\eta-\varphi)\big)-|v|
\end{array}\right].
\end{split}
\end{equation}
%and
%\begin{equation}\nonumber
%\begin{split}
%&\left[\begin{array}{ccc}
%\partial_{x_1}\tilde{x}_1 & \partial_{x_2} \tilde{x}_1\\
%\partial_{x_1}\tilde{x}_2 & \partial_{x_2} \tilde{x}_2
%\end{array}\right]= \left\{|v|^2 -|v|D^{-\frac{1}{2}} \Big[ v_2 (x_2-\tilde{x}_2) + v_2 (\eta-\varphi) \partial_{\tilde{x}_2}\varphi + v_1 (x_1-\tilde{x}_1)+ v_1 (\eta-\varphi) \partial_{\tilde{x}_1}\varphi\Big]\right\}^{-1}\\
%&\times
%\left[\begin{array}{ccc}
%\Big[(x_2-\tilde{x}_2) + (\eta-\varphi) \partial_{\tilde{x}_2} \varphi\Big] D^{-\frac{1}{2}}v_2 -|v| & - \Big[ (x_2-\tilde{x}_2) + (\eta-\varphi) \partial_{\tilde{x}_2}\varphi\Big] D^{-\frac{1}{2}} v_1\\
%-\Big[(x_1-\tilde{x}_1)+(\eta-\varphi)\partial_{\tilde{x}_1}\varphi\Big]D^{-\frac{1}{2}}v_2
%& \Big[(x_1-\tilde{x}_1) + (\eta-\varphi) \partial_{\tilde{x}_1}\varphi\Big]D^{-\frac{1}{2}}v_1 -|v|
%\end{array}\right]\\
%&\times \left[\begin{array}{ccc}
%v_1 D^{-\frac{1}{2}} \big((x_1-\tilde{x}_1)+ (\eta-\varphi) \partial_{x_1}\eta\big)-|v| & v_1 D^{-\frac{1}{2}} \big( (x_2-\tilde{x}_2) + (\eta-\varphi) \partial_{x_2}\eta\big)\\
%v_2 D^{-\frac{1}{2}} \big( (x_1-\tilde{x}_1)+ (\eta-\varphi)\partial_{x_1}\eta\big)&
%v_2 D^{-\frac{1}{2}}\big((x_2-\tilde{x}_2)+ \partial_{x_2}\eta (\eta-\varphi)\big)-|v|
%\end{array}\right].
%\end{split}
%\end{equation}

%Since we only need to know the determinant of $\left[\begin{array}{ccc}
%\partial_{x_1}\tilde{x}_1 & \partial_{x_2} \tilde{x}_1\\
%\partial_{x_1}\tilde{x}_2 & \partial_{x_2} \tilde{x}_2
%\end{array}\right]$ we compute determinants for 
Direct computations yield
\begin{equation}\notag
\begin{split}
J &= \frac{|v|  - D^{-\frac{1}{2}}\Big[
v_{1} (x_1-\tilde{x}_1) + v_1 (\eta-\varphi) \partial_{x_1} \eta + v_2 (x_2-\tilde{x}_2) + v_2(\eta-\varphi) \partial_{x_2}\eta
\Big]}{|v| - D^{-\frac{1}{2}}\Big[
v_2 (x_2-\tilde{x}_2) + v_2 (\eta-\varphi) \partial_{\tilde{x}_2} \varphi + v_1 (x_1-\tilde{x}_1) + v_1 (\eta-\varphi) \partial_{\tilde{x}_1}\varphi
\Big]}\\
&= \frac{|v|^2- \big[(v_1)^2 + (v_2)^2 + (v_3) (v_1 \partial_{x_1}\eta + v_2 \partial_{x_2}\eta)\big]}{|v|^2-\big[ (v_1)^2 + (v_2)^2 + (v_3) (v_1 \partial_{\tilde{x}_1} \varphi + v_2 \partial_{\tilde{x}_2}\varphi)\big]} = \frac{(\partial_{x_1}\eta, \partial_{x_2}\eta, -1)\cdot v}{(\partial_{\tilde{x}_1}\varphi, \partial_{\tilde{x}_{2}}\varphi, -1)\cdot v }\\
&= \frac{\sqrt{1+|\nabla \eta|^2}}{\sqrt{1+|\nabla \varphi|^2}} \times \frac{n(x)\cdot v}{n(\tilde{x})\cdot v}.
\end{split}
\end{equation}
Then we use (\ref{COV_lemma}) to conclude the proof.

Second{ly} we consider the case of $n_{1} (\xb(x,v)) \neq 0$ or $n_{2} (\xb(x,v)) \neq 0$. Without loss of generality we may assume $n_{2} (\xb(x,v)) \neq 0$ so that 
$$\tilde{x}= x_{\mathbf{b}}(x,v) = (\tilde{x}_{1}, \tilde{x}_{2}, \tilde{x}_{3}) = (\tilde{x}_{1}, \varphi( \tilde{x}_{1}, \tilde{x}_{3}),  \tilde{x}_{3}),$$ for some function $\varphi: \mathbb{R}^2 \rightarrow \mathbb{R}$. {Notice that \eqref{COV_lemma} still holds with $\tilde{x}_2$ replaced by $\tilde{x}_3$.} From the fact $v|x-\tilde{x}| = |v| (x-\tilde{x})$ we have 
\begin{equation}\label{xtildex_2}
\{(x_1-\tilde{x}_1)^2 + (x_2- \varphi(\tilde{x}_{1}, \tilde{x}_{3}))^2 + [\eta(x_1,x_2)- \tilde{x}_{3} ]^2\}^\frac{1}{2}\left(\begin{array}{ccc}
v_1 \\ v_2 \\ v_3\end{array}\right) = |v| \left(\begin{array}{ccc}x_1 -\tilde{x}_{1} \\
x_2 - \varphi(\tilde{x}_{1}, \tilde{x}_{3}) \\ \eta(x_1,x_2)- \tilde{x}_{3}.
\end{array}\right).
\end{equation}

We define $\tilde{D}=\{(x_1-\tilde{x}_1)^2 + (x_2- \varphi(\tilde{x}_{1}, \tilde{x}_{3}))^2 + [\eta(x_1,x_2)- \tilde{x}_{3} ]^2\}.$ 

By direct computation
\begin{equation}\notag
\begin{split}
& \left[ \begin{array}{cc}   \Big[ (x_{1} - \tilde{x}_{1}) + (x_{2} - \varphi) \partial_{\tilde{x}_{1}} \varphi  \Big] v_{1}  \tilde{D}^{-\frac{1}{2}} - |v| 
& \Big[  (x_{2} - \varphi) \partial_{\tilde{x}_{3}} \varphi + (\eta - \tilde{x}_{3}) \Big] v_{1} \tilde{D}^{-\frac{1}{2}} \\ 
\Big[ (x_{1} - \tilde{x}_{1}) + (x_{2} -\varphi) \partial_{\tilde{x}_{1}} \varphi \Big] v_{3} \tilde{D}^{-\frac{1}{2}} &
 \Big[ (x_{2} - \varphi) \partial_{\tilde{x}_{3}}  \varphi + (\eta - \tilde{x}_{3})  \Big]v_{3} \tilde{D}^{-\frac{1}{2}} - |v|
\end{array}\right]\\ & \times  
\left[\begin{array}{cc} \partial_{x_{1}} \tilde{x}_{1} & \partial_{x_{2}} \tilde{x}_{1} \\ \partial_{x_{1}} \tilde{x}_{3} & \partial_{x_{2}} \tilde{x}_{3} \end{array} \right]\\
 =& \left[\begin{array}{cc}
\Big[ (x_{1} - \tilde{x}_{1}) + (\eta - \tilde{x}_{3}) \partial_{x_{1}} \eta  \Big]v_{1} \tilde{D}^{-1/2}- |v| & \Big[ (x_{2} - \varphi) + (\eta - \tilde{x}_{3}) \partial_{x_{2}} \eta \Big] v_{1} \tilde{D}^{-1/2}\\
\Big[ (x_{1} - \tilde{x}_{1}) + (\eta - \tilde{x}_{3}) \partial_{x_{1}} \eta  \Big] v_{3} \tilde{D}^{-1/2} - |v| \partial_{x_{1}} \eta &  \Big[ (x_{2} - \varphi) + (\eta -\tilde{x}_{3}) \partial_{x_{2}} \eta  \Big] v_{3} \tilde{D}^{-1/2} - |v| \partial_{x_{2}} \eta 
 \end{array}\right],
\end{split}
\end{equation}
and 
\begin{equation}\notag
\begin{split}
&\det   \left[\begin{array}{cc} \partial_{x_{1}} \tilde{x}_{1} & \partial_{x_{2}} \tilde{x}_{1} \\ \partial_{x_{1}} \tilde{x}_{3} & \partial_{x_{2}} \tilde{x}_{3} \end{array} \right] \\
&= 
\frac{
|v|_{2} \partial_{x_{2}} \eta - \Big[ (x_{1} - \tilde{x}_{1}) + (\eta -\tilde{x}_{3}) \partial_{x_{1}} \eta  \Big] v_{1} |v| \tilde{D}^{-\frac{1}{2}} \partial_{x_{2}} \eta
+ \Big[ (x_{2} - \varphi) + (\eta - \tilde{x}_{3}) \partial_{x_{2}} \eta \Big] \tilde{D}^{-\frac{1}{2}}|v| (v_{1 } \partial_{x_{1}} \eta - v_{3})
}{|v|_{2} -  \Big[ (x_{1} - \tilde{x}_{1}) + (x_{2} - \varphi) \partial_{\tilde{x}_{1}} \varphi  \Big] |v|v_{1} \tilde{D}^{-\frac{1}{2}}  -  \Big[  (x_{2} - \varphi) \partial_{\tilde{x}_{3}} \varphi + (\eta - \tilde{x}_{3}) \Big] |v| v_{3} \tilde{D}^{-\frac{1}{2}} }
 \\ 
 &= \frac{ - v_{1} \partial_{\tilde{x}_{1}} \varphi + v_{2} - v_{3} \partial_{\tilde{x}_{3}} \varphi  }{ v_{1} \partial_{x_{1}} \eta + v_{2} \partial_{x_{2}} \eta - v_{3} } = \frac{  -\big(  \partial_{\tilde{x}_{1}}\varphi, -1,  \partial_{\tilde{x}_{3}} \varphi  \big)\cdot v}{\big( \partial_{x_{1}} \eta, \partial_{x_{2}} \eta, -1 \big)\cdot v}\\
 &= -\frac{\sqrt{1+ |\nabla \eta|_{2}}}{\sqrt{1+ |\nabla \varphi|^{2}}} \times \frac{n(\tilde{x})\cdot v}{n(x) \cdot v}.
\end{split}
\end{equation}
Then we use (\ref{COV_lemma}) {(with $\tilde{x}_2$ replaced by $\tilde{x}_3$)} to conclude the proof.
\end{proof}

\begin{proof}[\textbf{Proof of Proposition \ref{double_iteration}}]

It suffices to prove the estimate (\ref{L1trace}). 

Using (\ref{bdry_h}), we obtain
\begin{eqnarray}
\int_0^t |h^{m+1} (s)|_{\gamma_-,1} &:=& \int_0^t
\iint_{n(x)\cdot v<0} | h^{m+1}(s,x,v)| |n(x)\cdot v| \mathrm{d}S_x \mathrm{d}v
\mathrm{d}s\nonumber\\ &\lesssim& \underline{\int_0^t \iint_{n(x)\cdot  {v}>0} |h^{m}(s,x, {v})|
\mu( {v})^{\frac{1}{4}} |n(x)\cdot  {v}| \mathrm{d}S_x \mathrm{d}v
\mathrm{d}s}\label{split:incoming_bdry_1}  \\
&&+
  \int_{0}^{t} \iint_{n(x)\cdot v<0} |R^{m}(s,x,v)| \{ 1+ |n(x)\cdot v|\} \mathrm{d}S_{x} \mathrm{d}v \mathrm{d}s \nonumber
\end{eqnarray}
Clearly the last term is bounded by the RHS of (\ref{L1trace}).

Focus on the underlined term. Recall the almost grazing set $\gamma_{+}^{\delta}$ and the non-grazing set $\gamma_{+} \backslash \gamma_{+}^{\delta}$ in (\ref{n-grazing}) and (\ref{non-grazing}). We split the outgoing part as 
\begin{eqnarray*}\notag 
\gamma_{+} \ = \  \gamma_{+}^{\delta} \ \cup \ \gamma_{+} \backslash \gamma_{+}^{\delta}.\end{eqnarray*}
%\begin{eqnarray}
%\gamma_+ & =&
%\gamma_+^{\delta} \ \cup \ \gamma_+\backslash \gamma_+^\delta \nonumber\\ &:=&
%\underbrace{\big\{(x,v)\in\gamma_+ : n(x)\cdot v < \delta \ \text{or} \ |v|> 1/ \delta  \big\}}_{(\ref{split})-(i)} \
%\cup \ \underbrace{\big\{(x,v)\in\gamma_+ : n(x)\cdot v \geq \delta \ \text{and} \ |v|\leq
%1/ \delta\big\}}_{(\ref{split})-(ii)}.\notag \ \ \ \\ \label{split}
%\end{eqnarray}
Due to Lemma \ref{le:ukai}, the non-grazing {part $\gamma_{+}\backslash \gamma_{+}^{\delta}$ of the integral is bounded as
\begin{equation}\label{9_ii} 
\begin{split}
\int_0^t\iint_{\gamma_{+}\backslash \gamma_{+}^{\delta}}  & \  
 \lesssim_{t, \delta,\Omega} \ || h_0||_1 + \int_0^t \big\{ \big|\big|  
h^{m} (s)\big|\big|_1 + \big|\big|[\partial_t + v\cdot \nabla_x + \nu]  
h^{m}(s) \big|\big|_1\big\} \mathrm{d}s\\
& \ \lesssim_{t, \delta,\Omega} \ || h_0||_1 + \int_0^t  ||  
h^{m} ||_1 + \int_0^t || H^{m-1}  ||_1   .
\end{split}
\end{equation}}
For the almost grazing set $\gamma_{+}^{\delta}$, we claim that the following truncated term with a number $k \in\mathbb{N}$ is uniformly bounded in $k$ as
\begin{equation}\label{uniform_k}
\begin{split}
&\int^t_0 \iint_{ \substack{x\in\partial\Omega, \\ n(x)\cdot v>0}  } \mathbf{1}_{\{(x,v)\in\gamma_+^ \delta\}} \mathbf{1}_{\{1/k < |n(x_{\mathbf{b}}(x,v))\cdot v|\}} |h^{m}(s,x,v)| \mu(v)^\frac{1}{4}   \{n(x)\cdot v\}\mathrm{d}v\mathrm{d}S_x  \mathrm{d}s\\
\leq& \ O (\delta) \int_0^t  | h^{m-1}(s)|_{\gamma_+,1}  
 +  C_{ \delta} \Big\{    || h_{0}||_1 + \int_0^t ||h^{m-1}(s)||_1 +\int_0^t ||   H^{m-1}||_1       + t |   R^{m-1}|_{1}\Big\} .
 \end{split}
\end{equation}

%Note that $\mathbf{1}_{\{1/k < |n(x_{\mathbf{b}}(x,v))\cdot v|\}} |h^{m}(s,x,v)|\rightarrow |h^{m}(s,x,v)|$ a.e. as $k\rightarrow +\infty$. Once we prove (\ref{uniform_k}) then we use the monotone convergence theorem (or Lebesgue convergence theorem with the condition $h^{m} \in L_{loc}^{1} ([0,T]; L^{1}(\gamma_{+}, \mathrm{d}\gamma))$) to deduce the same bound of (\ref{uniform_k}) without the truncation $\mathbf{1}_{\{1/k < |n(x_{\mathbf{b}}(x,v))\cdot v|\}}$. Combing (\ref{uniform_k}) with (\ref{9_ii}) we conclude (\ref{L1trace}). 

In order to show (\ref{uniform_k}) we use the Duhamel formula of the equation (\ref{eqtn_h}) together with (\ref{bdry_h}): for $(x,v)\in \gamma_+^ \delta$ and $\frac{1}{k}<|n(x_{\mathbf{b}}(x,v))\cdot v|$ 
\begin{equation*}
\begin{split}
&| h^{m}(s,x,v)| \mathbf{1}_{\{(x,v)\in \gamma_+^ \delta\}} \mathbf{1}_{\{ 1/k < |n(x_{\mathbf{b}}(x,v))\cdot v |\}}\\
 \leq & \ \ \ \mathbf{1}_{\{s < t_{\mathbf{b}}(x,v)\}} | h_0(x-sv,v)| + \int_{\max\{0,s-t_{\mathbf{b}}(x,v) \}}^s | H^{m-1}(\tau,x-(s-\tau)v,v)|
\mathrm{d}\tau\\
& + \mathbf{1}_{\{s> t_{\mathbf{b}}(x,v)\}} \mathbf{1}_{\{ 1/k < |n(x_{\mathbf{b}}(x,v))\cdot v |\}} C_{1}\sqrt{\mu(v)} \left(1+ \frac{\langle
v\rangle }{|n(x_{\mathbf{b}}(x,v))\cdot v|}\right)\\
& \ \ \ \ \ \ \ \ \ \ \   \ \  \ \ \ \ \  \times  \int_{n(x_{\mathbf{b}}(x,v))\cdot v_1>0} | h^{m-1} (s-t_{\mathbf{b}}(x,v),x_{\mathbf{b}}(x,v),v_1)| \mu(v_1)^{\frac{1}{4}} \{n(x_{\mathbf{b}}(x,v) )\cdot v_1\}
\mathrm{d}v_1\\
& +\mathbf{1}_{\{s> t_{\mathbf{b}}(x,v)\}} \mathbf{1}_{\{ 1/k < |n(x_{\mathbf{b}}(x,v))\cdot v |\}}  \Big(1 + \frac{ e^{-C_{2} |v|^{2}} }{ |n(x_{\mathbf{b}}(x,v))\cdot v| }\Big)|R^{m-1}(s-t_{\mathbf{b}}(x,v) , x_{\mathbf{b}}(x,v),v) |.
\end{split}
\end{equation*}
We plug this estimate into the left hand side of (\ref{uniform_k}) to have
 \begin{eqnarray}
&& \int^t_0 \iint_{ \substack{x\in\partial\Omega, \\ n(x)\cdot v>0}  } \mathbf{1}_{\{(x,v)\in\gamma_+^\varepsilon\}} \mathbf{1}_{\{1/k < |n(x_{\mathbf{b}}(x,v))\cdot v|\}} |h^{m}(s,x,v)| \mu(v)^\frac{1}{4}   \{n(x)\cdot v\}\mathrm{d}v\mathrm{d}S_x  \mathrm{d}s\nonumber\\
&\leq&\int_0^t
\iint_{\gamma_+^ \delta}\mathbf{1}_{\{ 1/k < |n(x_{\mathbf{b}}(x,v))\cdot v |\}} | h_0(x-sv,v)| \mu(v)^{\frac{1}{4}}|n(x)\cdot
v| \mathrm{d}S_x \mathrm{d}v \mathrm{d}s\label{initial}\\
&+& \int_0^t
\iint_{\gamma_+^ \delta} \mathbf{1}_{\{ 1/k < |n(x_{\mathbf{b}}(x,v))\cdot v |\}}\mu(v)^{\frac{1}{4}} |n(x)\cdot v|\notag\\
&&   \ \ \ \ \ \  \ \ \ \ \ \ \ \ \ \ \   \times   \int_{\max\{0,s-t_{\mathbf{b}}(x,v)\}}^s   | H^{m-1}(\tau,x-(s-\tau)v,v) |
\mathrm{d}\tau\mathrm{d}S_x \mathrm{d}v \mathrm{d}s    \label{time_int}  \\
&+& \int_0^t \iint_{\gamma_+^\delta}
\mathbf{1}_{\{ 1/k < |n(x_{\mathbf{b}}(x,v))\cdot v |\}}\mu(v)^\frac{1}{2} \frac{ |n(x)\cdot v| }{|n(x_{\mathbf{b}}(x,v))\cdot v|} \int_{n(x_{\mathbf{b}}(x,v))\cdot
v_1>0}\1_{\{s>t_{\mathbf{b}}(x,v)\}} \notag
\\
&&   \ \ \ \ \ \  \ \ \ \ \ \ \ \ \ \ \   \times 
 | h^{m-1}(s-t_{\mathbf{b}}(x,v),x_{\mathbf{b}}(x,v),v_1)|
\mu(v_1)^{\frac{1}{4}}\{n(x_{\mathbf{b}}(x,v))\cdot v_1\}\mathrm{d}v_1\mathrm{d}S_x \mathrm{d}v \mathrm{d}s\notag \\ \label{main} \\
&+&
\int_0^t \iint_{\gamma_+^ \delta} \mathbf{1}_{\{ s> \tb(x,v) \}} \mathbf{1}_{\{ 1/k < |n(x_{\mathbf{b}}(x,v))\cdot v |\}} \mu(v)^{\frac{1}{4}}\frac{|n(x)\cdot
v|}{|n(x_{\mathbf{b}}(x,v))\cdot v|} \notag\\
&& \ \ \ \ \ \  \ \ \ \ \ \ \ \ \ \ \   \times  |R^{m-1}(s-t_{\mathbf{b}}(x,v), x_{\mathbf{b}}(x,v),v)| \mathrm{d}S_x \mathrm{d}v \mathrm{d}s.  \label{last}
\end{eqnarray}
 
 \bigskip

 \textit{Estimate of (\ref{initial})}: Note that $x\in\partial\Omega$ in (\ref{initial}). Without loss of generality we may assume that there exists $\eta: \mathbb{R}^2 \rightarrow \mathbb{R}$ {such that $x^{3} =  \eta(x^{1},x^{2})$}.
We apply the following change of variables: for fixed $v \in\mathbb{R}^3,$
\[
(x^1,x^2;s) \in \mathbb{R}^2 \times \{ 0 \leq s \leq \tb(x,v)\} \ \mapsto \ y=(x^1-sv^1,x^2-sv^2,\eta(x^1,x^2)-sv^3)  \in\bar\Omega.
\]
We compute the Jacobian:
\begin{equation}\nonumber
\begin{split}
\det\left( \frac{\partial (y^{1}, y^{2},y^{3})}{\partial( x^{1}, x^{2}, s)} \right)&= \det\left(\begin{array}{ccc}
1 & 0 & -v^{1} \\
0 & 1 & -v^{2} \\
\partial_{x^{1}}\eta(x^{1},x^{2}) & \partial_{x^{2}} \eta(x^{1},x^{2}) & -v^{3}
\end{array}\right) \\
 &=  v\cdot \left(\begin{array}{ccc} \partial_{x^{1}}\eta  \\ \partial_{x^{2}} \eta  \\ -1\end{array}\right) = v\cdot n\sqrt{1+|\partial_{x^{1}}\eta|^2 + |\partial_{x^{2}} \eta|^2}.
\end{split}
\end{equation}
Therefore
\[
\{v\cdot n(x)\} \mathrm{d}S_x \mathrm{d}s= \{v\cdot n(x)\} \sqrt{1+|\partial_{x^{1}}\eta|^2 + |\partial_{x^{2}} \eta|^2} \mathrm{d}x^{1} \mathrm{d}x^{2} \mathrm{d}s =
\mathrm{d}y = \ud y^{1} \ud y^{2} \ud y^{3} ,
\]
and
\begin{equation}\label{est:initial}
\begin{split}
(\ref{initial})& \leq \int_{\mathbb{R}^3} \mathrm{d}v \int_0^t \mathrm{d}s \int_{\partial\Omega} \mathrm{d}S_x \mathbf{1}_{\{(x,v) \in \gamma_+\}} \mathbf{1}_{\{1/k < |n(\xb(x,v))\cdot v|\}} |h_0(x-sv,v)| \mu(v)^\frac{1}{4} |n(x)\cdot v|\\
& \leq \int_{\mathbb{R}^3} \mathrm{d}v  \int_{\Omega} \mathrm{d}y
|h_0(y,v)| \mu(v)^{\frac{1}{4}} \\
& \leq || h_0||_1.
\end{split}
\end{equation}

\bigskip

\textit{Estimate of (\ref{time_int}):} Considering the region of $\big\{  (\tau, s) \in [0,t]\times [0,t] : \max \{ 0, s-\tb(x,v)  \} \leq \tau \leq s  \big\}$, 
\begin{equation}\label{term1}
(\ref{time_int}) \leq \int_{\mathbb{R}^3} \mathrm{d}v \int^t_{0} \mathrm{d}\tau \int_\tau^{\min\{t,\tau+\tb(x,v)\}} \mathrm{d}s \int_{\partial\Omega} \mathrm{d} S_x   |  H^{m-1} (\tau,x-(s-\tau)v,v) | \mu(v)^\frac{1}{4}|n(x)\cdot v|.
\end{equation}

Note that $x \in \partial\Omega.$ Without loss of generality we may assume that $x^{3} = \eta(x^{1},x^{2})$ for $\eta : \mathbb{R}^{2} \rightarrow \mathbb{R}$. We apply the change of variables: for fixed $v \in\mathbb{R}^3$ and $\tau \in [0,t],$
\[
(   x^{1}, x^{2}; s) \in   \mathbb{R}^2 \times [\tau, \min\{t,\tau+\tb(x,v)\}]  \ \mapsto \ y\equiv(x^{1} -(s-\tau)v^{1}, x^{2} -(s-\tau)v^{2}, \eta(x^{1},x^{2})-(s-\tau)v^{3}).
\]
The Jacobian is $\{v\cdot n(x)\} \sqrt{1+ |\partial_{x^{1}}\eta|^2 + |\partial_{x^{2}} \eta|^2}$ and $\{v\cdot n(x)\} \mathrm{d}s\mathrm{d}S_x \ \leq \ \mathrm{d}y.$ Applying the change of variables to (\ref{term1}) to have  
\begin{equation}
(\ref{time_int}) \leq\int^t_0  \int_{\mathbb{R}^3} \int_{\Omega}  |  H^{m-1}(\tau,y,v) |  \mu(v)^\frac{1}{4} \mathrm{d}y\mathrm{d}v \mathrm{d}\tau.\label{est:time_int}
\end{equation}

\bigskip

\textit{Estimate of (\ref{main}):} This part is the most delicate among (\ref{initial})$\sim$(\ref{last}). Rewrite (\ref{main}) as
\begin{equation}\label{main_detail}
\begin{split}
&\int_0^t
\mathrm{d}s \int_{\partial\Omega} \mathrm{d}S_x \int_{\mathbb{R}^3} \mathrm{d}v \int_{\mathbb{R}^3} \mathrm{d}v_1 \
\mathbf{1}_{\{(x,v) \in \gamma_+^\varepsilon\}} \mathbf{1}_{\{n(x_{\mathbf{b}}(x,v)  )\cdot v_1>0\}} \mathbf{1}_{\{s>t_{\mathbf{b}}(x,v)\}} \mathbf{1}_{\{   |n(x_{\mathbf{b}}(x,v))\cdot v| > 1/k\}} \\ & \ \ \ \ \ \ \  \    \ \  \times \mu(v_1)^{\frac{1}{4}} \mu(v)^\frac{1}{2} \frac{|n(x)\cdot
v|}{|n(x_{\mathbf{b}}(x,v) )\cdot v|}
|n(x_{\mathbf{b}}(x,v))\cdot v_1| |h^{m-1} (s-t_{\mathbf{b}}(x,v),x_{\mathbf{b}}(x,v),v_1)|. 
\end{split}
\end{equation}

First we apply the following change of variables
\begin{equation}\label{COV_time}
s\in [0,t] \ \mapsto \ \tilde{s} = s-t_{\mathbf{b}}(x,v)\in [0,t],
\end{equation}
where we have used the fact that $s$ is integrated over $[t_{\mathbf{b}}(x,v), t]$. Clearly the Jacobian is $1$ so that $\mathrm{d}\tilde{s}=\mathrm{d}s$ and hence
\begin{equation}\label{main1}
\begin{split}
(\ref{main_detail})&\leq  \int_0^t
\mathrm{d}\tilde{s} \int_{\partial\Omega} \mathrm{d}S_x \int_{\mathbb{R}^3} \mathrm{d}v \int_{\mathbb{R}^3} \mathrm{d}v_1 \
\underbrace{\mathbf{1}_{\{(x,v) \in \gamma_+^ \delta\}}} \ \mathbf{1}_{\{n(x_{\mathbf{b}}(x,v))\cdot v_1>0\}}   \mathbf{1}_{\{   |n(x_{\mathbf{b}})\cdot v| > 1/k\}}
  \\ & \ \ \ \ \ \ \ \ \ \ \ \ \ \ \ \ \ \ \ \ \ \ \ \ \ \ \times \mu(v_1)^{\frac{1}{4}} \mu(v)^\frac{1}{2} \frac{|n(x)\cdot
v|}{|n(x_{\mathbf{b}}(x,v))\cdot v|}
|n(x_{\mathbf{b}}(x,v))\cdot v_1| | h^{m-1}(\tilde{s},x_{\mathbf{b}}(x,v),v_1)|.
\end{split}
\end{equation}
Let us denote
\begin{equation}\label{x1}
\tilde{x}  : = \xb(x,v).
\end{equation} 
Note that since $(x,v)\in\gamma_+$ and $|n(x_{\mathbf{b}}(x,v))\cdot v|  > 1/k$, from Lemma \ref{COV}, the mapping $(x,v)\mapsto (\tilde{x} ,v)$ is one-to-one and
\begin{equation}\notag
\begin{split}
t_{\mathbf{b}}(x,v) \ =& \ t_{\mathbf{b}}(x_{\mathbf{b}}(x,v),-v),\\
x \ =& \ \xb(x,v)+ \tb(x,v) v = \xb(x,v) + \tb( \xb(x,v) ,-v)v = \xb(x,v)-\tb( \xb(x,v),-v)(-v)\\
 =& \ \tilde{x}  -\tb(\tilde{x},-v)(-v),
\end{split}
\end{equation}
and hence we can rewrite the underbraced term in (\ref{main1}) as
\begin{equation}
 {\1_{\{(x,v)\in\gamma_+^ \delta\}}} \ = \ \1_{\{0<n(\tilde{x}-\tb(\tilde{x},-v)(-v))\cdot v< \delta   \ \text{or} \  |v|> 1/\delta\}}.\label{nx1}
\end{equation}

Now we apply the change of variables of Lemma \ref{COV}: for $(x,v)\in\gamma_+$ and $|n(\xb(x,v))\cdot v|= |n(\tilde{x})\cdot v|> 1/k$, we apply the change of variables 
\begin{equation} \label{COV_space}
(x,v)\mapsto
(\tilde{x},v) := (\xb(x,v),v).
\end{equation}
From Lemma \ref{COV}, the Jacobian is
$$\det\left(\frac{\partial(\tilde{x},v)}{\p(x,v)}\right)=\det\left(\frac{\partial
\tilde{x}}{\partial x}\right) = \left|\frac{n(x)\cdot
v}{n( \tilde{x})\cdot v}\right|, \ \ \ \text{and} \ \ \  \mathrm{d}S_{\tilde{x}}:= \left|\frac{n(x)\cdot v}{n(\tilde{x})\cdot v}\right|
\mathrm{d}S_x.$$
 Then from (\ref{main1}) and (\ref{nx1}), 
\begin{equation}\label{main2}
\begin{split}
{(\ref{main_detail}) \leq}& \int^t_0 \ud \tilde{s} \int_{\R^3} \ud v_1 \int_{\R^3} \ud v \int_{\p\Omega} \ud S_{ \tilde{x}}
\big\{\1_{0 < n(\tilde{x} -\tb(\tilde{x} ,-v)(-v))\cdot v < \delta}+\1_{|v|>1/\e}\big\}\\
&  \ \ \  \ \ \ \ \ \ \ \ \ \ \ \ \  \times \1_{\{n( \tilde{x} )\cdot v_1 >0\}} \mathbf{1}_{\{|n( \tilde{x}  )\cdot v|> 1/k\}} \mu(v)^{\frac{1}{2}} \mu(v_1)^{\frac{1}{4}} |n( \tilde{x} )\cdot v_1| \  | h^{m-1} (\tilde{s},\tilde{x} ,v_1)|\\
\leq& \int_0^t\iint_{\gamma_+ } | h^{m-1} (\tilde{s}, \tilde{x},v_1)| \mu(v_1)^\frac{1}{4} |n(\tilde{x})\cdot v_1|\mathrm{d}S_{\tilde{x}} \mathrm{d}v_1 \ud \tilde{s}\\
  & \ \ \ \ \ \ \ \ \ \ \ \ \ \ \ \ \ \ \ \ \ \ \times \underbrace{\sup_{\tilde{x}\in\partial\Omega}\int_{\R^3}  \1_{\{- \delta< n(\tilde{x}-\tb(\tilde{x},-v)(-v))\cdot (-v)< 0\}} \mu(v)^\frac{1}{2}\ud v}\\
& + O (  \delta) \int^t_0 \iint_{\gamma_+} | h^{m-1} (\tilde{s},\tilde{x},v_1)| \mu(v_1)^\frac{1}{4}  |n(\tilde{x})\cdot v_{1}| \mathrm{d}S_{\tilde{x}} \mathrm{d} v_{1} \ud \tilde{s},
\end{split}
\end{equation}
where $O(\delta) \sim \int_{\mathbb{R}^{3} }  \mathbf{1}_{|v|> 1/  \delta} \mu(v)^{\frac{1}{2}} \mathrm{d}v.$

We claim that, for any small $0<  \delta^\prime\ll 1$, we can choose sufficiently small $0< \delta\ll 1$ such that
\begin{equation}\label{small}
\sup_{\tilde{x}\in\partial\Omega}\int_{\R^3}  \1_{\{- \delta< n(\tilde{x} - \tb(\tilde{x}, -v)(-v)  )\cdot (-v)< 0\}} \mu(v)^{\frac{1}{2}}\ud v \leq \delta^\prime.
\end{equation}
This is consequence of Lemma \ref{guo_covering}. For given $ \delta^\prime>0$, we choose a sufficiently large $N \gg \frac{1}{ \delta^\prime}$ and we take $\delta_{ \delta^\prime, N}>0$ as in Lemma \ref{guo_covering}. Then we choose a sufficiently small $ \delta=\delta( \delta^\prime, N)>0$ such that $ \delta \ll \delta_{ \delta^\prime,N }$ in Lemma \ref{guo_covering}. Due to Lemma \ref{guo_covering} and (\ref{guo_covering_inclusion}),
$$ \max_{i}\sup_{\tilde{x} \in B(x_i;r_i)}\mathrm{m}_3\{ v \in \mathbb{R}^3 : |v|\leq N, \ | n(\xb(\tilde{x},-v))\cdot (-v) | \leq  \delta\} \ \leq  \   \max_{i}\mathrm{m}_3(\mathcal{O}_{x_i})\leq  \delta^\prime.$$
Finally we conclude the claim (\ref{small}) by
\begin{eqnarray*}\nonumber
&& \int_{\mathbb{R}^3} \mathbf{1}_{\{- \delta < n(\xb(\tilde{x},-v))\cdot(-v)<0\}} \mu(v)^\frac{1}{2} \mathrm{d}v \\ 
&=& \int_{|v|\geq N} + \int_{|v|\leq N}\\
 &\leq& e^{-N^2/4} + \max_i \mathrm{m}_3 (\mathcal{O}_i)\\
 &\leq&  e^{-\frac{1}{4 ( \delta^{\prime})^{2}}} + \delta^\prime.
\end{eqnarray*}

Therefore, from (\ref{main1}), (\ref{main2}), (\ref{small}), we have, for $0 <  \delta,  \delta^\prime \ll 1,$
\begin{equation}\label{est:main}
\begin{split}
(\ref{main}) & \lesssim [O( \delta) + O( \delta^\prime)] \times \int^t_0 \int_{\gamma_+} |h^{m-1}(\tilde{s},\tilde{x},v_1)|\mu(v_1)^\frac{1}{4}
|n(\tilde{x})\cdot v_{1}| \ud S_{\tilde{x}} \ud v_{1}
  \ud \tilde{s}.
\end{split}
\end{equation}

\bigskip

\textit{Estimate of (\ref{last})}: We apply the change of variables (\ref{COV_time}) and then apply (\ref{COV_space}) and use Lemma \ref{COV} to bound 
\begin{equation}\label{est:last}
\begin{split}
(\ref{last})&\lesssim \int^t_0 \mathrm{d}s \int_{\partial\Omega} \int_{\mathbb{R}^3}  
\mu(v)^{\frac{1}{4}}  |R^{m-1}(\tilde{s}, \tilde{x}, v)|
\mathrm{d}S_{\tilde{x}}  \mathrm{d}v \mathrm{d}\tilde{s}.
\end{split}
\end{equation}

\bigskip

Finally from (\ref{est:initial}), (\ref{est:time_int}), (\ref{est:main}) and (\ref{est:last}), we prove our claim (\ref{uniform_k}).
%\begin{equation}\label{final;cutoof}
%\begin{split}
%&\int_0^t \int_{\gamma_+^\varepsilon} \mathbf{1}_{\{1/k < |n(\xb(x,v))\cdot v|\}} | h^{m}(s,x,v)| \mu(v)^\frac{1}{4}|n(x)\cdot v| \mathrm{d}S_x \mathrm{d}v \mathrm{d}s\\
%\lesssim& \ || h_0||_1 + \int^t_0 \iint_{\Omega\times\mathbb{R}^3} | h^{m-1}(s,x,v)| + \int_0^t \iint_{\Omega\times\mathbb{R}^3}|  H^{m-1}(s,x,v)| \mu(v)^{\frac{1}{4}}  \\ &+ \int^{t}_{0} \int_{\partial\Omega\times \mathbb{R}^{3}} |R^{m-1}(s,x,v)| \mu(v)^{\frac{1}{4}} \ud S_{x} \ud v \ud s +  O(\varepsilon) \int_0^t \int_{\gamma_+} |h^{m-1}(s,x,v)|\mu(v)^\frac{1}{4} \mathrm{d}\gamma \mathrm{d}s,\\
%\end{split}
%\end{equation}
 
\bigskip

The last step is to pass a limit $k\rightarrow \infty.$ Clearly the sequence is non-decreasing in $k$:
$$0 \leq \mathbf{1}_{\{\frac{1}{ k } < |n(\xb(x,v))\cdot v|\}} |h^{m}(s,x,v)| \leq \mathbf{1}_{\{\frac{1}{ k+1 } < |n(\xb(x,v))\cdot v|\}} |h^{m}(s,x,v)|.$$

We claim, as $k \rightarrow \infty,$
\[ \mathbf{1}_{\{\frac{1}{k } < |n(\xb(x,v))\cdot v|\}} \mu(v)^\frac{1}{4}| h^{m}(s,x,v)| \ \rightarrow \ \mu(v)^\frac{1}{4}
| h^{m}(s,x,v)|  \ \ \ \text{a.e.} \ (x,v)\in\gamma_+ \ \text{with} \ \mathrm{d}\gamma. \]
It suffices to show $\mathbf{1}_{\{\frac{1}{k } < |n(\xb(x,v))\cdot v|\}} \mu(v)^\frac{1}{4} \rightarrow \mu(v)^\frac{1}{4}$ a.e. on $\gamma_{+}$. For $\varepsilon>0$ and $N\gg\varepsilon^{-1}\gg 1$, choose $k \gg 1$ such that $\frac{1}{k}<
\delta_{\e,N}$ in Lemma \ref{guo_covering}. Then \begin{equation}\nonumber \begin{split}
&\Big[1-\1_{\{|n(\xb(x,v))\cdot v|>\frac{1}{k}\}}(x,v)\Big]\mu(v)^\frac{1}{4}\\
 &\leq \max_{1\leq i \leq
l_{\e,N ,\Omega}} \1_{B(x_i;r_i)}(x) \times \1_{\{|n(\xb(x,v))\cdot
v|\leq\frac{1}{k}\}} \mu(v)^\frac{1}{4} \\
& \leq \max_{1\leq i \leq
l_{\e,N, \Omega}} \1_{B(x_i;r_i)}(x) \times \1_{\{|n(\xb(x,v))\cdot
v|\leq\delta_{\e,N }\}} \mu(v)^\frac{1}{4} \\
& \leq \max_{1\leq i \leq
l_{\e,N, \Omega}} \1_{B(x_i;r_i)}(x) \times \Big\{\1_{\{|v| \leq N , v \in
\mathcal{O}_i\}}(v)  \mu(v)^\frac{1}{4} + \1_{|v|\geq N}(v) e^{-\frac{N^2}{16}}\mu(v)^\frac{1}{8} \Big\}, \end{split}
\end{equation}
and hence
\begin{equation}\notag
\begin{split}
&\int_{\gamma_+}
|1-\1_{\{|n(\xb(x,v))\cdot v| >\frac{1}{k}\}}(x,v)| \mu(v)^{\frac{1}{2}}\ud \gamma\\
&\leq \max_{1\leq i \leq l_{\e,N, \Omega}} \int_{\partial\Omega} \int_{n\cdot
v>0} \1_{\{|v|\leq N, v\in \mathcal{O}_i\}}\ud v\ud S_x + O(\frac{1}{N})\\
&\lesssim  \e + O(\frac{1}{N}) \lesssim \varepsilon, \end{split} \end{equation}
which concludes the claim.

Now we use the monotone convergence theorem to conclude
\begin{equation}\nonumber
\begin{split}
 \int_0^t \int_{\gamma_+^\varepsilon} \mathbf{1}_{\{1/k < |n(\xb(x,v))\cdot v|\}}|h^{m}(s,x,v)| \mu(v)^\frac{1}{4} \mathrm{d}\gamma \mathrm{d}s \rightarrow \int_0^t \int_{\gamma_+^\varepsilon} | h^{m}(s,x,v)| \mu(v)^\frac{1}{4} \mathrm{d}\gamma \mathrm{d}s,
\end{split}
\end{equation}
as $k \rightarrow \infty$
and therefore $\int_0^t \int_{\gamma_+^\varepsilon} | h^{m}(s,x,v)| \mu(v)^\frac{1}{4} \mathrm{d}\gamma \mathrm{d}s$ has the same upper bound of (\ref{uniform_k}). Together with (\ref{9_ii}) we conclude (\ref{L1trace}).
\end{proof}

\section{$\varepsilon-$Neighborhood of the Singular set}

 In this section, we construct an open covering of the singular set $\mathfrak{S}_{\mathrm{B}}$ (proof of Lemma \ref{open_cover}) and a smooth function that cuts off the open covering of $\mathfrak{S}_{\mathrm{B}}$ (Definition 1). Moreover, we prove their crucial properties in Lemma \ref{open_cover}, Lemma \ref{cut_off}, and Lemma \ref{small_boundary_lemma}. 
 
\begin{lemma}\label{open_cover}
For $0< \varepsilon\leq \varepsilon_{1} \ll1$ and $\theta>0,$ we construct an open set $\mathcal{O}_{\varepsilon, \varepsilon_{1}}   \subset   \bar{\Omega}\times \mathbb{R}^3,$
such that, 
\begin{equation}\label{SinO}
\mathfrak{S}_{\mathrm{B}} \  \subset \   {{\mathcal{O}}_{\varepsilon,\varepsilon_{1} }}.
\end{equation}
There exists ${C}_{*}={C}_{*}(\Omega)\gg1$ such that for any $0< \varepsilon \leq \varepsilon_{1}\ll1$  
\begin{equation}\label{barO_O}
\overline{{\mathcal{O}}_{  \varepsilon, \varepsilon_{1} }}  \ \subset \ \mathcal{O}_{\varepsilon, {C}_{*} \varepsilon_{1}}.
\end{equation}
Moreover {there exist} $C_1=C_1(\theta,\Omega, C_{*})>0, \ C_2=C_2(\Omega, C_{*})>0,$ such that 
\begin{equation}\label{small_bulk}
 \iint_{\Omega\times\mathbb{R}^3} \mathbf{1}_{\mathcal{O}_{\varepsilon, {C}_{*} \varepsilon}}(x,v)  
e^{-\theta |v|^{2}}
 \mathrm{d}v\mathrm{d}x < C_{1} \varepsilon,
 \end{equation}
 and 
 \begin{equation}\label{dist_ep}
 \mathrm{dist}\big(  \bar{\Omega} \times \mathbb{R}^3  \backslash \mathcal{O}_{\varepsilon,  {C}_{*}\varepsilon }, \mathfrak{S}_{\mathrm{B}}\big)> C_2   \varepsilon .
\end{equation}
%Moreover $\mathfrak{S}_{\mathrm{B}} \backslash \mathrm{cl}({\mathcal{O}}_{\varepsilon/2})$ is a 5 dimensional smooth manifold, and there exists an open set $\mathcal{U}_\varepsilon \subset \bar{\Omega}\times \mathbb{R}^3$ such that, for any $\delta>0$ we have $C_3=C_3(\delta,\Omega)>0, \ C_4=C_4(\Omega)>0$
%\begin{eqnarray}\nonumber
%&\iint_{\Omega\times\mathbb{R}^3}\mathbf{1}_{\mathcal{U}_\varepsilon}(x,v) \langle v\rangle^{-2-\delta}\mathrm{d}v\mathrm{d}x< C_3   \varepsilon, \ \ \ \mathfrak{S}_{\mathrm{B}} \backslash \mathrm{cl}(\mathcal{O}_{\varepsilon/2}) \subset \mathcal{U}_\varepsilon,&\\   &\mathrm{dist}\big(\mathrm{cl}({\Omega})\times \mathbb{R}^3 \backslash \mathcal{U}_\varepsilon, \mathfrak{S}_{\mathrm{B}} \backslash \mathrm{cl}({\mathcal{O}}_{\varepsilon/2})\big)> C_4   \varepsilon.&
%\end{eqnarray}
%Furthermore, there exist $c_\Omega>0$ and $C_\Omega>0$ such that
%\begin{equation}\label{UO}
%\begin{split}
% & \iint_{\Omega\times\mathbb{R}^3}\mathbf{1}_{\mathcal{U}_\varepsilon \cup \mathcal{O}_\varepsilon}(x,v)\langle v\rangle^{-2-\delta}\mathrm{d}v \mathrm{d}x   <   C_\Omega \varepsilon,   \ \ \mathfrak{S}_{\mathrm{B}}   \subset   \mathrm{cl}({\mathcal{U}_{\varepsilon/2} \cup \mathcal{O}_{\varepsilon/2}}) \ \subset \ \mathcal{U}_\varepsilon \cup \mathcal{O}_\varepsilon,  \\
%& \ \ \ \ \ \ \ \ \ \ \ \ \ \ \  \ \ \ \ \ \ \ \ \ \  \  \ \mathrm{dist}\big(\mathrm{cl}(\Omega)\times \mathbb{R}^3 \backslash [\mathcal{U}_\varepsilon \cup \mathcal{O}_\varepsilon], \mathfrak{S}_{\mathrm{B}}\big)    >   c_\Omega \varepsilon.
%\end{split}
%\end{equation}
\end{lemma}

\begin{proof}

\noindent \textit{Construction of $\mathcal{O}_{\varepsilon,\varepsilon_{1}}$:} 
Let us fix $\delta>0$ ($\delta$ will be chosen later in (\ref{choice_delta})). Since the boundary $\partial\Omega$ is locally a graph of smooth functions, there exists a finite number $M_{\Omega,\delta}$ of small open balls $\mathcal{U}_{1},\mathcal{U}_{2},..,\mathcal{U}_{M_{\Omega,\delta}}\subset \mathbb{R}^{3}$ with $\mathrm{diam}(\mathcal{U}_{m})<2\delta$ for all $m$, such that 
\begin{equation}\label{boundary_decompose}
\partial\Omega \ \subset \ \bigcup_{m=1}^{M_{\Omega,\delta}} [  \ \mathcal{U}_{m} \ \cap \ \partial \Omega \ ] \ \ \text{
with}  \ \ M_{\Omega, \delta}= O(\frac{1}{\delta^{2}}),
\end{equation}
and for every $m$, inside $\mathcal{U}_m$ the boundary $\mathcal{U}_{m} \ \cap \ \partial \Omega$ is exactly described by a smooth function $\eta_m$ defined on a (small) open set $\mathcal{A}_{m}\subset\R^2$.

For all $m$, without loss of generality (up to rotations and translations \textit{depending on} $m$, and up to reducing the size of the ball $\mathcal{U}_m$) we will always assume that
\begin{eqnarray}\label{Um}
\mathcal{U}_{m} \cap \partial\Omega &=&\big\{(x_{1},x_{2}, \eta_{m}  (x_{1},x_{2})) \in \mathcal{A}_{m} \times \mathbb{R}  \big\},\\
\mathcal{U}_{m} \cap  \Omega &=&\big\{
(x_{1},x_{2},x_{3}) \in \mathcal{A}_{m} \times \mathbb{R} : x_{3} > \eta_{m} (x_{1},x_{2})  \big\},\nonumber
\end{eqnarray}
and 
$$(0,0) \in \mathcal{A}_{m} \ \subset_{\text{open}} \  [-\delta,\delta] \times [-\delta,\delta],$$
$$\partial_{1} \eta_{m}(0,0) =0 = \partial_{2} \eta_{m}(0,0).$$
Therefore
\[
n(0,0, \eta (0,0)) = \frac{1}{\sqrt{1+ |\partial_{1} \eta_{m} (0,0)|^{2} + |\partial_{2} \eta_{m}(0,0)|^{2}  }} (\partial_{1} \eta_{m} (0,0), \partial_{2} \eta_{m} (0,0),-1) = (0,0,-1).
\] 

Recall that $\partial\Omega$ is locally $C^{2}$. Then we can choose $\delta>0$ small enough to satisfy {for all $m\in \{1,..,M_{\Omega,\delta}\}$}
\begin{equation}\label{choice_delta}
\begin{split}
&|\partial_{1} \eta_{m} (x_{1},x_{2})-\partial_{1} \eta_{m} (0,0) |+|\partial_{2} \eta_{m} (x_{1},x_{2}) -\partial_{2} \eta_{m} (0,0) | \\
=& |\partial_{1} \eta_{m} (x_{1},x_{2})  |+|\partial_{2} \eta_{m} (x_{1},x_{2}) | 
\leq \frac{1}{8} \ \ \ \ \text{for} \  (x,y)\in \mathcal{A}_{m},
\end{split}
\end{equation}
and
\begin{equation}\label{eta_C2}
| \partial_{1}^{2} \eta_{m}(x_{1},x_{2}) | + |\partial_{2}^{2} \eta_{m} (x_{1}, x_{2})| + |\partial_{1} \partial_{2} \eta_{m}(x_{1},x_{2})| \leq C_{\eta} \ \ \ \ \text{for} \ (x,y) \in \mathcal{A}_{m}.
\end{equation}

Now we define the lattice point on $\mathcal{A}_{m}$ as
\begin{equation}\label{lattice}
c_{m,i,j,\varepsilon} := (\varepsilon i, \varepsilon j) \ \ \ \text{for} \ 
-N_{\varepsilon} \leq i,j \leq N_{\varepsilon} = O(\frac{\delta}{\varepsilon}).
 \end{equation}
Then we define the $(i,j)$-\textit{rectangular} $\mathcal{R}_{m,i,j,\varepsilon,\varepsilon_{1}}$ which is centered at $c_{m,i,j,\varepsilon}$ and whose side is $2\varepsilon_{1}$:
\begin{equation}\label{Rij}
\mathcal{R}_{m,i,j,\varepsilon,\varepsilon_{1}} :=  \Big\{ (x_{1},x_{2})\in  \big( \varepsilon  i -\varepsilon_{1},    \varepsilon  i +\varepsilon_{1} \big)\times   \big( \varepsilon  i -\varepsilon_{1},    \varepsilon  i +\varepsilon_{1} \big) \Big\} \ \cap \ {\mathcal{A}_{m}}.
\end{equation}  
Note that if $\varepsilon_{1} \geq \varepsilon$ then $\{  \mathcal{R}_{m,i,j,\varepsilon,\varepsilon_{1}} \}$ is open covering of $\mathcal{A}_{m}$, i.e.
\begin{equation}
\mathcal{A}_{m}  \subset \bigcup_{ -N _{\varepsilon}\leq i,j \leq N_{\varepsilon}} \mathcal{R}_{m,i,j,\varepsilon,\varepsilon_{1}} \ \ \ \ \text{with}  \ \ \ N_{\varepsilon} = O(\frac{\delta}{\varepsilon}).\label{A_R}
\end{equation}
For each rectangle we define the representative outward normal 
\[
n_{m,i,j,\varepsilon}:= \frac{1}{\sqrt{1+ |\partial_{1} \eta_{m} (c_{m, i,j,\varepsilon})|^{2} + |\partial_{2} \eta_{m}( c_{m, i,j,\varepsilon})|^{2}  }} (\partial_{1} \eta_{m} ( c_{ m,i,j,\varepsilon}), \partial_{2} \eta _{m}( c_{ m,i,j,\varepsilon}),-1).
\] 
{ Let $\{\hat{x}_{1,m,i,j,\varepsilon},\hat{x} _{2,m,i,j,\varepsilon}\}\subset\mathbb{S}^{2}$ be an orthonormal basis of the tangent space of $\partial\Omega$ at $(c_{m,i,j,\varepsilon}, \eta_{m}(c_{m,i,j,\varepsilon}))$. Remark that the three vectors $\hat{x}_{1,m,i,j,\varepsilon}$, $\hat{x} _{2,m,i,j,\varepsilon},$ and $n_{m,i,j,\varepsilon}$ are fixed for each $m,i,j,\varepsilon$ and that $\{ \hat{x}_{1,m, i,j,\varepsilon},\hat{x} _{2,m, i,j,\varepsilon}, n_{m,i,j,\varepsilon} \}$ is an orthonormal basis of $\mathbb{R}^{3}$.}

We split the tangent velocity space at $(c_{m,i,j,\varepsilon}, \eta_{m}(c_{m,i,j,\varepsilon}))\in\partial\Omega$ as 
\[ 
\big\{ v \in \mathbb{R}^{3} :    v\cdot n_{ m,i,j, \varepsilon} =0  \big\} \   \subseteq  { \  \bigcup_{\ell=0}^{L_{\varepsilon}}  {\Theta}_{m, i,j,\varepsilon,\varepsilon_{1}, \ell},} \ \ \ \text{with} \ \ L_{\varepsilon} = O(\frac{1}{\varepsilon}),
\]   
where  
\begin{equation}\label{Theta}
\begin{split}
&  {\Theta} _{m, i,j,\varepsilon,\varepsilon_{1},\ell}  \\
& : = 
 \Big\{    
 r_{v} \cos \theta_{v} \cos\phi_{v} \hat{x}_{1,m, i,j,\varepsilon} 
 + r_{v}  \sin\theta_{v} \cos\phi_{v} \hat{x}_{2,m, i, j,\varepsilon} 
 + r_{v} \sin \phi_{v}  n_{m, i,j,\varepsilon} \in \mathbb{R}^{3} : \\
 & \ \ \ \ \ \    \ \  \ \ \ \ \ \  \ \ \ \ 
 |r_{v} \sin\phi_{v}| < 8 C_{\eta} \varepsilon_{1} \     \text{for} \ r_{v} \in [0,1],   \ \  |  \sin\phi_{v}| < 8 C_{\eta} \varepsilon_{1}      \ \text{for} \ r_{v} \in [1,\infty),     \\ 
  & \ \ \ \ \ \    \ \  \ \ \ \ \ \  \ \ \ \  |\theta_{v}- \varepsilon \ell| < \varepsilon_{1}      \ \text{for} \ r_{v} \in [0,\infty)
  \Big\},
 \end{split}
\end{equation} 
with the constant $C_{\eta}>0$ from (\ref{eta_C2}).
 
Remark that {for $\varepsilon_1\ge\varepsilon$,
\begin{equation}\label{in_Theta} 
 \bigcup_{\ell=0}^{L_{\varepsilon}}  {\Theta}_{m, i,j,\varepsilon ,\varepsilon_{1},\ell} 
 =\left\{ v\in\R^3: \ 
 \begin{array}{c} 
 |v \cdot n_{m,i,j,\varepsilon}| < 8 C_{\eta} \varepsilon_{1}  \text{ for }   |v|\leq 1,    
 \\
   \text{ or }  \  \big| \frac{v}{|v|}\cdot n_{m,i,j,\varepsilon}\big|< 8 C_{\eta} \varepsilon_{1}  \text{  for } |v|\geq 1 
 \end{array}
   \right\}. 
\end{equation}}

Now we are ready to construct {the desired} open cover corresponding to $\mathcal{R}_{m,i,j,\varepsilon,\varepsilon_{1}} \times \Theta_{m,i,j,\varepsilon,\varepsilon_{1},\ell}$ as 
\begin{equation}\label{O_mijl}
\begin{split}
\mathcal{O}_{m,i,j,\varepsilon,\varepsilon_{1},\ell}  \ : = \ & \Big[\bigcup_{x \in \mathcal{X}_{m,i,j,\varepsilon,\varepsilon_{1} ,\ell  }} B_{\mathbb{R}^{3}}( x; \varepsilon_{1})\Big] \times \Theta_{m,i,j,\varepsilon,\varepsilon_{1},\ell},
\end{split}
\end{equation}
where  
\begin{equation}\label{cal_X}
\begin{split}
 \mathcal{X}_{m,i,j,\varepsilon,\varepsilon_{1},\ell} 
 & : = \Big\{  (x_{1},x_{2}, \eta_{m} (x_{1},x_{2}) ) +    \tau [ \cos \theta \hat{x}_{1,m,i,j,\varepsilon} + \sin\theta \hat{x}_{2,m,i,j,\varepsilon}] + s n_{m,i,j,\varepsilon} \in \mathbb{R}^{3}: \\  
 & \ \ \ \ \ \ \ \ \ \ \ \ \  \ \ \ \ \ \ \ \ \     (x_{1},x_{2}) \in \mathcal{R}_{m,i,j,\varepsilon, \varepsilon_{1}}, \ \theta \in \big(\varepsilon \ell-\varepsilon_{1},  \varepsilon \ell+\varepsilon_{1} \big) , \ s\in (-\varepsilon, \varepsilon) \\
& \ \ \ \ \ \ \ \ \ \ \ \ \ \ \ \ \ \ \ \ \ \     \tau \in {\big[0, t_{\mathbf{f}} \big((x_{1},x_{2}, \eta_{m} (x_{1},x_{2}) ),   \cos \theta \hat{x }_{1,m,i,j,\varepsilon} + \sin\theta \hat{x}_{2,m,i,j,\varepsilon} \big) \big]}   \Big\}.
\end{split}
\end{equation}
We note that $\mathcal{O}_{m,i,j,\varepsilon,\varepsilon_{1},\ell}$ is an infinite union of open sets and hence is an open set. 

Finally we define
\begin{equation}\label{O_ep}
 \mathcal{O}_{\varepsilon,\varepsilon_{1}}  :=  \bigcup_{m,i,j,\ell} \mathcal{O}_{m,i,j,\varepsilon,\varepsilon_{1} ,\ell}  \ \cup \ \big[\mathbb{R}^{3} \times  B_{\mathbb{R}^{3}}(0;\varepsilon_{1}) \big],
\end{equation}
where $1\leq m \leq  M_{\Omega, \delta} = O(\frac{1}{\delta^{2}}), \ -N_{\varepsilon} \leq i,j \leq  N_{\varepsilon}= O(\frac{\delta}{\varepsilon}), \ 0 \leq \ell \leq L_{\varepsilon}= O(\frac{1}{\varepsilon}).$ Since $\mathcal{O}_{\varepsilon, \varepsilon_{1}}$ is a union of open sets, it is an open set.

\vspace{8pt}

\noindent \textit{Proof of (\ref{SinO}):} Suppose there exists $(x,v) \in \mathfrak{S}_{\mathrm{B}}$. By the definition {of $\mathfrak{S}_{\mathrm{B}}$ in (\ref{singular_set})} there exists 
$y= x_{\mathbf{b}}(x,v) \in \partial\Omega,$ such that $x= y + t_{\mathbf{f}}(y,v)v$ and $v\cdot n(y)=0.$ Then $y\in\mathcal{U}_{m}$ for some $m$. Without {loss of generality} (up to rotations and translations) we may assume that $y=(y_{1}, y_{2} ,\eta_{m}(y_{1},y_{2}))$ and $(y_{1}, y_{2}) \in \mathcal{R}_{m,i,j,\varepsilon,\varepsilon_{1}}$ for some $i,j.$ 

{First} we consider the case of $|v|\leq 1$. Then from (\ref{eta_C2})
\begin{eqnarray*}
|v\cdot n_{m,i,j,\varepsilon}| &\leq& |v\cdot n(y)| + |v\cdot (n(y) - n_{m,i,j,\varepsilon})|\\
&\leq& 4\varepsilon_{1} || \eta ||_{C^{2} (\mathcal{R}_{m,i,j,\varepsilon, \varepsilon_{1}})} \ 
 \leq  \ 4 \varepsilon_{1} || \eta ||_{C^{2} (\mathcal{A}_{m})}\\
&\leq& 8 C_{\eta} \varepsilon_{1}.
\end{eqnarray*}
By the statement of (\ref{in_Theta}), $v \in \bigcup_{\ell=0}^{L_{\varepsilon}} \Theta_{m,i,j,\varepsilon,\ell}$ and hence $(x,v) \in \mathcal{O}_{m,i,j,\varepsilon, \varepsilon_{1},\ell}\subset \mathcal{O}_{\varepsilon,\varepsilon_{1}}.$   

Secondly we consider the case of $|v|\geq 1$. Then we check that 
\begin{equation}\notag
\begin{split}
\big|n_{m,i,j,\varepsilon} \cdot \frac{v}{|v|} \big| \ = \ &\Big| n(y_{1},y_{2}, \eta_{m}(y_{1}, y_{2})) \cdot \frac{v}{|v|} +  [n_{m,i,j,\varepsilon} -n(y_{1},y_{2}, \eta_{m}(y_{1}, y_{2}))  ]\cdot \frac{v}{|v|} \Big|\\
 \ \leq  \ & 
 \Big| n(y_{1},y_{2}, \eta_{m}(y_{1}, y_{2})) \cdot \frac{v}{|v|} \Big|+ \Big| [n_{m,i,j,\varepsilon} -n(y_{1},y_{2}, \eta_{m}(y_{1}, y_{2}))  ]\cdot \frac{v}{|v|} \Big|
 \\
 \ = \ & 0+ \big|n(c_{m,i,j,\varepsilon},\eta_{m}(c_{m,i,j,\varepsilon}) )  -n(y_{1},y_{2}, \eta_{m}(y_{1}, y_{2}))  \big|\\
\ \leq \ & 
\frac{\big| \sqrt{1+ |\nabla \eta (y_{1},y_{2})|^{2}}- \sqrt{1+ |\nabla \eta (c_{m,i,j,\varepsilon})|^{2}}  \big|}{\sqrt{1+ |\nabla \eta ( c_{m,i,j,\varepsilon})|^{2}}}
+ \frac{|\nabla \eta(c_{m,i,j,\varepsilon}) - \nabla\eta (y_{1},y_{2})|}{\sqrt{1+ |\nabla \eta ( c_{m,i,j,\varepsilon})|^{2}}}.
\end{split}
\end{equation}
Due to (\ref{choice_delta}), we deduce $\frac{1}{\sqrt{1+ |\nabla \eta ( c_{m,i,j,\varepsilon})|^{2}}} \leq  \sqrt{\frac{64}{65}}$. Use (\ref{eta_C2}), we have 
\begin{eqnarray*}
|\nabla \eta(c_{m,i,j,\varepsilon}) - \nabla\eta (y_{1},y_{2})| &\leq& 4 \varepsilon_{1} \times || \eta ||_{C^{2}(\mathcal{R}_{m,i,j,\varepsilon, \varepsilon_{1}})}\\
 & \leq&    4 \varepsilon_{1} \times 
|| \eta ||_{C^{2}( \mathcal{A}_{m}
)}\\
&\leq& 4 C_{\eta} \varepsilon_{1}.
\end{eqnarray*}
Therefore we conclude
\[
\big| n_{m,i,j,\varepsilon} \cdot \frac{v}{|v|}\big| \ \leq \  8 C_{\eta} \varepsilon_{1}   .
\]
By (\ref{in_Theta}), $v \in \bigcup_{\ell=0}^{L_{\varepsilon}} \Theta_{m,i,j,\varepsilon,\ell}$ and hence $(x,v) \in   \mathcal{O}_{\varepsilon,\varepsilon_{1}}.$

\vspace{8pt}

\noindent \textit{Proof of (\ref{barO_O}):} It suffices to show that there exists a constant $C_{*}\gg1$ such that if $(x,v) \in \overline{\mathcal{O}_{\varepsilon, \varepsilon_{1}}} $ then $(x,v) \in\mathcal{O}_{\varepsilon, C_{*}  \varepsilon_{1} }$. 

{Since in the definition \eqref{O_ep} the union on $m,i,j,\ell$ is finite, we have}
\begin{equation}\notag
\begin{split}
\overline{\mathcal{O}_{\varepsilon, \varepsilon_{1}}} =& \bigcup_{m,i,j,\ell} \overline{\mathcal{O}_{m,i,j,\varepsilon,\varepsilon_{1},\ell}}  \ \cup \ 
\big[ \mathbb{R}^{3} \times \big\{ v\in \mathbb{R}^{3} :  |v|\leq\varepsilon_{1}  \big\} \big] \\
=& \bigcup_{m,i,j,\ell} \bigg[ \underbrace{\overline{\Big( \bigcup_{x \in \mathcal{X}_{m,i,j,\varepsilon,\varepsilon_{1} ,\ell  }  } B_{\mathbb{R}^{3}} (x;\varepsilon_{1}) \Big)  }} \ \times \  \overline{\Theta_{m,i,j,\varepsilon,\varepsilon_{1},\ell}}\bigg] \ \ \cup \ \  \big[ \mathbb{R}^{3} \times \big\{ v\in \mathbb{R}^{3} :  |v|\leq\varepsilon_{1}  \big\} \big] .
\end{split}
\end{equation}
First we define an open set including the underbraced set (a close set). For $0<  \varsigma,$ we define
\begin{equation}\label{delta_neighbor}
\begin{split}
&\bigcup_{x\in \mathcal{X}_{m,i,j,\varepsilon,\varepsilon_{1},\ell}   } \bigcup_{y \in B_{\mathbb{R}^{3}}(x;\varepsilon_{1}) } B_{\mathbb{R}^{3}}(y;\varsigma)\\
&=\Big\{ z\in \mathbb{R}^{3} : \text{there exists } x\in \mathcal{X}_{m,i,j,\varepsilon,\varepsilon_{1},\ell}  \ \text{and} \  y \in B_{\mathbb{R}^{3}}(x;\varepsilon_{1}) 
 \ \text{such that} \ 
 z \in B_{\mathbb{R}^{3}}(y; \varsigma) \Big\}.
\end{split}
\end{equation}
Since it is an infinite union of open balls, (\ref{delta_neighbor}) is open and the underbraced set is contained in (\ref{delta_neighbor}) for any $\varsigma>0.$ 

Now we claim that, there exists $C_{*}= C_{*}(\Omega ) \gg 1$ such that  such that for $0< \varepsilon \leq \varepsilon_{1} \ll 1$, there
exists $0 < \varsigma = \varsigma(\varepsilon_{1},C^{*}) \ll 1$ such that
\begin{equation}\label{delta_incls}
\bigcup_{x\in \mathcal{X}_{m,i,j,\varepsilon, \varepsilon_{1},\ell   }} \bigcup_{y \in B_{\mathbb{R}^{3}}(x;\varepsilon_{1}) } B_{\mathbb{R}^{3}}(y;\varsigma) 
\ \subset \  \bigcup_{x \in \mathcal{X}_{m,i,j,\varepsilon,C_{*}\varepsilon_{1}}} B_{\mathbb{R}^{3}} (x;C_{*}\varepsilon_{1}).
\end{equation} 
Choose $z \in \bigcup_{x\in \mathcal{X}_{m,i,j,\varepsilon,\varepsilon_{1},\ell}  } \bigcup_{y \in B_{\mathbb{R}^{3}}(x;\varepsilon_{1}) } B_{\mathbb{R}^{3}}(y;\varsigma).$ From (\ref{delta_neighbor}) there exist $x\in \mathcal{X}_{m,i,j,\varepsilon,\varepsilon_{1},\ell}  $ and $y \in B_{\mathbb{R}^{3}}(x;\varepsilon_{1})$ such that $z \in B_{\mathbb{R}^{3}}(y;\varsigma).$ If we choose $\varsigma < \e_{1}$ then $|x-z|\leq |x-y| + |y-z| \leq 2\varepsilon_{1}  < C_{*}  \varepsilon_{1}$ and therefore $z\in B_{\mathbb{R}^{3}}(x;C_{*}\varepsilon_{1})$. Clearly $x\in \mathcal{X}_{m,i,j,\varepsilon ,C_{*} \e_{1}}$. This proves our claim (\ref{delta_incls}).

On the other hand, from $(\ref{Theta}), \ C_{*} \gg1 $ and the fact that the vectors $\hat{x}_{1,m,i,j,\varepsilon}$, $\hat{x} _{2,m,i,j,\varepsilon},$ and $n_{m,i,j,\varepsilon}$ are fixed for given $m,i,j$,
\begin{equation}\label{theta_incls}
\begin{split}
\overline{ {\Theta} _{m, i,j,\varepsilon,\varepsilon_{1} ,\ell} } & \ = \ \Big\{ v=  r_{v} \cos \theta_{v} \hat{x}_{1,m, i,j,\varepsilon} +  r_{v} \sin\theta_{v}  \hat{x}_{2,m, ij,\varepsilon} +r_{v} \sin \phi_{v}  n_{m, i,j,\varepsilon} \in \mathbb{R}^{3}  : \\
 & \ \ \ \ \ \    \ \  \ \ \ \ \ \ 
 |r_{v} \sin\phi_{v}| \leq 8 C_{\eta} \varepsilon_{1} \     \text{for} \ r_{v} \in [0,1],   \ \  |  \sin\phi_{v}| \leq 8 C_{\eta} \varepsilon_{1}      \ \text{for} \ r_{v} \in [1,\infty),     \\ 
  & \ \ \ \ \ \    \ \  \ \ \ \ \ \  |\theta_{v}- \varepsilon \ell| \leq \varepsilon_{1}      \ \text{for} \ r_{v} \in [0,\infty)  \Big\}\\
& \ \subset \  
\Big\{ v=  r_{v} \cos \theta_{v} \hat{x}_{1,m, i,j,\varepsilon} +  r_{v} \sin\theta_{v}  \hat{x}_{2,m, ij,\varepsilon} +r_{v} \sin \phi_{v}  n_{m, i,j,\varepsilon} \in \mathbb{R}^{3}  : \\
 & \ \ \ \ \ \    \ \  \ \ \ \ \ \ 
 |r_{v} \sin\phi_{v}| < 8 C_{\eta} C_{*}\varepsilon_{1} \     \text{for} \ r_{v} \in [0,1],   \ \  |  \sin\phi_{v}| < 8 C_{\eta} C_{*} \varepsilon_{1}      \ \text{for} \ r_{v} \in [1,\infty),     \\ 
  & \ \ \ \ \ \    \ \  \ \ \ \ \ \  |\theta_{v}- \varepsilon \ell| < C_{*} \varepsilon_{1}      \ \text{for} \ r_{v} \in [0,\infty)  \Big\}  \\  
&  \ =  \ \Theta_{m,i,j, \varepsilon,C_{*}\varepsilon_{1},\ell}. %\ \sout{\cup  \ B_{\mathbb{R}^{3}}(0;C_{*}\varepsilon_{1})}.
 \end{split}
\end{equation}

Finally we conclude, from (\ref{delta_incls}) and (\ref{theta_incls}),
\begin{equation*}
\begin{split}
\overline{\mathcal{O}_{\varepsilon,\varepsilon_{1}}}  \ &\subset \  
\bigcup_{m,i,j,\ell} \Big[ \bigcup_{x \in \mathcal{X}_{m,i,j,\varepsilon,C_{*}\varepsilon_{1}}}  B_{\mathbb{R}^{3}} (x;C_{*}\varepsilon_{1}) \times \Theta_{m,i,j, \varepsilon,C_{*}\varepsilon_{1},\ell}\Big]
\ \cup  \  \big[  \mathbb{R}^{3} \times B_{\mathbb{R}^{3}}(0;C_{*}\varepsilon_{1})\big] \ \\
& =  \ \mathcal{O}_{\varepsilon,C_{*}\varepsilon_{1}}.
\end{split}
\end{equation*}

\vspace{4pt}

\noindent \textit{Proof of (\ref{small_bulk}):} From (\ref{O_ep}), we deduce 
\begin{eqnarray*} 
&&\iint_{\Omega\times\mathbb{R}^{3}} \mathbf{1}_{\mathcal{O}_{\varepsilon,C_{*}\varepsilon}} (x,v) e^{-\theta |v|^{2}} \mathrm{d}v \mathrm{d}x\\
&\leq& \sum_{m,i,j,\ell}  \iint_{\Omega\times\mathbb{R}^{3}} \mathbf{1}_{\mathcal{O}_{m,i,j,\varepsilon, C_{*}\varepsilon,\ell}} (x,v) e^{-\theta |v|^{2}} \mathrm{d}v \mathrm{d}x +  \mathrm{m}_{3}(\Omega) 
O(|\varepsilon|^{3})  \\
&\leq& M_{\Omega, \delta}   (2N_{\varepsilon})^{2}   L_{\varepsilon} \times
\sup_{m,i,j ,\ell}\iint_{\Omega\times\mathbb{R}^{3}} \mathbf{1}_{\mathcal{O}_{m,i,j,\varepsilon, C_{*}\varepsilon,\ell}} (x,v) e^{-\theta |v|^{2}} \mathrm{d}v \mathrm{d}x +  \mathrm{m}_{3}(\Omega) 
O(|\varepsilon|^{3})\\
& \ \lesssim_{\Omega}&  O(\frac{1}{\varepsilon^{3}})\times
\sup_{m,i,j,\ell}\iint_{\Omega\times\mathbb{R}^{3}} \mathbf{1}_{\mathcal{O}_{m,i,j,\varepsilon,C_{*}\varepsilon,\ell}} (x,v) e^{-\theta |v|^{2}} \mathrm{d}v \mathrm{d}x + 
O(|\varepsilon|^{3}) .
\end{eqnarray*}
Therefore, to prove (\ref{small_bulk}), it suffices to show 
\begin{equation}\label{ep_4}
\sup_{m,i,j,\ell}\iint_{\Omega\times\mathbb{R}^{3}} \mathbf{1}_{\mathcal{O}_{m,i,j,\varepsilon,C_{*}\varepsilon,\ell}} (x,v) e^{-\theta |v|^{2}} \mathrm{d}v \mathrm{d}x \ \lesssim_{\delta,\Omega} \   \varepsilon ^{4}.
\end{equation}
From (\ref{Theta}),
\begin{eqnarray*} 
&&\int_{\mathbb{R}^{3}} \mathbf{1}_{\Theta_{m,i,j,\varepsilon, C_{*}\varepsilon ,\ell}} (v) e^{- \theta |v|^{2}} \mathrm{d}v\\
&=& \int_{|v| \leq 1} + \int_{|v|\geq 1}\\
&=& \int_{|r_{v} \sin\phi_{v}| \leq 8 C_{\eta} C_{*}\varepsilon} \mathrm{d}|r_{v} \sin\phi_{v}|
\int_{0}^{\infty} |r_{v} \cos\phi_{v}| e^{-\theta | r_{v} \cos\phi_{v}|^{2}} \mathrm{d} | r_{v} \cos\phi_{v}|
\int_{|\theta_{v}- \varepsilon \ell| < C_{*}\varepsilon} \mathrm{d} \theta_{v}\\
& &+\int_{1}^{\infty} |r_{v}|^{2}e^{-\theta |r_{v}|^{2}} \mathrm{d}r_{v}
\int_{|\sin \phi_{v}| < 8 C_{\eta} C_{*}\varepsilon} \mathrm{d}\phi_{v}
\int_{|\theta_{v}- \varepsilon \ell| < C_{*}\varepsilon} \mathrm{d}\theta_{v} \\
&\lesssim_{\Omega}&  \varepsilon ^{2}. 
\end{eqnarray*}
Now we claim that, for $\varepsilon_1\ge\varepsilon$,
\begin{equation}\label{v_e12}
\mathrm{m}_{3} \Big( \bigcup_{x\in\mathcal{X}_{m,i,j,\varepsilon,\varepsilon_{1},\ell}  }   B_{\mathbb{R}^{3}} (x;\varepsilon_{1}) \Big) \ \lesssim_{\Omega}  \  \varepsilon_{1} ^{2}.
\end{equation}
{Without loss of generality we assume that $i=j=0$ and $l=0$. Therefore $c_{m,i,j,\varepsilon}=0$ in (\ref{lattice}) and}
\begin{equation}\notag
\begin{split}
 & \mathcal{X}_{m,i,j,\varepsilon,\varepsilon_{1},\ell}   \subset \Big\{  (x_{1},x_{2}, \eta (x_{1},x_{2}) ) +    \tau [ \cos \theta\mathbf{e}_{1} + \sin\theta \mathbf{e}_{2}   ] + s  \mathbf{e}_{3}\in \mathbb{R}^{3}: \\  
 & \ \ \ \ \ \ \ \ \ \ \ \ \  \ \ \ \ \ \ \ \ \  \ \ \  (x_{1},x_{2}) \in (-\varepsilon_{1}, \varepsilon_{1})^{2}, \ \theta \in (-\varepsilon_{1} ,  \varepsilon_{1}  ), \\
& \ \ \ \ \ \ \ \ \ \ \ \ \ \ \ \ \ \ \ \ \ \   \ \ \  \tau \in {\big[0, t_{\mathbf{f}} ((x_{1},x_{2}, \eta (x_{1},x_{2}) ),   \cos \theta  \mathbf{e}_{1} + \sin\theta  \mathbf{e}_{2} ) \big]} , \ s\in (-\varepsilon_{1}, \varepsilon_{1})  \Big\}.
\end{split}
\end{equation}
Since $\Omega$ is bounded, we have that $\mathrm{diam}({\Omega}) = \sup_{x,y \in \Omega}|x-y| < + \infty$ and hence
\[
t_{\mathbf{f}} ((x_{1},x_{2}, \eta (x_{1},x_{2}) ),   \cos \theta  \mathbf{e}_{1} + \sin\theta  \mathbf{e}_{2} ) \leq \mathrm{diam}(\Omega).
\]
We have
\begin{equation*}
\begin{split}
 \bigcup_{x\in\mathcal{X}_{m,i,j,\varepsilon,\varepsilon_{1},\ell}  }   B_{\mathbb{R}^{3}} (x;\varepsilon_{1}) & \subset 
 \bigcup_{\tau=0}^{2 \mathrm{diam}({\Omega})} B_{\mathbb{R}^{3}} (\tau \mathbf{e}_{1}; [10+ || \eta ||_{C^{1}(\mathcal{A}_{m})} + \tau || \eta ||_{C^{2}(\mathcal{A}_{m})} ] \varepsilon_{1} ).
\end{split}
\end{equation*}
More precisely $\bigcup_{x\in\mathcal{X}_{m,i,j,\varepsilon,\varepsilon_{1},\ell}  }   B_{\mathbb{R}^{3}} (x;\varepsilon_{1})$ is included in the truncated cone with height $\mathrm{diam}(\Omega)$, top radius $[10+ || \eta ||_{C^{1}(\mathcal{A}_{m})}]\varepsilon_{1}$, and the bottom radius $[10+ || \eta ||_{C^{1}(\mathcal{A}_{m})} + \mathrm{diam}(\Omega) || \eta ||_{C^{2}(\mathcal{A}_{m})}]\varepsilon_{1}$.

Therefore, using (\ref{choice_delta}) and (\ref{eta_C2}), we conclude (\ref{v_e12})
\begin{eqnarray*}
\mathrm{m}_{3}\Big( \bigcup_{x\in\mathcal{X}_{m,i,j,\varepsilon,\varepsilon_{1},\ell}  }   B_{\mathbb{R}^{3}} (x;\varepsilon_{1}) \Big)&  \leq  &
\mathrm{m}_{3} \bigg(
\bigcup_{\tau=0}^{2 \mathrm{diam}({\Omega})} B_{\mathbb{R}^{3}} (\tau \mathbf{e}_{1}; [10+ || \eta ||_{C^{1} (\mathcal{A}_{m})} + \tau || \eta ||_{C^{2} (\mathcal{A}_{m})} ] \varepsilon_{1} )
\bigg)\\
 & \leq &  { 3\, \mathrm{diam}(\Omega)\,}
 \Big[
 10 + || \eta ||_{C^{1}(\mathcal{A}_{m})} + \mathrm{diam} (\Omega) || \eta ||_{C^{2}(\mathcal{A}_{m})}
 \Big]^{2} \times (\varepsilon_{1})^{2}\\
 &\leq& { 3\,  \mathrm{diam}(\Omega)\,}\big[10 + \frac{1}{8} + C_{\eta} \mathrm{diam}(\Omega) \big]^{2} (\varepsilon_{1})^{2}\\
 &\lesssim_{\Omega}&  \varepsilon_{1} ^{2}.
\end{eqnarray*}
Finally {selecting $\varepsilon_1=C_* \varepsilon$ in \eqref{v_e12}} we conclude (\ref{ep_4}) as
\begin{equation*}\label{e2v}
\begin{split}
& \mathrm{m}_{3} \Big( \bigcup_{x\in\mathcal{X}_{m,i,j,\varepsilon,  C_* \varepsilon,\ell}}   B_{\mathbb{R}^{3}} (x;C_* \varepsilon) \Big) \times  \int_{\mathbb{R}^{3}} \mathbf{1}_{\Theta_{m,i,j,\varepsilon,C_* \varepsilon,\ell}}(v)  e^{-\theta |v|^{2}} \mathrm{d}v   \\
 &\ \lesssim   \ \mathrm{m}_{3} \Big( \bigcup_{x\in\mathcal{X}_{m,i,j,\varepsilon, C_* \varepsilon,\ell}}   B_{\mathbb{R}^{3}} (x;\varepsilon_{1}) \Big) \times  (\varepsilon)^{2}\\
 & \ \lesssim \   \varepsilon ^{4}.
\end{split}
\end{equation*} 

 \vspace{8pt}

\noindent \textit{Proof of (\ref{dist_ep}):} Due to (\ref{SinO}), it suffices to show that there exists $C_{2}  = C_{2} (C_{*})>0$ such that
\begin{equation}\label{dist_ep_cover}
\mathrm{dist} \big( \bar{\Omega} \times \mathbb{R}^{3} \backslash \mathcal{O}_{\varepsilon, C_{*}\varepsilon }   , \ \overline{ \mathcal{O}_{\varepsilon,\varepsilon}  } \big) > C_{2} \varepsilon.
\end{equation}
By the definition of $\mathcal{O}_{\varepsilon,\varepsilon}$ in (\ref{O_ep}),
\begin{eqnarray}\notag 
&&\mathrm{dist} ( \bar{\Omega} \times \mathbb{R}^{3} \backslash \mathcal{O}_{\varepsilon,C_{*}\varepsilon}, \overline{ \mathcal{O}_{\varepsilon,\varepsilon}  } ) \nonumber\\
 & =& \inf \big\{ 
 |(x,v)-(y,u)| : (x,v) \in  ( \mathcal{O}_{\varepsilon,C_{*}\varepsilon})^{c}, \ (y,u) \in \overline{{O}_{\varepsilon,\varepsilon}}
\big\} \nonumber\\
&=& \inf_{m,i,j,\ell} \inf \big\{ |(x,v)-(y,u)| : (x,v) \in ( \mathcal{O}_{\varepsilon,C_{*}\varepsilon})^{c}, \ (y,u)  \in \overline{\mathcal{O}_{m,i,j,\varepsilon,\varepsilon,\ell}} \cup [\mathbb{R}^{3} \times \overline{  B_{\mathbb{R}^{3}} (0;\varepsilon)  }   ]
\big\} \nonumber\\
&\geq  & \inf_{m,i,j,\ell} \inf\big\{ |(x,v)-(y,u)| : (x,v) \in (\mathcal{O}_{m,i,j,\varepsilon,C_{*}\varepsilon, \ell})^{c} \cap [ \mathbb{R}^{3} \times B_{\mathbb{R}^{3}} (0,C_{*}\varepsilon)^{c}], \nonumber\\
&& \ \ \ \ \ \ \ \ \ \ \ \ \ \  \ \ \ \ \ \ \ \ \ \ \ \ \   \ \ \ \ \ \ \ \    (y,u) \in \overline{\mathcal{O} _{ m,i,j,\varepsilon,\varepsilon,\ell}} \cup [ \mathbb{R}^{3} \times B_{\mathbb{R}^{3}}(0;\varepsilon)]
\big\} \nonumber \\
&=& \inf_{m,i,j,\ell}\min\Big\{ \inf \big\{
|(x,v)-(y,u)| : (x,v) \in (\mathcal{O}_{m,i,j,\varepsilon,C_{*} \varepsilon, \ell})^{c} \cap [\mathbb{R}^{3} \times B_{\mathbb{R}^{3}}(0;C_{*} \varepsilon)^{c}] , \notag \\ 
&&  \ \ \ \ \ \ \ \ \ \ \ \ \ \ \ \ \ \ \ \ \ \ \ \ \ \ \ \ \ \ \ \ \ \ \ \ \ \ \ \ \ \ \ \ (y,u) \in \mathbb{R}^{3} \times B_{\mathbb{R}^{3}} (0;\varepsilon)  
\big\}, \label{dist_O1}\\
&&  \ \ \ \ \ \ \ \ \ \   \ \ \ \ \ \ 
 \inf \big\{ |(x,v)-(y,u)| :   (x,v) \in (\mathcal{O}_{m,i,j,\varepsilon,C_{*}\varepsilon, \ell})^{c}  \ \cap \ [ \mathbb{R}^{3} \times B_{\mathbb{R}^{3}} (0,C_{*}\varepsilon)^{c}], \nonumber \\ 
 &&   \ \ \ \ \ \ \ \ \ \ \ \ \ \ \ \ \ \ \ \ \ \ \ \ \ \ \ \ \ \ \ \ \ \ \ \ \ \ \ \ \ \ \ \  (y,u) \in \overline{\mathcal{O} _{ m,i,j,\varepsilon,\varepsilon,\ell}} \ \cap \ [ \mathbb{R}^{3} \times B_{\mathbb{R}^{3}} (0, \varepsilon)^{c}]   \big\}  \ 
\Big\}. \label{dist_O2}
\end{eqnarray}
Clearly, 
\begin{eqnarray*}
(\ref{dist_O1}) &\geq & \inf \big\{
|(x,v)-(y,u)| : (x,v) \in  \mathbb{R}^{3} \times B_{\mathbb{R}^{3}}(0;C_{*} \varepsilon)^{c}  , \ (y,u) \in \mathbb{R}^{3} \times B_{\mathbb{R}^{3}} (0;\varepsilon)  
\big\} \\
&\geq & \inf\big\{ |v-u| : v \in B_{\mathbb{R}^{3}}(0; C_{*} \varepsilon)^{c}, \  
u \in B_{\mathbb{R}^{3}}(0;\varepsilon)
  \big\}\\
&=& (C_{*}-1) \varepsilon.
\end{eqnarray*}

Now we claim that (\ref{dist_O2}) is bounded below by the minimum of (\ref{dist_O2a}) and (\ref{dist_O2b}):
\begin{eqnarray}
&&(\ref{dist_O2})\nonumber \\
&\geq&  \min  \bigg(  \inf \Big\{
|(x,v)-(y,u)| : (x,v) \in \bigcup_{x\in \mathcal{X}_{m,i,j, \varepsilon, C_{*}\varepsilon},\ell} B_{\mathbb{R}^{3}} (x,C_{*}\varepsilon) \times\big[ \big( \Theta_{m,i,j,\varepsilon,C_{*}\varepsilon, \ell}  \big)^{c} \backslash B_{\mathbb{R}^{3}}(0;C_{*}\varepsilon)\big], \nonumber \\
&& \ \ \ \ \ \ \ \ \  \ \ \ \ \ \ \ \ \ \ \ \ \ \ \ \ \ \ \   (y,u)  \in \Big[\bigcup_{x\in \mathcal{X}_{m,i,j,\varepsilon,\frac{C_{*}}{2} \varepsilon,\ell}}
 {B_{\mathbb{R}^{3}}(x;  \frac{C_{*}}{2}\varepsilon )   } \Big] \ \times  \ 
\big[ \ \overline{\Theta_{m,i,j,\varepsilon,\varepsilon,\ell}} \backslash B_{\mathbb{R}^{3}}(0; \varepsilon)\big]
\Big\}, \ \ \ \ \ \ \label{dist_O2a}  \\
&&  \inf \Big\{     
|(x,v)-(y,u)| :   (x,v) \in \big[ \bigcap_{ x \in\mathcal{X}_{m,i,j,\varepsilon,C_{*}\varepsilon,\ell}} \big(B_{\mathbb{R}^{3}} (x; C_{*} \varepsilon) \big)^{c}\Big] \times \big[  \mathbb{R}^{3} \backslash B_{\mathbb{R}^{3}}(0;C_{*}\varepsilon)   \big]
 , \notag\\
&& \ \ \ \ \ \ \ \ \ \ \ \ \ \ \ \ \ \ \  \ \ \ \ \ \ \ \ \ (y,u)  \in \Big[\bigcup_{x\in \mathcal{X}_{m,i,j,\varepsilon,\frac{C_{*}}{2} \varepsilon,\ell  }  }
 {B_{\mathbb{R}^{3}}(x;  \frac{C_{*}}{2}\varepsilon )   } \Big] \ \times  \ 
\big[ \ \overline{\Theta_{m,i,j,\varepsilon,\varepsilon,\ell}}\backslash B_{\mathbb{R}^{3}}(0; \varepsilon)\big]
\Big\}
 \bigg). \ \ \ \ \ \ \label{dist_O2b}
\end{eqnarray}

 Firstly, we divide $\{x \in ( \mathcal{O}_{m,i,j,\varepsilon,C_{*}\varepsilon,\ell})^{c}\}$ in (\ref{dist_O2}) into two parts: from the definition of $\mathcal{O}_{m,i,j,\varepsilon,C_{*}\varepsilon,\ell}$ in (\ref{O_mijl}), we deduce that 
\begin{equation*}
\begin{split}
 (\mathcal{O}_{m,i,j,\varepsilon,C_{*} \varepsilon, \ell})^{c} 
 \ \  =  &  \  \ 
\Big[ \bigcup_{ x  \in \mathcal{X}_{m,i,j,\varepsilon,C_{*}\varepsilon ,\ell  }} B_{\mathbb{R}^{3}} (x; C_{*} \varepsilon)  \Big] \times \big( \Theta_{m,i,j,\varepsilon,C_{*}\varepsilon, \ell}\big)^{c}\\
&\cup  \   \Big[ \bigcap_{ x \in\mathcal{X}_{m,i,j,\varepsilon,C_{*}\varepsilon,\ell}} \big(B_{\mathbb{R}^{3}} (x; C_{*} \varepsilon) \big)^{c}\Big] \times 
\mathbb{R}^{3}
%\Theta_{m,i,j,\varepsilon,C_{*}\varepsilon, \ell} 
.
\end{split}
\end{equation*}
Therefore, $(\ref{dist_O2})$ is bounded below by the minimum of following two numbers:
\begin{equation}\label{decom_{O}}
\begin{split}
& \inf \big\{ |(x,v)-(y,u)| :   (x,v) \in  \Big[ \bigcup_{ x  \in \mathcal{X}_{m,i,j,\varepsilon,C_{*}\varepsilon,\ell}} B_{\mathbb{R}^{3}} (x; C_{*} \varepsilon)  \Big] \times [\big( \Theta_{m,i,j,\varepsilon,C_{*}\varepsilon, \ell}\big)^{c} \backslash B_{\mathbb{R}^{3}} (0,C_{*}\varepsilon)^{c}],   \\ 
 &    \ \ \ \ \ \ \ \ \ \ \ \   \ \ \ \ \ \ \ \ \ \ \ \ \ \ \ \  (y,u) \in \overline{\mathcal{O} _{ m,i,j,\varepsilon,\varepsilon,\ell}} \ \cap \ [ \mathbb{R}^{3} \times B_{\mathbb{R}^{3}} (0, \varepsilon)^{c}]   \big\}  \ 
\Big\},\\
& \inf \big\{ |(x,v)-(y,u)| :   (x,v) \in  \Big[ \bigcap_{ x \in\mathcal{X}_{m,i,j,\varepsilon,C_{*}\varepsilon,\ell}} \big(B_{\mathbb{R}^{3}} (x; C_{*} \varepsilon) \big)^{c}\Big] \times 
[  \mathbb{R}^{3} \backslash B_{\mathbb{R}^{3}} (0,C_{*}\varepsilon) ]  ,   \\ 
 &    \ \ \ \ \ \ \ \ \ \ \ \   \ \ \ \ \ \ \ \ \ \ \ \ \ \ \ \  (y,u) \in \overline{\mathcal{O} _{ m,i,j,\varepsilon,\varepsilon,\ell}} \ \cap \ [ \mathbb{R}^{3} \times B_{\mathbb{R}^{3}} (0, \varepsilon)^{c}]   \big\}  \ 
\Big\}.
\end{split}
\end{equation}

Secondly, we consider $\{y \in \overline{ \mathcal{O}_{m,i,j,\e, \e, \ell}}\}$. From (\ref{delta_incls}) with $\e_{1} =\e$, for some $\varsigma= \varsigma(\e, C_{*})>0$ 
\begin{equation}\notag
\overline{\bigcup_{x\in \mathcal{X} _{m,i,j,\varepsilon,\varepsilon,\ell}} B_{\mathbb{R}^{3}} (x;\varepsilon) }  \
\subset \ 
\bigcup_{x\in\mathcal{X}_{m,i,j,\varepsilon,\varepsilon,\ell}} \bigcup_{y \in B_{\mathbb{R}^{3}} (x;\varepsilon)} B_{\mathbb{R}^{3}} (y;\varsigma) \
\subset \ \bigcup_{x\in\mathcal{X}_{m,i,j,\varepsilon,\frac{C_{*} }{2} \varepsilon,\ell}} B_{\mathbb{R}^{3}} (x; \frac{C_{*}}{2}\varepsilon),
\end{equation}
and from the definition of $\mathcal{O}_{m,i,j,\varepsilon,\varepsilon,\ell}$ in (\ref{O_mijl}), we conclude
\begin{equation}\notag
\begin{split}
\overline{\mathcal{O}_{m,i,j,\varepsilon,\varepsilon,\ell}}   \ \ \ \   =  &  \ \ \ \overline{\bigcup_{x\in \mathcal{X} _{m,i,j,\varepsilon,\varepsilon,\ell}} B_{\mathbb{R}^{3}} (x;\varepsilon) } \ \times \  
\overline{\Theta_{m,i,j,\varepsilon,\varepsilon,\ell}} \\
  \subset &   \ \ \   \Big[\bigcup_{x\in \mathcal{X}_{m,i,j,\varepsilon,\frac{C_{*}}{2} \varepsilon,\ell}}
 {B_{\mathbb{R}^{3}}(x;  \frac{C_{*}}{2}\varepsilon )   } \Big] \ \times  \ 
\overline{\Theta_{m,i,j,\varepsilon,\varepsilon,\ell}}.
\end{split}
\end{equation}
Therefore, the first number of (\ref{decom_{O}}) is bounded below by (\ref{dist_O2a}) and the second of (\ref{decom_{O}}) by (\ref{dist_O2b}). This proves the claim.

Now we claim that
\[
(\ref{dist_O2a})\gtrsim   \varepsilon, \ \ \ \text{and} \ (\ref{dist_O2b})\gtrsim   \varepsilon.
\]

Firstly, we prove $(\ref{dist_O2a})  \gtrsim   \varepsilon.$ Let $v \in \big( \Theta_{m,i,j,\varepsilon,C_{*}\varepsilon, \ell}  \big)^{c}\backslash B_{\mathbb{R}^{3}}(0;C_{*}\varepsilon)$. By (\ref{Theta})
\begin{eqnarray*} 
v = r_{v} \cos\theta_{v} \cos \phi_{v}\hat{x}_{1,m,i,j,\varepsilon} +r_{v} \sin \theta_{v} \cos \phi_{v} \hat{x}_{2,m,i,j,\varepsilon} + r_{v} \sin\phi_{v} n_{m,i,j,\varepsilon},
\end{eqnarray*}
where 
\begin{eqnarray*}
&&|r_{v}\sin\phi_{v}|  \geq 8 C_{\eta} C_{*} \varepsilon \ \ \text{and}  \ \ |r_{v}| \leq 1,  \\
&\text{or}&  |\sin \phi_{v}| \geq  8  C_{\eta}   C_{*} \varepsilon \   \ \text{and} \  \ |r_{v}| \geq 1,\\
&\text{or}&|\theta_{v}- \varepsilon \ell |  \geq      C_{*} \varepsilon.
\end{eqnarray*}
Let $u \in \overline{\Theta_{m,i,j,\varepsilon,\varepsilon,\ell}} \backslash B_{\mathbb{R}^{3}}(0,  \varepsilon).$ Again from (\ref{Theta}) we have
\begin{eqnarray*} 
u  =  r_{u} \cos\theta_{u} 
\cos \phi_{u}
\hat{x}_{1,m,i,j,\varepsilon} + r_{u} \sin \theta_{u} \cos \phi_{u}\hat{x}_{2,m,i,j,\varepsilon} + r_{u} \sin\phi_{u} n_{m,i,j,\varepsilon},
\end{eqnarray*}
where
\begin{eqnarray*}
&&|\theta_{u}- \varepsilon \ell |  \leq   \varepsilon,\\
&\text{and}&| r_{u}\sin\phi_{u}|  \leq  8C_{\eta} \varepsilon    \ \  \text{for} \ |r_{u}| \leq 1,\\
&\text{and}&  |\sin \phi_{u}| \leq 8C_{\eta} \varepsilon   \ \ \text{for} \  |r_{u}| \geq 1.
\end{eqnarray*}
If $|\theta_{v}-\e \ell| \geq C_{*} \e$ then clearly $|v-u|\gtrsim \e$ since $|\theta_{u} -\e \ell| \leq \e$.

Now we consider the case of $|\theta_{v} -\e \ell| \leq C_{*} \e.$ 

If $|r_{v} |\leq  1$ and $|r_{u}| \leq 1$, then 
\begin{eqnarray*}
|v-u| &\geq& | (v-u)\cdot n_{m,i,j,\varepsilon}   | 
\geq  |  v \cdot n_{m,i,j,\varepsilon}   | - |  u \cdot n_{m,i,j,\varepsilon}   |
\\
& \geq& |r_{v} \sin \phi_{v}| - |r_{u} \sin\phi_{u}|
\geq 8 C_{\eta} C_{*} \varepsilon - 8 C_{\eta} \varepsilon
\\
& \gtrsim&   \varepsilon.
\end{eqnarray*}
If $|r_{v} | \geq 1$ and $|r_{u}| \leq 1$, then 
\begin{eqnarray*} 
|v-u| &\geq&  |(v-u)\cdot n_{m,i,j,\varepsilon}  | \geq 
| r_{v} \sin\phi_{v} - r_{u} \sin\phi_{u} | \geq | r_{v} \sin\phi_{v} | - |r_{u} \sin\phi_{u}|\\
&\geq& |\sin\phi_{v}| - 8 C_{\eta} \varepsilon
\\
&\geq  &  8 C_{\eta} C_{*} \varepsilon  -  8 C_{\eta}\varepsilon\\
& \gtrsim&  \varepsilon. 
\end{eqnarray*}
If $|r_{v} | \leq 1$ and $|r_{u}|\geq 1$, then $|r_{v} \sin\phi_{v}| \geq 8C_{\eta} C_{*}\varepsilon, \ |r_{v} \cos\phi_{v}| \leq |\cos\phi_{v}|$ and $|u_{n}| \leq \varepsilon + |u| \sin\varepsilon = \varepsilon +( |u_{n}|
+ |u_{\tau}| ) \sin\varepsilon,
$ and $$|u_{n}| \leq \frac{\varepsilon + |u_{\tau}| \sin\varepsilon}{1-\sin\varepsilon} \leq \frac{\varepsilon}{1-\sin\varepsilon}(1+|u_{\tau}|).$$ 

Fix $0< c_{*}\ll1\ll  C_{*}.$ If $C_{*} -c_{*} \geq |u|$, then
\begin{eqnarray*} 
|v-u| &  \geq & |(v-u)\cdot n_{m,i,j,\varepsilon}| \geq |v\cdot n_{m,i,j,\varepsilon} | -| u\cdot n_{m,i,j,\varepsilon} | \\
& \geq    & 
8 C_{\eta} C_{*} \varepsilon - |u| \times 8 C_{\eta} \varepsilon \geq 8C_{\eta} \varepsilon (C_{*} - |u|)
   \\ 
& \geq& 8 C_{\eta}  \varepsilon \times c_{*}.   
\end{eqnarray*}
On the other hand, if $C_{*}-c_{*} \leq |u|$, then 
\begin{eqnarray*} 
|v-u| & \geq &   \big| [u- (u\cdot n_{m,i,j,\varepsilon})n_{m,i,j,\varepsilon} ] -  [v- (v\cdot n_{m,i,j,\varepsilon})n_{m,i,j,\varepsilon} ] \big|\\
& \geq &
|u| |\cos\phi_{u}| - |v| |\cos\phi_{v}|\\
 & \geq & |u|  \sqrt{1-64 (C_{\eta})^{2}\varepsilon^{2} } - |\cos\phi_{v}|\\
 & \geq & (C_{*}-c_{*}) \sqrt{1-64 (C_{\eta})^{2} \varepsilon^{2}} -1\\
 & \gtrsim & 1. 
\end{eqnarray*}
For $|r_{v}|\geq 1$ and $|r_{u}|\geq 1,$ we have
\begin{eqnarray*}
|v-u|&\geq& \min \Big\{ |  \sin \phi_{v} -   \sin\phi_{u}|,  \\
&& \ \ \ \ \ \ \ \ 
|   (  \cos\theta_{v} \cos \phi_{v} -     \cos\theta_{u} \cos \phi_{u}   )
\hat{x}_{1,m,i,j,\varepsilon} + ( \sin \theta_{v} \cos \phi_{v} -
 \sin \theta_{u} \cos \phi_{u}
) \hat{x}_{2,m,i,j,\e}
|
\Big\}\\
&\gtrsim& C_{\eta}(C_{*}-1)  \varepsilon.
\end{eqnarray*} 
Combining all cases, we deduce $(\ref{dist_O2a})  \gtrsim   \varepsilon.$ 

Secondly, we prove $(\ref{dist_O2b})  \gtrsim  \varepsilon.$ The proof is due to
\begin{equation}\notag
\begin{split}
 (\ref{dist_O2b})  & \geq \inf \bigg\{  |x-y| : x \in \bigcap_{ z\in \mathcal{X}_{m,i,j,\varepsilon,C_{*}\varepsilon,\ell} } ( B_{\mathbb{R}^{3}}(z;C_{*}\varepsilon))^{c} , \ 
y \in \bigcup_{z\in \mathcal{X}_{m,i,j, \varepsilon,\frac{C_{*}}{2} \varepsilon,\ell}} B_{\mathbb{R}^{3}} (z; \frac{C_{*}}{2} \varepsilon)
   \bigg\}\\
   & \geq \inf \bigg\{  |x-y| : x \in \bigcap_{ z\in \mathcal{X}_{m,i,j,\varepsilon,\frac{C_{*}}{2}\varepsilon,\ell} } ( B_{\mathbb{R}^{3}}(z;C_{*}\varepsilon))^{c} , \ 
y \in \bigcup_{z\in \mathcal{X}_{m,i,j,\varepsilon, \frac{C_{*}}{2} \varepsilon,\ell  }} B_{\mathbb{R}^{3}} (z; \frac{C_{*}}{2} \varepsilon)
   \bigg\}\\
   & \geq \inf_{z\in  \mathcal{X}_{m,i,j ,\varepsilon, \frac{C_{*}}{2} \varepsilon ,\ell } }
\inf \Big\{  |x-y| : x\in ( B_{\mathbb{R}^{3}}(z;C_{*}\varepsilon))^{c} , \  y \in  B_{\mathbb{R}^{3}} (z; \frac{C_{*}}{2} \varepsilon)\Big\}\\
& \geq \frac{C_{*}}{2} \varepsilon.
\end{split}
\end{equation}

\end{proof}

Recall the standard mollifier $\varphi : \mathbb{R}^{3} \times \mathbb{R}^{3} \rightarrow [0,\infty)$,
\[
\varphi(x,v) : = C \exp \Big( \frac{1}{|x|^{2} + |v|^{2}-1}\Big), \  \text{for} \ \sqrt{|x|^{2} + |v|^{2}} < 1, \ \ \ \ \varphi(x,v) :=0 , \ \text{for} \ \sqrt{|x|^{2} + |v|^{2} }\geq 1,
\]
where the constant $C>0$ is selected so that $\int_{\mathbb{R}^{3} \times \mathbb{R}^{3}} \varphi (x,v)\mathrm{d}v\mathrm{d}x=1.$ 

For each $\varepsilon>0,$ set
\begin{equation}\label{mollifier}
\varphi_\varepsilon(x,v):= (\varepsilon/\tilde{C})^{-6}  \varphi(\frac{\sqrt{|x|^{2} + |v|^{2}}}{  \varepsilon/\tilde{C}}),
\end{equation}
where $\tilde{C}\gg C_{*}\gg 1.$ Clearly $\varphi_{\varepsilon}$ is smooth and bounded and satisfies
\[
\iint_{\mathbb{R}^{3} \times \mathbb{R}^{3}} \varphi_{\varepsilon} (x,v) \mathrm{d}v \mathrm{d}x =1, \ \ \ \mathrm{spt}( \varphi_{\varepsilon}) \subset B_{\mathbb{R}^{3} \times \mathbb{R}^{3}}(0; {\varepsilon}/{\tilde{C}}).
\]

\begin{definition}
We define a smooth cut-off function $\chi_\varepsilon: \bar{\Omega} \times \mathbb{R}^3\rightarrow [0,2]$ as
\begin{equation}\label{cutoff}
\begin{split}
\chi_\varepsilon(x,v) :=& \ \mathbf{1}_{ \bar{\Omega}\times\mathbb{R}^3 \backslash  \mathcal{O}_{\varepsilon, {C}_{*}\varepsilon} }  * \varphi_{\varepsilon} (x,v)\\
=& \ \iint_{\mathbb{R}^{3} \times \mathbb{R}^{3}}
\mathbf{1}_{\bar{\Omega} \times \mathbb{R}^{3} \backslash \mathcal{O}_{\varepsilon, C_{*} \varepsilon}}  (y,u) \varphi_{\varepsilon}(x-y, v-u) \mathrm{d}u \mathrm{d}y
.
\end{split}
\end{equation} 
\end{definition}

The following properties of the cut-off function are crucial for our analysis.

\begin{lemma}\label{cut_off}
There exists $\tilde{C}\gg C_{*} \gg 1$ in (\ref{mollifier}) and (\ref{cutoff}) and $\varepsilon_0 = \varepsilon_0(\Omega)>0$ such that if $0<\varepsilon< \varepsilon_0$ then
\begin{equation}
\mathfrak{S}_{\mathrm{B}} \  \ \subset \ \big\{(x,v) \in  \bar{\Omega} \times \mathbb{R}^3: \chi_\varepsilon(x,v)=0\big\}, \label{S_inc_zero}
\end{equation}
and, for either $\partial=\nabla_x$ or $\partial = \nabla_v$,
\begin{eqnarray} 
\iint_{ {\Omega}\times \mathbb{R}^{3}} [1-\chi_\varepsilon(x,v) ]e^{-\theta|v|^{2}} \mathrm{d}v \mathrm{d}x &   \lesssim_{\Omega } &\varepsilon, \label{chi_W11bound1}\\
 \iint_{\bar{\Omega}\times  \mathbb{R}^{3}} |\partial \chi_\varepsilon(x,v)| e^{-\theta|v|^{2}} \mathrm{d}v\mathrm{d}x &   \lesssim_{\Omega } & 1.\label{chi_W11bound2}
\end{eqnarray} 

\end{lemma} 

\begin{proof}
Firstly we prove (\ref{S_inc_zero}). Let $(x,v) \in \mathfrak{S}_{\mathrm{B}}.$ Due to (\ref{mollifier}) if $|(x,v)-(y,u)| \geq  \varepsilon / \tilde{C}$ then $\varphi_{\varepsilon} (x-y, v-u) =0.$ 
Therefore
\begin{equation}\notag
\begin{split}
(\ref{cutoff}) = \iint_{B_{\mathbb{R}^{6}}((x,v); \varepsilon/\tilde{C})  } \mathbf{1}_{\bar{\Omega} \times \mathbb{R}^{3} \backslash \mathcal{O}_{\varepsilon, C_{*} \varepsilon} } (y,u) \varphi_{\varepsilon} (x-y, v-u) \mathrm{d}y \mathrm{d}u.
\end{split}
\end{equation}
On the other hand, due to (\ref{dist_ep}) with $\varepsilon_{1} = \varepsilon$ and $\tilde{C}\gg C_{*}$, we have $(y,u) \in \mathcal{O}_{\varepsilon, C_{*} \varepsilon}$ and 
\[
\mathbf{1}_{\bar{\Omega} \times \mathbb{R}^{3} \backslash \mathcal{O}_{\varepsilon, C_{*} \varepsilon} } (y,u)\equiv0, \ \text{ on } \  (y,u) \in B_{\mathbb{R}^{6}}((x,v); \varepsilon/\tilde{C}).
\]
Therefore we conclude $\chi_{\varepsilon}(x,v)=0$ and (\ref{S_inc_zero}).

Secondly we deduce (\ref{chi_W11bound1}). We use (\ref{small_bulk}) with $\varepsilon_{1} = \varepsilon$ to have
\begin{equation*}
\begin{split}
&\iint_{\mathbb{R}^{3} \times \mathbb{R}^{3}}
\iint_{\mathbb{R}^{3} \times \mathbb{R}^{3}} 
\big[1-\mathbf{1}_{\bar{\Omega} \times \mathbb{R}^{3} \backslash \mathcal{O}_{\varepsilon, C_{*} \varepsilon}} (y,u)\big] \varphi_{\varepsilon}(x-y, v-u) e^{-\theta |v|^{2}} \mathrm{d}u
 \mathrm{d}y
 \mathrm{d}v \mathrm{d}x\\
& \leq \iint_{\mathbb{R}^{3} \times \mathbb{R}^{3}} \mathbf{1}_{  \mathcal{O}_{\varepsilon, C_{*} \varepsilon}} (y,u)
e^{-\frac{\theta}{2}| u|^{2}}
 \mathrm{d}u \mathrm{d}y  \iint_{\mathbb{R}^{3} \times \mathbb{R}^{3}}
\varphi_{\varepsilon} (x-y, v-u) e^{  \theta |v-u|^{2}}
 \mathrm{d}v \mathrm{d}x
\\
&\leq \ C_{1}  \frac{ \varepsilon}{2}\iint_{B_{\mathbb{R}^{6}}(0; \varepsilon/ \tilde{C})  }
\varphi_{\varepsilon} (x , v )  e^{\theta  {\varepsilon}^{2} / {\tilde{C}}^{2}}
 \mathrm{d}v \mathrm{d}x\\
 &\lesssim \  \varepsilon,
\end{split}
\end{equation*}
where we used
\begin{equation}\label{exponent}
-\theta |v|^{2}=\theta|v-u|^{2}- \theta|v-u|^{2} -\theta |v|^{2}  \leq\theta|v-u|^{2}   -\frac{\theta}{2} |u|^{2}.
\end{equation}

Thirdly we prove (\ref{chi_W11bound2}). Note that  \begin{equation}\notag
\begin{split}
 {|}\partial \varphi_{\varepsilon} (x,v)  {|} = (\varepsilon / \tilde{C})^{-6}  \frac{O(|(x,v)|)}{|(x,v)|}\varphi^{\prime} (\frac{\sqrt{|x|^{2} + |v|^{2}}}{ \varepsilon / \tilde{C}}) = O(\varepsilon^{-6} \tilde{C}^{6}) \mathbf{1}_{B_{\mathbb{R}^{6}} (0; \varepsilon / \tilde{C})  } (x,v),
\end{split}
\end{equation}
and the LHS of (\ref{chi_W11bound2}) equals 
\begin{equation}\notag
\begin{split}
& \iint_{\mathbb{R}^{3} \times \mathbb{R}^{3}}
\iint_{\mathbb{R}^{3} \times \mathbb{R}^{3}} 
 \mathbf{1}_{\bar{\Omega} \times \mathbb{R}^{3} \backslash \mathcal{O}_{\varepsilon, C_{*} \varepsilon}} (y,u) \partial\big[ \varphi_{\varepsilon}(x-y, v-u) {\big]} e^{-\theta |v|^{2}}  \mathrm{d}u
 \mathrm{d}y
 \mathrm{d}v \mathrm{d}x
\\
&= \ O(\varepsilon^{-6} \tilde{C}^{6})\iint_{\mathbb{R}^{3} \times \mathbb{R}^{3}}
 \mathbf{1}_{\bar{\Omega} \times \mathbb{R}^{3} \backslash \mathcal{O}_{\varepsilon, C_{*} \varepsilon}} (y,u) e^{-  \frac{\theta}{2}|u|^{2}}
 \iint_{\mathbb{R}^{3} \times \mathbb{R}^{3}} 
\mathbf{1}_{B_{\mathbb{R}^{6}} (0; \varepsilon/ \tilde{C})  }(x-y,v-u)  
e^{\theta |v-u|^{2}}
\mathrm{d}u \mathrm{d}y \mathrm{d}v \mathrm{d}x\\
%& \ \ \ + O(1) \iint_{\Omega \times \mathbb{R}^{3}} [1- \chi_{\varepsilon}(x,v)] e^{-\frac{\theta}{2} |v|^{2}} \mathrm{d}v \mathrm{d}x\\
& = \  O (\varepsilon^{-6})  \iint_{\mathbb{R}^{3} \times \mathbb{R}^{3}} 
\mathbf{1}_{B_{\mathbb{R}^{6}} (0; \varepsilon/ \tilde{C})  }(x ,v )  
  \mathrm{d}v \mathrm{d}x\\
& =  \ O (\varepsilon^{-6}) O(\varepsilon^{6}) \\
& \lesssim \ 1.
\end{split}
\end{equation}\end{proof}

\begin{lemma}\label{small_boundary_lemma}
With the same constants $\tilde{C}\gg C_{*} \gg 1$ in Lemma \ref{cut_off} and $0< \varepsilon< \varepsilon_{0},$
\begin{equation}\label{S_inc_zero_boundary}
\mathfrak{S}_{\mathrm{B}} \cap [ \partial\Omega \times \mathbb{R}^{3} ]  \ \subset \ \big\{  (x,v) \in \partial\Omega \times \mathbb{R}^{3} :  \chi_{\varepsilon}(x,v)=0 \big\}.
\end{equation}
Moreover if $|(x,v)-(y,u)|\leq \varepsilon/ \tilde{C}$ for $\tilde{C} \gg C_{*}\gg 1$ then 
\begin{equation}\label{small_boundary}
\int_{\partial\Omega} \int_{n(x)\cdot v <0} \mathbf{1}_{  \mathcal{O}_{\varepsilon, C_{*} \varepsilon}}(x-y, v-u) e^{- \theta |v-u|^{2}} 
|n(x-y) \cdot (v-u)|
\mathrm{d}v \mathrm{d}S_{x} \ \lesssim  \ \varepsilon,
\end{equation}
and 
\begin{eqnarray}
\int_{ \gamma_{-}  }[1- \chi_\varepsilon(x,v) ]e^{-\theta|v|^{2}}  \mathrm{d}\gamma &   \lesssim_{\Omega }  & \varepsilon,\label{chi_W11bound_boundary1}\\ 
\int_{     \gamma_{-}     } |\partial \chi_\varepsilon(x,v)| e^{-\theta|v|^{2}} \mathrm{d}\gamma & 
  \lesssim_{\Omega }& 1.\label{chi_W11bound_boundary2}
\end{eqnarray} 
\end{lemma}

The following fact is crucial to prove Lemma \ref{small_boundary_lemma} and especially (\ref{small_boundary}):

\begin{lemma}\label{cone_lemma}
 We fix $m_{0} =1,2,\cdots, M_{\Omega, \delta}$ in (\ref{boundary_decompose}). From (\ref{Um}), we may assume (up to rotations and translations) there exists a $C^{2}-$function $\eta_{m_{0} } : [-\delta, \delta] \times [-\delta, \delta]   \rightarrow \mathbb{R},$ whose graph is the boundary $\mathcal{U}_{m_{0} } \cap \partial\Omega$.

Let $(x_{1},x_{2}) \in \mathcal{A}_{m_{0} } \ \cap \ [- {\delta}/{2}, {\delta}/{2}] \times [-{\delta}/{2}, {\delta}/{2}]$ and $(x_{1},x_{2}) \in \mathcal{R}_{m_{0} ,i_{0} ,j_{0} , \varepsilon, C_{*} \varepsilon}$ for $|i_{0} |,|j_{0} |\leq N_{\varepsilon}$. (see (\ref{lattice}), (\ref{Rij}), and (\ref{A_R}))
 
Suppose $i)$ $|y| \leq \varepsilon/ \tilde{C} $ and
\begin{equation}\label{incl_O}
\big((x_{1}, x_{2}, \eta_{m_{0} }(x_{1},x_{2})) - y,v\big)  \ \in \ \mathcal{O}_{\varepsilon, C_{*} \varepsilon},
\end{equation}
and $ii)$ for large but fixed $s_{*} \gg1$.
\begin{equation}\label{C2_bound}
- 1 \leq n_{m_{0} }(0,0)\cdot \frac{v}{|v|} \leq - s_{*}  {C_{2}}  \sqrt{\varepsilon},  \ \ \ \text{with} \ C_{2}   : =  \sqrt{\frac{8 C_{*}}{3}} \big[1+ || {\eta }_{m_{0}}  ||_{C^{2} (\mathcal{A}_{m _{0}})}\big]^{1/2}.
\end{equation}

Then either $|v|<  \varepsilon^{1/3}$ 
% If $
%x= (x_{1},x_{2}, \eta_{m_{0}}(x_{1},x_{2}))$ and $(x_{1},x_{2}) \in \mathcal{A}_{m_{0}} \cap [- {\delta}/{2}, {\delta}/{2}] \times [-{\delta}/{2}, {\delta}/{2}]$ and 
%\[
%(x_{1},x_{2}) \in \mathcal{R}_{m_{0},i_{0},j_{0}, \varepsilon, C_{*} \varepsilon} ,
%\] 
%for some $|i_{0}|,|j_{0}|\leq N_{\varepsilon}$ where $\mathcal{R}_{m_{0},i_{0},j_{0}, \varepsilon, C_{*} \varepsilon}$ is defined as in (\ref{Rij}) and if $|y| \leq \varepsilon/ {C}_{*}$ and 
%\begin{equation}\label{incl_O}
%\big((x_{1}, x_{2}, \eta_{m_{0}}(x_{1},x_{2})) - y,v\big)  \ \in \ \mathcal{O}_{\varepsilon, C_{*} \varepsilon},
%\end{equation}
%and for large but fixed $s_{*}\gg1$
% \begin{equation}\label{C2_bound}
%- 1 \leq n_{m_{0}}(0,0)\cdot \frac{v}{|v|} \leq - s_{*}  {C_{2}}  \sqrt{\varepsilon},  \ \ \ \text{with} \ C_{2}   : =  \sqrt{\frac{8 C_{*}}{3}} \big[1+ || {\eta_{m_0}}  ||_{C^{2} (\mathcal{A}_{m_{0}})}\big]^{1/2},
%\end{equation} 
or there exists $(i,j) \in [-N_{1} + i_{0}, N_{1} + i_{0}] \times [-N_{1} + j_{0}, N_{1} + j_{0}],$ with
\begin{equation}
N_{1}:=  \big\lfloor \frac{8C_{3}}{\sqrt{\varepsilon}}  \big\rfloor, \ \ 
C_{3}   : =    
\frac{4 C_{*} + 8C_{*} \big[1+  || \eta_{m _{0}} ||_{C^{1} (\mathcal{A}_{m_{0} })} \big]^{1/2} + 2/\tilde{C}  }{C_{2}}
,\label{C3}
\end{equation}
such that
\[
\big((x_{1},x_{2},\eta_{m_{0} }(x_{1},x_{2}))-y,v\big) \ \in \ \bigcup_{0\leq \ell \leq L_{\varepsilon}} \mathcal{O}_{m_{0} ,i,j,\varepsilon, C_{*} \varepsilon, \ell} \ \cap \ \bar{\Omega} \times \{v \in \mathbb{R}^{3} : |v| \geq  \varepsilon^{1/3}\},
\]
and
\begin{equation}\label{C4}
\big|n_{m_{0} }(0,0) \cdot \frac{v}{|v|}\big| \ \leq \ C_{4} \sqrt{\varepsilon} \ \ \ \text{with} \ C_{4} = C_{3} \big[1+ || {\eta_{m_{0} }}  ||_{C^{2} (\mathcal{A}_{m_{0} })}\big] .
\end{equation}

\end{lemma}

Remark that the constant $N_{1}$ in (\ref{C3}) does not depend on  $x,y,v$.

 \begin{proof}[Proof of Lemma \ref{cone_lemma}]
 Without loss of generality ({up to rotations and translations}), we may assume 
 \begin{equation}\label{cone_WLOG}
 (i_{0},j_{0})=(0,0) \ \text{ and }   \ \eta_{m_{0}}(0,0)=0 \ \text{ and } \ \nabla \eta_{m_{0}} (0,0)=0. 
 \end{equation}
 
 Consider the case of $|v|\geq  \varepsilon^{1/3}.$ Since $\big((x_{1}, x_{2}, \eta_{m_{0}} (x_{1},x_{2}))-y,v\big)\in \mathcal{O}_{\varepsilon, C_{*} \varepsilon}$ we use the definition of $\mathcal{O}_{\varepsilon, C_{*} \varepsilon}$ in (\ref{O_ep}) to have 
\begin{equation}\label{conclusion_1}
\text{ either } \ \ \underbrace{|v| < C_{*} \varepsilon}_{(\ref{conclusion_1})-(i)}  \ \ \text{ or } \ \ \underbrace{(x-y,v) \in \bigcup_{m,i,j,\ell} \mathcal{O}_{m,i,j,\varepsilon,C_{*} \varepsilon, \ell}}_{(\ref{conclusion_1})-(ii)}.
\end{equation}
For small $0< \varepsilon \ll 1$, we can exclude the case of $(\ref{conclusion_1})-(i)$ since $|v|>  \varepsilon^{1/3}\gg C_{*} \varepsilon.$

Consider the case of $(\ref{conclusion_1})-(ii)$. In this case, we claim that
\begin{equation}\label{same_m_0}
\big((x_{1},x_{2}, \eta_{m_{0} }(x_{1},x_{2}))-y,v\big) \ \in  \ \bigcup_{i,j,\ell} \mathcal{O}_{m_{0} , i,j,\varepsilon, C_{*} \varepsilon,\ell}. 
\end{equation}
From $(\ref{conclusion_1})-(ii)$ and the definition of $\mathcal{O}_{m_{0}, i,j, \e, C_{*} \e, \ell}$ in (\ref{O_mijl}), there exist $m,i,j,\ell$ such that
\[
\big((x_{1},x_{2}, \eta_{m_{0}}(x_{1},x_{2}))-y,v\big) \ \in  \ \Big[ \bigcup_{p\in \mathcal{X}_{m,i,j,\varepsilon, C_{*} \varepsilon,\ell}} B_{ {\mathbb{R}^{3}} }(p;C_{*} \varepsilon) \Big] \times \Theta_{m,i,j,\varepsilon,C_{*} \varepsilon,\ell}.
\]
In particular, there exists $p \in \mathcal{X}_{m,i,j,\varepsilon, C_{*}\varepsilon,\ell}$ satisfying $$\big|p- \big((x_{1},x_{2}, \eta_{m_{0}}(x_{1},x_{2}))-y \big)\big|< C_{*} \varepsilon.$$

By the definition of $\mathcal{X}_{m,i,j,\varepsilon, C_{*}\varepsilon,\ell}$ in (\ref{cal_X}), 
\[
p= (\bar{p}_{1}, \bar{p}_{2},   {\eta_{m }}  (\bar{p}_{1}, \bar{p}_{2})) 
+ \bar{\tau} \big[ \cos \bar{\theta} \hat{x}_{1,m,i,j,\varepsilon} + \sin\bar{\theta} \hat{x}_{2,m,i,j,\varepsilon}\big] + \bar{s} n_{m,i,j,\varepsilon},
\]
for some 
\begin{eqnarray*}
(\bar{p}_{1}, \bar{p}_{2})&\in&\mathcal{R}_{m,i,j,\varepsilon,C_{*} \varepsilon},\\
\bar{\theta} &\in& (\varepsilon \ell- C_{*} \varepsilon, \varepsilon \ell + C_{* } \varepsilon),\\
\bar{\tau} &\in& [0,t_{\mathbf{f}} ( (\bar{p}_{1} , \bar{p}_{2}, \eta_{m } (\bar{p}_{1}, \bar{p}_{2})), \cos\bar{\theta} \hat{x}_{1,m,i,j,\varepsilon} + \sin\bar{\theta} \hat{x}_{2,m,i,j,\varepsilon}  ) ],\\
\bar{s}  &\in& [-\e,\e]. 
\end{eqnarray*}

By the definition of $t_{\mathbf{f}}$ in (\ref{forwardexit}),
\[
z : = p- \bar{s} n_{m,i,j,\varepsilon} = (\bar{p}_{1}, \bar{p}_{2},  {\eta_{m }}  (\bar{p}_{1}, \bar{p}_{2})) 
+ \bar{\tau} \big[ \cos \bar{\theta} \hat{x}_{1,m,i,j,\varepsilon} + \sin\bar{\theta} \hat{x}_{2,m,i,j,\varepsilon}\big] \in \Omega.
\]
And 
\begin{equation}\label{z-x}
\begin{split}
&|z- \big((x_{1},x_{2}, \eta_{m_{0}}(x_{1},x_{2}))-y \big)|\\
& \leq |z-p| + 
|p- \big((x_{1},x_{2}, \eta_{m_{0}}(x_{1},x_{2}))-y \big)  |\\
& \leq 2C_{*}\varepsilon.
\end{split}
\end{equation}
From (\ref{cone_WLOG}), (\ref{z-x}), and $|y| \leq \varepsilon/ \tilde{C}$, we deduce 
\begin{equation}\notag
\begin{split}
&|z- (0,0, \eta_{m_{0}}(0,0))| \\
&\leq |z- \big((x_{1},x_{2}, \eta_{m_{0}}(x_{1},x_{2}))-y \big)| + | (x_{1},x_{2}, \eta_{m_{0}}(x_{1},x_{2}))  - (0,0, \eta_{m_{0}}(0,0))| +|y|\\
&\leq 2C_{*} \varepsilon +  4 C_{*} \varepsilon (1+ || \eta_{m_{0}}||_{C^{1}(\mathcal{A}_{m_{0}})})  + \varepsilon/ \tilde{C}.
\end{split} 
\end{equation}
 
 Denote $(\bar{z}_{1}, \bar{z}_{2}) = (\bar{p}_{1}, \bar{p}_{2})$. By the definition of $t_{\mathbf{b}}$ and $t_{\mathbf{f}}$ in (\ref{backexit}) and (\ref{forwardexit}) 
\begin{equation}\label{z_b}
x_{\mathbf{b}}(z, \cos \bar{\theta} \hat{x}_{1,m,i,j,\varepsilon} + \sin\bar{\theta} \hat{x}_{2,m,i,j,\varepsilon}+0 n_{m,i,j,\varepsilon} ) = (\bar{z}_{1}, \bar{z}_{2},  {\eta_{m_0}}  (\bar{z}_{1}, \bar{z}_{2})).
\end{equation}

On the other hand, by the definition of $\Theta_{m,i,j,\varepsilon ,C_{*} \varepsilon, \ell}$ in (\ref{Theta}), 
\begin{equation}\label{v_mod_v}
\frac{v}{|v|}= \cos\theta_{v} \cos \phi_{v} \hat{x}_{1,m,i,j,\varepsilon} + \sin\theta_{v} \cos\phi_{v} \hat{x}_{2,m,i,j,\varepsilon} + \sin\phi_{v} n_{m,i,j,\varepsilon}, \    \text{ with } \ |\theta_{v} - \varepsilon \ell| < C_{*} \varepsilon,
\end{equation}
and
\begin{equation}\notag
\begin{split}
|v\cdot n_{m,i,j,\varepsilon}| &< 8 C_{\eta} C_{*} \varepsilon, \   \ \text{for} \ \  \varepsilon^{1/3} \leq |v| \leq 1,\\
\big| \frac{v}{|v|}\cdot n_{m,i,j,\varepsilon}\big| & < 8 C_{\eta} C_{*} \varepsilon,  \ \   \text{for} \ \ 1 \leq|v|.
\end{split}
\end{equation} 

Therefore, for $0< \varepsilon \ll 1$,
\begin{equation}\label{v_n_1}
\big| \frac{v}{|v|}\cdot n_{m,i,j,\varepsilon} \big| =  |\sin\phi_{v}| < \max\big\{    {8C_{\eta} C_{*}}  \varepsilon^{2/3},  8 C_{\eta} C_{*} \varepsilon\big\} \leq  {16 C_{\eta} C_{*}}  \varepsilon^{2/3}.
\end{equation}

Now we estimate as
\begin{equation}\notag
\begin{split}
&n_{m_{0}}(0,0)\cdot (\cos \bar{\theta} \hat{x}_{1,m,i,j,\varepsilon} + \sin \bar{\theta} \hat{x}_{2,m,i,j,\varepsilon} +0 n_{m,i,j,\varepsilon}  )\\
\leq& \ n_{m_{0}}(0,0) \cdot \frac{v}{|v|} 
+  n_{m}(0,0)  
 \cdot \underbrace{\big(\frac{v}{|v|}-  (\cos \bar{\theta} \hat{x}_{1,m,i,j,\varepsilon} + \sin \bar{\theta} \hat{x}_{2,m,i,j,\varepsilon} +0 n_{m,i,j,\varepsilon}  )\big)}_{\mathbf{(a)}}.
\end{split}
\end{equation}

We use (\ref{v_mod_v}), (\ref{v_n_1}), and $\bar{\theta} \in (\varepsilon \ell - C_{*} \varepsilon, \varepsilon \ell + C_{*} \varepsilon)$ to conclude that, for $0 < \varepsilon\ll 1$,
\begin{eqnarray*}
\mathbf{(a)} &\leq& 2 \big\{ | \cos \theta_{v} - \cos \bar{\theta}| + |\cos\theta_{v} ||\cos \phi_{v} -1| 
+ | \sin \theta_{v} - \sin \bar{\theta}| + |\sin \theta_{v}||\cos \phi_{v} -1| 
  + |\sin \phi_{v}|  \big\}\\
  &\leq&  
  2\{
  4C_{*} \e + 16 C_{\eta} C_{*} \e^{2/3} + 2(16)^{2} C_{\eta}^{2} C_{*}^{2} \e^{4/3} 
  \}
  \\
  &\leq&  {200 C_{\eta} C_{*}} \varepsilon^{2/3}.
\end{eqnarray*}
%On the other hand, due to (\ref{v_n_1}) 
%\[
%\mathbf{(b)}  \ \leq \ \big| n_{m,i,j,\varepsilon} \cdot \frac{v}{|v|}\big|  <  {16 C_{\eta } C_{*}}  \varepsilon^{2/3}.
%\]
Finally from (\ref{C2_bound}), for $0< \varepsilon \ll 1,$
\begin{equation}\label{n_0_u}
\begin{split}
 -1 & \  \leq \ n_{m_{0}}(0,0)\cdot (\cos \bar{\theta} \hat{x}_{1,m,i,j,\varepsilon} + \sin \bar{\theta} \hat{x}_{2,m,i,j,\varepsilon} +0 n_{m,i,j,\varepsilon}  )\\
 &  \ \leq  \ 
 - s_{*} \times  C_{2} \sqrt{\varepsilon} +  {400 C_{\eta} C_{*}}  \varepsilon^{2/3}\\
 &  \ \leq  \ 
  -\frac{s_{*}C_{2}}{2}\sqrt{\varepsilon}.
 \end{split}
\end{equation}

Now we are ready to prove the first claim (\ref{same_m_0}). Denote
\[
\hat{u} := \cos\bar{\theta} \hat{x}_{1,m,i,j,\varepsilon} + \sin \bar{\theta} \hat{x}_{2,m,i,j,\varepsilon}.
\]
Recall that $|z|\leq (2C_{*} + 4C_{*} [1+ || \eta_{m_{0}}||_{C^{1} (\mathcal{A}_{m_{0}})}] + 1/ \tilde{C} )\varepsilon$ and $z\in \Omega.$ Therefore for $0< \varepsilon \ll 1$ the function $\eta_{m_{0}}$ is defined around $(z_{1},z_{2})$ and $z_{3} >\eta_{m_{0}}(z_{1},z_{2}).$

We define, for $|\tau| \ll1$, 
\begin{equation}
\Phi(\tau) =   z_{3} - \hat{u}_{3} \tau - \eta_{m_{0}}(z_{1}-\hat{u}_{1} \tau, z_{2} - \hat{u}_{2} \tau).
\end{equation}
Clearly $\Phi(0) >0.$ Expanding $\Phi(\tau)$ in $\tau$, from $ - \hat{u}_{3}=n_{m_{0}}(0,0)\cdot (\cos \bar{\theta} \hat{x}_{1,m,i,j,\varepsilon} + \sin \bar{\theta} \hat{x}_{2,m,i,j,\varepsilon}   )$, and (\ref{n_0_u}), we have
\begin{equation}\notag
\begin{split}
\Phi(\tau) & \leq - \hat{u}_{3} \tau + |z_{3} | + |\eta_{m_{0}}(z_{1}-\hat{u}_{1} \tau, z_{2} -\hat{u}_{2} \tau)|\\
&\leq 
-s_{*} \times\frac{C_{2}}{2}\sqrt{\varepsilon} \tau \\
& \ \ + (2C_{*} + 4C_{*} [1+ || \eta_{m_{0}}||_{C^{1} (\mathcal{A}_{m_{0}})}] + 1/ \tilde{C}  )\varepsilon\\
& \  \ + || \eta_{m_{0}} ||_{C^{2}(\mathcal{A}_{m_{0}})}(2C_{*} + 4C_{*} [1+ || \eta_{m_{0}}||_{C^{1} (\mathcal{A}_{m_{0}})}] + 1/ \tilde{C}  )^{2}\varepsilon^{2}\\
& \ \ + || \eta_{m_{0}} ||_{C^{2}(\mathcal{A}_{m_{0}})} |\tau|^{2},
\end{split}
\end{equation}
where we have used
\begin{equation}\notag
\begin{split}
& \eta_{m_{0}}(z_{1}  -\hat{u}_{1} \tau,  z_{2} - \hat{u}_{2} \tau)\\
   =&  \ \eta_{m_{0}}(z_{1}   , z_{2}   )  +\int_{0}^{\tau} \frac{d}{ds} \eta_{m_{0}} (z_{1}  - \hat{u}_{1}s, z_{2}  - \hat{u}_{2} s) \mathrm{d}s\\\
 =&   \ \eta_{m_{0}}(z_{1}  , z_{2}   ) -    (\hat{u}_{1},\hat{u}_{2}) \cdot \nabla \eta_{m_{0}} (z_{1} , z_{2}   )   {\tau} + \int_{0}^{\tau} \int_{0}^{s}  \frac{d}{d s_{1}^{ \ 2}} \eta_{m} ( z_{1}  -\hat{u}_{1} s_{1}, z_{2} - \hat{u}_{2} s_{1}   )  \mathrm{d}s_{1}  \mathrm{d}s  \\
\leq & \  || \eta_{m_{0}} ||_{C^{2}(\mathcal{A}_{m_{0}})} \frac{|z|^{2}}{2} 
- (\hat{u}_{1}, \hat{u}_{2})\cdot \nabla \eta_{m_{0}}(0,0) | \tau|
+ || \eta_{m_{0}}||_{C^{2}(\mathcal{A}_{m_{0}})} |z||\tau|
+ || \eta_{m_{0}} ||_{C^{2} (\mathcal{A}_{m_{0}})} \frac{|\tau|^{2}}{2}\\
\leq & \   {|| \eta_{m_{0}}||_{C^{2}(\mathcal{A}_{m_{0}})}}   \Big( |z|^{2}+|\tau|^{2}\Big) 
.
\end{split}
\end{equation}
Now we plug $\tau= \frac{1}{s_{*}}\times C_{3}\sqrt{\varepsilon}$ with the constant $C_{3}$ in (\ref{C3}) to have, for $s_{*}\gg 1$ and $0< \varepsilon \ll1,$
\begin{equation}\notag
\begin{split}
\Phi(\tau) & \leq - \Big[\frac{C_{2} C_{3}}{2} - \big(2C_{*} + 4C_{*} [1+ || \eta_{m_{0}}||_{C^{1} (\mathcal{A}_{m_{0}})}] + 1/ \tilde{C}  \big)  -  \frac{ || \eta_{m_{0}} ||_{C^{2}(\mathcal{A}_{m_{0}})}  C_{3}^{2}  }{(s_{*})^{2}} \Big] \varepsilon + O(\varepsilon^{2})\\
& <0.
\end{split}
\end{equation}
%\begin{equation}\label{C_3}
%C_{3}:= \frac{ 4 C_{*}  + 8 C_{*} [1+ ||\eta_{m_{0}} ||_{C^{1}(\mathcal{A}_{m_{0}})}  ]  + 2/ C_{*} }{C_{2}}.
%\end{equation}
By the mean value theorem, there exists at least one $\tau\in (0, C_{3}\sqrt{\varepsilon}]$ satisfying $\Phi(\tau)=0.$ We choose the smallest one of them and denote it as $\tau_{0} \in(0, C_{3}\sqrt{\varepsilon}].$ By this definition and (\ref{z_b}), for $0< \varepsilon \ll 1,$
\begin{equation}\notag
\begin{split}
x_{\mathbf{b}}(z, \hat{u}) &= x_{\mathbf{b}}(z, \cos\bar{\theta} \hat{x}_{1,m,i,j,\varepsilon} + \sin\bar{\theta} \hat{x}_{2,m,i,j,\varepsilon})\\
&= z- \tau_{0} \hat{u}\\
&= \big( z_{1} -\tau_{0}  \hat{u}_{1} , z_{2} - \tau_{0} \hat{u}_{2}  , z_{3} -\tau_{0} \hat{u}_{3}  \big).
\end{split}
\end{equation}
 Therefore, $x_{\mathbf{b}}(z, \hat{u})\in \partial\Omega \cap \mathcal{U}_{m_{0}}$ and this proves (\ref{same_m_0}).

For $0< \varepsilon \ll 1$
\begin{equation}\notag
\begin{split}
|\big( z_{1} -\tau_{0}  \hat{u}_{1} , z_{2} - \tau_{0} \hat{u}_{2}    \big)|\leq &
\big(  2C_{*} + 4C_{*} (1+ || \eta_{m_{0}} ||_{C^{1} (\mathcal{A}_{m})} +1/\tilde{C}  ) \big)\varepsilon + C_{3}\sqrt{\varepsilon}\\
\leq &2{C_{3}}  \sqrt{\varepsilon}.
\end{split}
\end{equation}
Moreover, 
\[
\big( z_{1} -\tau_{0}  \hat{u}_{1} , z_{2} - \tau_{0} \hat{u}_{2}    \big) \in \mathcal{R}_{m_{0}, i,j, \varepsilon, C_{*}\varepsilon},
\]
for
\[
|i-i_{0}|, |j-j_{0}| \ \leq \   (2{C_{3}} \sqrt{\varepsilon}) / \varepsilon \ \leq \  2{C_{3}}   \frac{1}{\sqrt{\varepsilon}} \leq N_{1}.
\]

We only need to prove (\ref{C4}). From (\ref{v_n_1}) and (\ref{C3})
\begin{equation}\notag
\begin{split}
\big|  n_{m_{0}}(0,0)\cdot\frac{v}{|v|}\big| & \leq \  \big|  n_{m_{0},i,j,\varepsilon, C_{*} \varepsilon}  \cdot\frac{v}{|v|}\big|  +  \big|  (n_{m_{0}}(0,0) - n_{m_{0},i,j,\varepsilon, C_{*} \varepsilon})\cdot\frac{v}{|v|}\big| \\
& \leq  \ {16 C_{\eta} C_{*}}  \varepsilon^{2/3} + || n_{m_{0}} ||_{C^{1} (\mathcal{A}_{m_{0}})} | N_{1} \varepsilon + C_{*} \varepsilon |\\
& \leq \    {16 C_{\eta} C_{*}} \varepsilon^{2/3} + || n_{m_{0}} ||_{C^{1} (\mathcal{A}_{m_{0}})} \big\{ 2C_{3} \sqrt{\varepsilon} + C_{* } \varepsilon  \big\}\\
& \leq \ 10C_{3  }  (1+ || \eta_{m_{0}} ||_{C^{2} (\mathcal{A}_{m_{0}})} )  \sqrt \varepsilon\\
&\ \leq \ \ C_{4} \sqrt{\varepsilon}, 
\end{split}
\end{equation}
and (\ref{C4}) follows.
 \end{proof}

\begin{proof}[\textbf{Proof of Lemma \ref{small_boundary_lemma}}]

The first statement (\ref{S_inc_zero_boundary}) is clear from (\ref{S_inc_zero}). Once we assume (\ref{small_boundary}) then it is easy to prove (\ref{chi_W11bound_boundary1}), (\ref{chi_W11bound_boundary2}):

Firstly we prove (\ref{chi_W11bound_boundary1}). Due to properties of the standard mollifier (\ref{mollifier}), we obtain
\begin{equation}\notag
\begin{split}
&\iint_{x\in \partial\Omega, n(x)\cdot v <0}
\big[1- \chi_{\varepsilon} (x,v)\big]e^{-\theta |v|^{2}}
  |n(x)\cdot v|\mathrm{d}S_{x} \mathrm{d}v\\
=&  \  \iint_{x\in \partial\Omega, n(x)\cdot v <0}
\iint_{\mathbb{R}^{3} \times \mathbb{R}^{3}}
\big[ 1- \mathbf{1}_{\bar{\Omega} \times \mathbb{R}^{3} \backslash \mathcal{O}_{\varepsilon, C_{*} \varepsilon}} (x-y,v-u)  \big] \varphi_{\varepsilon} (y,u)
e^{-\theta |v|^{2}}
  \mathrm{d}u \mathrm{d}y |n(x)\cdot v|\mathrm{d}S_{x} \mathrm{d}v\\
  \leq & \ \iint_{\mathbb{R}^{3} \times \mathbb{R}^{3}} \varphi_{\varepsilon} ( y,  u) e^{\theta | u|^{2}}  \mathrm{d}u \mathrm{d}y
  \iint_{x\in\partial\Omega, n(x)\cdot v<0}
  \mathbf{1}_{\mathcal{O}_{\varepsilon, C_{*} \varepsilon}} (x-y,v-u) e^{-\frac{\theta }{2}|v-u|^{2} }  |n(x)\cdot v| \mathrm{d}S_{x} \mathrm{d}v\\
  = & \ 
  \iint_{B_{\mathbb{R}^{6}}(0; \varepsilon / \tilde{C})   } \varphi_{\varepsilon} ( y,  u) e^{ \frac{\theta}{2} | u|^{2}}  \mathrm{d}u \mathrm{d}y\\
  & \times
  \iint_{x\in\partial\Omega, n(x)\cdot v<0}
  \mathbf{1}_{\mathcal{O}_{\varepsilon, C_{*} \varepsilon}} (x-y,v-u) e^{-\frac{\theta }{4}|v-u|^{2} } e^{ -\frac{\theta}{2}|v|^{2}}  |n(x)\cdot v| \mathrm{d}S_{x} \mathrm{d}v,
  \end{split}
  \end{equation}
  where we used
  \begin{eqnarray*}
  -\theta |v|^{2} &\leq& -\frac{\theta}{2}|v|^{2} -\Big(\frac{\theta}{2}|v|^{2}  - \frac{\theta}{4} |v-u|^{2}\Big) - \frac{\theta}{4} |v-u|^{2}\\
  &\leq& -\frac{\theta}{2}|v|^{2} -
  \Big(
   \frac{\theta}{2} |v|^{2} -   \frac{\theta}{2} |v|^{2} -  \frac{\theta}{2} |u|^{2}
  \Big) 
  - \frac{\theta}{4} |v-u|^{2}\\
  &\leq& -\frac{\theta}{2}|v|^{2} + \frac{\theta}{2}|u|^{2}   - \frac{\theta}{4} |v-u|^{2}.
  \end{eqnarray*}
  Since $|y|+ |u| \leq \varepsilon / \tilde{C}$ and $n(x)\cdot v<0$, we have
  \begin{equation}\notag
  \begin{split}
  n(x)\cdot v &=  n(x)\cdot v - n(x-y)\cdot (v-u) + n(x-y)\cdot (v-u)\\
  & =   n(x-y)\cdot (v-u) +[n(x)- n(x-y)] \cdot v + n(x-y)\cdot u\\
  &=    n(x-y)\cdot (v-u) + O(\frac{\varepsilon}{\tilde{C}})( 1+|v|). 
  \end{split}
  \end{equation}
  Therefore, we use (\ref{small_boundary}) to bound (\ref{chi_W11bound_boundary1}) further as
  \begin{equation}\notag
  \begin{split}
(\ref{chi_W11bound_boundary1})  \leq \ \ &   \iint_{B_{\mathbb{R}^{6}}(0; \varepsilon / \tilde{C})   } \varphi_{\varepsilon} ( y,  u) e^{ \frac{\theta}{2} | u|^{2}}  \mathrm{d}u \mathrm{d}y \\
  & \times 
  \iint_{x\in\partial\Omega, n(x)\cdot v<0}
  \mathbf{1}_{\mathcal{O}_{\varepsilon, C_{*} \varepsilon}} (x-y,v-u) e^{-\frac{\theta }{4}|v-u|^{2} } e^{ -\frac{\theta}{2}|v|^{2}}  |n(x-y)\cdot (v-u)| \mathrm{d}S_{x} \mathrm{d}v\\
   + & \ O(\frac{\varepsilon}{\tilde{C}}) e^{ \frac{\theta\varepsilon^{2}}{2\tilde{C}^{2}}} \times \mathrm{m}_{3} (\partial\Omega) \times \int_{\mathbb{R}^{3}} (1+|v|) e^{-\frac{\theta}{2}|v|^{2}} \mathrm{d}v\\
   \lesssim_{\Omega} & \ \ \varepsilon \times e^{\frac{\theta \varepsilon^{2}}{2 (\tilde{C})^{2}} }\\
    \lesssim_{\Omega} &  \ \  \varepsilon.
  \end{split}
  \end{equation}
 
Secondly we prove (\ref{chi_W11bound_boundary2}). Following the same proof of (\ref{chi_W11bound_boundary1}) , we deduce
\begin{equation}\notag
\begin{split}
&\Big|\iint_{x\in \partial\Omega, n(x)\cdot v <0}
\partial  \chi_{\varepsilon} (x,v) e^{-\theta |v|^{2}}
  |n(x)\cdot v|\mathrm{d}S_{x} \mathrm{d}v\Big|\\
  =& \ \Big|\iint_{x\in \partial\Omega, n(x)\cdot v <0}
\partial \big[ \chi_{\varepsilon} (x,v) -1\big]e^{-\theta |v|^{2}}
  |n(x)\cdot v|\mathrm{d}S_{x} \mathrm{d}v\Big| \\
=& \ 
\Big|  \iint_{x\in \partial\Omega, n(x)\cdot v <0}
\partial \Big[
\iint_{\mathbb{R}^{3} \times \mathbb{R}^{3}}
  \mathbf{1}_{  \mathcal{O}_{\varepsilon, C_{*} \varepsilon}} (y,u)  \varphi_{\varepsilon} (x-y,v-u)
  \mathrm{d}u \mathrm{d}y 
\Big]
e^{-\theta |v|^{2}} |n(x)\cdot v| \mathrm{d}S_{x} \mathrm{d}v
\Big|
\\  
\leq& \  \Big| \iint_{x\in \partial\Omega, n(x)\cdot v <0}
\iint_{\mathbb{R}^{3} \times \mathbb{R}^{3}}
  \mathbf{1}_{  \mathcal{O}_{\varepsilon, C_{*} \varepsilon}} (x-y,v-u)   |\partial \varphi_{\varepsilon} (y,u)|
  \mathrm{d}u \mathrm{d}y  |n(x)\cdot v| e^{-\theta |v|^{2}}\mathrm{d}S_{x} \mathrm{d}v\Big|\\
  =& \ 
  \iint_{B_{\mathbb{R}^{6}}(0; \varepsilon / \tilde{C})   } | \partial \varphi_{\varepsilon} ( y,  u) | e^{ \frac{\theta}{2} | u|^{2}}  \mathrm{d}u \mathrm{d}y \\
  & \times 
  \iint_{x\in\partial\Omega, n(x)\cdot v<0}
  \mathbf{1}_{\mathcal{O}_{\varepsilon, C_{*} \varepsilon}} (x-y,v-u) e^{-\frac{\theta }{4}|v-u|^{2} } e^{ -\frac{\theta}{2}|v|^{2}}  |n(x)\cdot v| \mathrm{d}S_{x} \mathrm{d}v\\
\lesssim& \  \sup_{(y,u) \in B_{\mathbb{R}^{6}}(0; \varepsilon / \tilde{C})  }
  \iint_{x\in\partial\Omega, n(x)\cdot v<0}
  \mathbf{1}_{\mathcal{O}_{\varepsilon, C_{*} \varepsilon}} (x-y,v-u) e^{-\frac{\theta }{4}|v-u|^{2} } e^{ -\frac{\theta}{2}|v|^{2}}  (1+|v|) \mathrm{d}S_{x} \mathrm{d}v\\
  +  & O(\frac{1}{\varepsilon}) \sup_{(y,u) \in B_{\mathbb{R}^{6}}(0; \varepsilon / \tilde{C})  }
    \iint_{x\in\partial\Omega, n(x)\cdot v<0}
  \mathbf{1}_{\mathcal{O}_{\varepsilon, C_{*} \varepsilon}} (x-y,v-u) e^{-\frac{\theta }{4}|v-u|^{2} } e^{ -\frac{\theta}{2}|v|^{2}}   |n(x-y)\cdot (v-u)|\mathrm{d}S_{x} \mathrm{d}v\\
 \lesssim  & \ \ 1. 
 \end{split}
\end{equation}

 \vspace{8pt}
 
 \noindent\textit{Proof of (\ref{small_boundary}).} Let $|(y,u)|\leq \varepsilon/ \tilde{C}.$ We use (\ref{boundary_decompose}) to decompose
 \begin{equation}\notag
 \begin{split}
 (\ref{small_boundary})  & \leq \sum_{m=1}^{M_{\Omega, \delta}} \int_{\mathcal{U}_{m} \cap \partial\Omega} \int_{n_{m}(x)\cdot v<0} 
 \mathbf{1}_{\mathcal{O}_{\varepsilon, C_{*} \varepsilon}} (x-y, v-u) e^{-\theta |v-u|^{2}}  e^{-\frac{\theta}{2}|v|^{2}} |n_{m}(x-y)\cdot (v-u)|
 \mathrm{d}v \mathrm{d}S_{x}\\
 & \leq M_{\Omega, \delta} \times  \sup_{m}
\int_{\mathcal{U}_{m} \cap \partial\Omega} \int_{n_{m}(x)\cdot v<0} 
 \mathbf{1}_{\mathcal{O}_{\varepsilon, C_{*} \varepsilon}} (x-y, v-u) e^{-\theta |v-u|^{2}} e^{-\frac{\theta}{2}|v|^{2}} |n_{m}(x-y)\cdot (v-u)|
 \mathrm{d}v \mathrm{d}S_{x}\\
 & \lesssim_{\Omega} \frac{1}{\delta^{2}}  \sup_{m}
\int_{\mathcal{U}_{m} \cap \partial\Omega} \int_{n_{m}(x)\cdot v<0} 
 \mathbf{1}_{\mathcal{O}_{\varepsilon, C_{*} \varepsilon}} (x-y, v-u) e^{-\theta |v-u|^{2}} e^{-\frac{\theta}{2}|v|^{2}} |n_{m}(x-y)\cdot (v-u)|
 \mathrm{d}v \mathrm{d}S_{x}.
\end{split}
 \end{equation}
For fixed $m=1,2,\cdots, M_{\Omega, \delta},$ we use (\ref{Um}) and (\ref{A_R}) again to decompose
\begin{equation}\notag
\begin{split}
&\int_{\mathcal{U}_{m} \cap \partial\Omega} \int_{n_{m}(x)\cdot v<0} 
 \mathbf{1}_{\mathcal{O}_{\varepsilon, C_{*} \varepsilon}} (x-y, v-u) e^{-\theta |v-u|^{2}} e^{-\frac{\theta}{2}|v|^{2}} |n_{m}(x-y)\cdot (v-u)|
 \mathrm{d}v \mathrm{d}S_{x}\\
 = &\int_{\mathcal{A}_{m}} \int_{n_{m}(x_{1},x_{2})\cdot v <0} 
 \mathbf{1}_{\mathcal{O}_{\varepsilon, C_{*} \varepsilon}} (x_{1}-y_{1},x_{2}-y_{2}, \eta_{m}(x_{1}, x_{2}) -y_{3}, v-u) \\
 & \ \ \ \ \ \ \ \ \ \ \ \ \ \ \ \ \ \ \ \ \ \ \ \ 
\times e^{-\theta |v-u|^{2}} e^{-\frac{\theta}{2}|v|^{2}} |n_{m}(x-y)\cdot (v-u)|
 \mathrm{d}v
 \sqrt{1+ | \nabla \eta_{m} (x_{1}, x_{2})|^{2}} \mathrm{d}x_{1} \mathrm{d}x_{2}\\
\leq & \sum_{- N_{\varepsilon} \leq i,j \leq N_{\varepsilon}}
\int_{\mathcal{R}_{m,i,j,\varepsilon, C_{*} \varepsilon}} \int_{n_{m}(x_{1},x_{2})\cdot v <0} 
 \mathbf{1}_{\mathcal{O}_{\varepsilon, C_{*} \varepsilon}} (x_{1}-y_{1},x_{2}-y_{2}, \eta_{m}(x_{1}, x_{2}) -y_{3}, v-u) \\
 &  \ \ \ \ \ \ \ \ \ \ \ \ \ \ \ \ \ \ \ \ \ \ \ \ 
\times e^{-\theta |v-u|^{2}} e^{-\frac{\theta}{2}|v|^{2}} |n_{m}(x-y)\cdot (v-u)|
 \mathrm{d}v
 \sqrt{1+ | \nabla \eta_{m} (x_{1}, x_{2})|^{2}} \mathrm{d}x_{1} \mathrm{d}x_{2}\\
\lesssim & \  \frac{\delta^{2}}{\varepsilon^{2}}\sup_{{- N_{\varepsilon} \leq i,j \leq N_{\varepsilon}}}\int_{\mathcal{R}_{m,i,j,\varepsilon, C_{*} \varepsilon}} \int_{n_{m}(x_{1},x_{2})\cdot v <0} 
 \mathbf{1}_{\mathcal{O}_{\varepsilon, C_{*} \varepsilon}} (x_{1}-y_{1},x_{2}-y_{2}, \eta_{m}(x_{1}, x_{2}) -y_{3}, v-u) \\
 &   \ \ \ \ \ \ \ \ \ \ \ \ \ \ \ \ \ \ \ \ \ \ \ \ 
\times e^{-\theta |v-u|^{2}} e^{-\frac{\theta}{2}|v|^{2}} |n_{m}(x-y)\cdot (v-u)|
 \mathrm{d}v
 \sqrt{1+ | \nabla \eta_{m} (x_{1}, x_{2})|^{2}} \mathrm{d}x_{1} \mathrm{d}x_{2},\\
\end{split}
\end{equation} 
where $n_{m}(x_{1},x_{2}) = \frac{1}{\sqrt{1+ |\partial_{1} \eta_{m}(x_{1},x_{2})|^{2} + |\partial_{1} \eta_{m}(x_{1},x_{2})|^{2}  }}
\big( \partial_{1} \eta_{m}(x_{1},x_{2}), \partial_{2} \eta_{m}(x_{1}, x_{2}), -1  \big).$

We fix $i,j$. Without loss of generality (up to rotations and translations), we may assume
\[
c_{m,i,j,\varepsilon} = (0,0), \ \ \partial_{1} \eta_{m}(0,0) = 0 = \partial_{2} \eta_{m}(0,0), \ \ n_{m,i,j,\varepsilon} = (0,0,-1).
\]
We claim
\begin{equation}\label{claim_small_boundary}
\begin{split}
&\int_{[-C_{*}\varepsilon, C_{*} \varepsilon]^{2}} 
\int_{n_{m}(x_{1},x_{2})\cdot (v+u)<0} \mathbf{1}_{\mathcal{O}_{\varepsilon, C_{*} \varepsilon}}(x_{1}-y_{1}, x_{2} - y_{2}, \eta_{m}(x_{1},x_{2}) -y_{3}, v )
 \\
 &
\ \ \ \ \ \ \ \ \ \ \ \ \ \ \ \ \ \ \ \ \ \ \ \ 
\times e^{-\theta |v |^{2}} e^{-\frac{\theta}{2}|v+u|^{2}} |n_{m}(x-y)\cdot v|
 \mathrm{d}v
 \sqrt{1+ | \nabla \eta_{m} (x_{1}, x_{2})|^{2}} \mathrm{d}x_{1} \mathrm{d}x_{2} 
 \\
 &\lesssim \varepsilon^{3}.
\end{split}
\end{equation}
Once we prove (\ref{claim_small_boundary}), due to the above estimates for the decomposition, we conclude (\ref{small_boundary}) directly.

For $(x_{1},x_{2}) \in [-C_{*} \varepsilon, C_{*} \varepsilon]^{2}, \ |(y,u)| < \varepsilon / \tilde{C}$, and $n_{m}(x_{1},x_{2})\cdot (v+u) <0,$ we deduce 
\begin{equation}\label{bdry_step1}
\begin{split}
n_{m,i,j,\varepsilon} \cdot  v& = n_{m}(0,0)\cdot v\\
&= n_{m}(x_{1},x_{2}) \cdot (v+u) + \big[ n_{m}(0,0)\cdot v -n_{m}(x_{1},x_{2}) \cdot ( v+u)\big]\\
&<0 + | n_{m}(x_{1},x_{2}) | |u| +  |n_{m}(0,0) - n_{m}(x_{1},x_{2})||v|\\
&\leq   {\varepsilon}/{\tilde{C}} + 2 C_{*} \varepsilon || \eta_{m}||_{C^{2}  ([-C_{*} \varepsilon, C_{*} \varepsilon]^{2}) }  |v|\\
& \leq C_{5}(1+|v|)\varepsilon,
\end{split}
\end{equation}
where $C_{5}= \max\big\{ 1/\tilde{C}, 2C_{*}|| \eta_{m}||_{C^{2}  ([-C_{*} \varepsilon, C_{*} \varepsilon]^{2}) }   \big\}$. Therefore 
\[
(\ref{claim_small_boundary}) \ \leq  \ \int_{[-C_{*} \varepsilon, C_{*} \varepsilon]^{2}}
\int_{n_{m,i,j,\varepsilon} \cdot v< C_{5} (1+|v|)\varepsilon} \cdots.
\]

According to Lemma \ref{cone_lemma}, we decompose 
\begin{equation}\label{claim_small_boundary_split}
\begin{split}
&\int_{[-C_{*}\varepsilon, C_{*} \varepsilon]^{2}} 
\int_{n_{m}(0,0) \cdot v \leq C_{5} (1+|v|)\varepsilon} \mathbf{1}_{\mathcal{O}_{\varepsilon, C_{*} \varepsilon}}(x_{1}-y_{1}, x_{2} - y_{2}, \eta_{m}(x_{1},x_{2}) -y_{3}, v )
 \\
 &
\ \ \ \ \ \ \ \ \ \ \ \ \ \ \ \ \ \ \ \ \ \ \ \ 
\times e^{-\theta |v |^{2}} e^{-\frac{\theta}{2}|v+u|^{2}} |n_{m}(x-y)\cdot  v |
 \mathrm{d}v
 \sqrt{1+ | \nabla \eta_{m} (x_{1}, x_{2})|^{2}} \mathrm{d}x_{1} \mathrm{d}x_{2} \\
 =& \underbrace{ \int_{[-C_{*}\varepsilon, C_{*} \varepsilon]^{2}} 
\int_{\big\{ { -s_{*}C_{2} \sqrt{\varepsilon}  \leq n_{m}(0,0)\cdot \frac{v}{|v|}\leq  C_{5}\frac{1+|v|}{|v|}\varepsilon } \big\} } \cdots  }_{\mathbf{(I)}}\\
&
 +   \underbrace{
\int_{[-C_{*}\varepsilon, C_{*} \varepsilon]^{2}} 
\int_{  \left\{   -1 \leq n_{m}(0,0)\cdot \frac{v}{|v|} \leq -s_{*}C_{2} \sqrt{\varepsilon}   \right\}}  \cdots 
}_{\mathbf{(II)}}
.
\end{split}
\end{equation}
First we consider $\mathbf{(I)}$. If $-s_{*}C_{2} \sqrt{\varepsilon}\leq n_{m}(0,0)\cdot \frac{v}{|v|} \leq 0$ then $0 \leq v_{3} = - n_{m}(0,0)\cdot v \leq s_* C_2|v| \sqrt{\varepsilon}$ and for $0< \varepsilon \ll 1$ 
\[
0  \ \leq  \ v_{3}  \ \leq \ 2 s_* C_2 \sqrt{|v_{1}|^{2} + |v_{2}|^{2}} \sqrt{\varepsilon}.
\]
Moreover
\begin{eqnarray*}
|n_{m}(x-y)\cdot v|& \leq& |n_{m}(0,0)\cdot v | +||  n_{m} ||_{C^{1} ( [-C_{*}\varepsilon, C_{*} \varepsilon]^{2} )} (C_{*} + {1}/{\tilde{C}}) |v|\varepsilon\\
& \leq& s_* C_2|v| \sqrt{\varepsilon}+4|| \eta_{m} ||_{C^{2} ( [-C_{*}\varepsilon, C_{*} \varepsilon]^{2} )} (C_{*} +  {1}/{\tilde{C}})|v| \varepsilon.
\end{eqnarray*}
If $n_{m}(0,0) \cdot \frac{v}{|v|} \leq C_{5} \frac{1+ |v|}{|v|}\varepsilon$ then for $0 < \varepsilon \ll 1$
\[
|v_{3} |=|n_{m}(0,0) \cdot v| \leq 2C_{5} (1+ \sqrt{ |v_{1}|^{2}+|v_{2}|^{2} }) \varepsilon.
\]
Therefore,
\begin{equation}\label{mathbf_I}
\begin{split}
\mathbf{(I)} &= \int_{[-C_{*}\varepsilon, C_{*} \varepsilon]^{2}}  
\int_{0 \leq v_{3} \leq 2s_* C_2 \sqrt{ |v_{1}|^{2} + |v_{2}|^{2}  } \sqrt{\varepsilon}  } 
 e^{-\theta|v|^{2}} \big\{ s_* C_2|v| \sqrt{\varepsilon}+4|| \eta_{m} ||_{C^{2} ( [-C_{*}\varepsilon, C_{*} \varepsilon]^{2} )} (C_{*} +  {1}/{\tilde{C}})|v| \varepsilon   \big\}
\\  & \ + 
\int_{[-C_{*}\varepsilon, C_{*} \varepsilon]^{2}} 
 \int_{ |v_{3}| \leq 2 C_{5} (1+ \sqrt{|v_{1}|^{2}  + |v_{2}|^{2} }) \varepsilon }  e^{-\theta |v|^{2}} \\
&  \lesssim  \mathrm{m}_{2} ( [-C_{*}\varepsilon, C_{*} \varepsilon]^{2}  ) \times \sqrt{\varepsilon} \iint_{\mathbb{R}^{2}} \mathrm{d}v_{1} \mathrm{d}v_{2} \ e^{- \frac{\theta}{2}|v_{1}|^{2}}e^{- \frac{\theta}{2}|v_{2}|^{2}} \int_{0}^{2s_* C_2 \sqrt{|v_{1}|^{2} + |v_{2}|^{2} } \sqrt{\varepsilon}} \mathrm{d}v_{3}\\
&  \ +  \mathrm{m}_{2}  ( [-C_{*}\varepsilon, C_{*} \varepsilon]^{2}  ) \times\iint_{\mathbb{R}^{2}} \mathrm{d}v_{1} \mathrm{d}v_{2} \ e^{-  {\theta} |v_{1}|^{2}}e^{-  {\theta} |v_{2}|^{2}} \int_{0}^{2C_{5} (1+\sqrt{|v_{1}|^{2} + |v_{2}|^{2} }) {\varepsilon}} \mathrm{d}v_{3}\\
& \lesssim \varepsilon^{3}.
\end{split}
\end{equation}

We decompose $\mathbf{(II)}$, according to Lemma \ref{cone_lemma}:
\begin{equation}\notag
\mathbf{(II)} = \int_{[-C_{*}\varepsilon, C_{*} \varepsilon]^{2}} \int_{|v|< \varepsilon^{1/3}} + \int_{[-C_{*}\varepsilon, C_{*} \varepsilon]^{2}} 
\int_{  \left\{   -1 \leq n_{m}(0,0)\cdot \frac{v}{|v|} \leq -s_* C_2 \sqrt{\varepsilon}      \ \text{and}  \ |v|\geq \varepsilon^{1/3} \right\}  } .
\end{equation}
The first term is clear bounded by $O(1) \varepsilon^{3}.$ For the second term we use (\ref{C4}) to have 
\begin{eqnarray*}
 &&\big\{   -1 \leq n_{m}(0,0)\cdot \frac{v}{|v|} \leq -s_* C_2 \sqrt{\varepsilon}      \ \text{and}  \ |v|\geq \varepsilon^{1/3} \big\}\\
 &   \subset &  \big\{   |n_{m}(0,0)\cdot \frac{v}{|v|} |\leq C_{4} \sqrt{\varepsilon}      \ \text{and}  \ |v|\geq \varepsilon^{1/3} \big\}.
\end{eqnarray*}
Therefore we follow the same proof for $(\ref{mathbf_I})$ to obtain 
\begin{eqnarray} 
\mathbf{(II)} & \lesssim& \varepsilon^{3} +  \int_{[-C_{*}\varepsilon, C_{*} \varepsilon]^{2}}    \int_{ | v_{3} |\leq 2C_{4} \sqrt{ |v_{1}|^{2} + |v_{2}|^{2}  } \sqrt{\varepsilon}  } \notag \\
&&   \ \ \ \ \ \ \ \ \ \ \ \ \ \ \ \ \ \ \ \ \ \ \ \times 
 e^{-\theta|v|^{2}} \big\{ C_{4}|v| \sqrt{\varepsilon}+4|| \eta_{m} ||_{C^{2} ( [-C_{*}\varepsilon, C_{*} \varepsilon]^{2} )} (C_{*} +  {1}/{\tilde{C}})|v| \varepsilon   \big\} \notag\\
& \lesssim& \varepsilon^{3}. \label{mathbf_II} 
\end{eqnarray}
 We conclude (\ref{claim_small_boundary_split}) from (\ref{mathbf_I}) and (\ref{mathbf_II}).  
\end{proof}

\section{Linear and Nonlinear Estimates}

We consider $f$ solving (\ref{transport}) and satisfying the in-flow boundary condition
\begin{equation}\label{inflow}
f(t,x,v)|_{\gamma_-} = g(t,x,v).
\end{equation}

Let $\{\tau_{1}(x), \tau_{2}(x)\}$ be bases of the tangent space at $x\in \partial\Omega$ (therefore $\{\tau_{1}(x), \tau_{2}(x), n(x)\}$ is orthonormal bases of $\mathbb{R}^{3}).$ Denote $\partial_{\tau_{i}}$ to be the (tangential) $\tau_{i}-$directional derivative and $\partial_{n}$ to be the normal derivative.

\begin{lemma}\label{inflow_linear}
Assume {$\mathcal{U}$ is an open subset of $\mathbb{R}^{3} \times \mathbb{R}^{3}$} such that $\mathfrak{S}_{\mathrm{B}} \subset \mathcal{U}$. Assume 
\begin{equation}\label{vanishing}
f_{0}(x,v) \equiv 0, \  g(t,x,v) \equiv 0, \  H(t,x,v) \equiv 0, \  \ \text{for   } \ \ (t,x,v)\in [0,T]\times\{\mathcal{U}\cap (\bar{\Omega}\times\R^3)\}.  
\end{equation}
{Assume further that} for $0<\theta< \frac{1}{4},$
\[
e^{\theta|v|^{2}} f_0 \in L^\infty(\Omega\times\mathbb{R}^3), \   e^{\theta|v|^{2}} g \in L^\infty([0,T]\times\gamma_-), \ e^{\theta|v|^{2}} H \in L^\infty([0,T]\times\Omega\times\mathbb{R}^3),
\]
and 
\begin{equation}\notag
\begin{split}
& \nabla_x f_0 , \ \nabla_v f_0  \in \ L^1(\Omega\times\mathbb{R}^3),\\
 &\partial_{\tau_{i}} g, \ \frac{1}{n(x)\cdot v} \Big\{-\partial_t g -\sum_{i} (v\cdot \tau_i) \partial_{\tau_i}g -\nu g + H\Big\}, \   \nabla_v g,
 e^{-\theta|v|^{2}} \nabla_{x} \nu, e^{-\theta|v|^{2}}\nabla_{v} \nu
   \ \in \ L^1([0,T]\times\gamma_-),\\
 &\nabla_x H, \ \nabla_v H, \ {e^{-\theta|v|^{2}}\nabla_{x} \nu, \ e^{-\theta|v|^{2}}\nabla_{v} \nu} \ \in \ L^1([0,T]\times \Omega\times\mathbb{R}^3).
\end{split}
\end{equation}
%{\color{blue}I THINK WE HAVE NOT DEFINED THE ORTHONORMAL BASIS AT THE BOUNDARY $(n,\tau_1,\tau_2)(x)$. I DO NOT KNOW IF YOU PREFER TO DEFINE IT HERE OR IN THE INTRODUCTION.}\\
Then there exists a unique solution $f$ to the transport equation (\ref{transport}) with in-flow {boundary condition} (\ref{inflow}) such that $e^{\theta |v|^{2}}f \in C^{0} ([0,T] \times \bar{\Omega} \times \mathbb{R}^{3})$ and $\nabla_x f, \ \nabla_v f \in C^0([0,T];L^1(\Omega\times\mathbb{R}^3))$ and the traces satisfy
\begin{equation}\notag
\begin{split}
 &\nabla_ xf  = \nabla_ x g, \ \ \nabla_v f  = \nabla_v g , \ \ \text{on} \ \gamma_-,\\
&\nabla_x f(0,x,v) = \nabla_x f_0, \ \ \nabla_v f(0,x,v) = \nabla_v f_0 , \ \ \text{in} \ \Omega\times\mathbb{R}^3,
\end{split}
\end{equation}
where $\nabla_x g$  {is defined} by 
\begin{equation}\notag\label{gboundary}
\nabla_{x} g   = \sum_{i=1,2} \tau_{i} \partial_{\tau_{i}} g +  \frac{ n}{n\cdot v}  \big\{-\partial_t g -\sum_{i} (v\cdot \tau_i) \partial_{\tau_i}g -\nu g + H \big\}.
\end{equation}
Moreover
{ 
\begin{eqnarray}
%&& ||\partial _{t}f(t)|| +\int_{0}^{t}|\partial _{t}f|_{\gamma
%_{+} }+\int_{0}^{t}||\nu \partial _{t}f||
%=||\partial _{t}f_{0}|| +\int_{0}^{t}|\partial _{t}g|_{\gamma
%_{-} } + \int_{0}^{t}\iint_{\Omega \times \mathbb{R}^{3}} \mathrm{sgn}(f_t)\partial
%_{t}H  , \label{global_t}   \\
&& ||\nabla _{x}f(t)||_1 +\int_{0}^{t}|\nabla _{x}f|_{\gamma
_{+},1 }+\int_{0}^{t}||\nu  \nabla _{x}f||_1 \nonumber\\
&&  \ \ \ \ \ \ \ \ \ \ \ \ \ \   =||\nabla _{x}f_{0}||_1 +\int_{0}^{t}|\nabla _{x}g|_{\gamma
_{-},1 } + \int_{0}^{t}\iint_{\Omega \times \mathbb{R}^{3}}  \mathrm{sgn}(\nabla_x f) \big\{\nabla
_{x}H - \nabla_{x} \nu f   \big\}, \ \ \label{global_x}   \\
&& ||\nabla _{v}f(t)||_1 +\int_{0}^{t}|\nabla _{v}f|_{\gamma
_{+},1 } + \int_{0}^{t}||\nu  \nabla _{v}f||_1   \notag
\\
&& \ \ \ \ \ \ \ \ \ \ \ \ \ \  =||\nabla _{v}f_{0}||_1 +\int_{0}^{t}|\nabla _{v}g|_{\gamma
_{-},1 } + \int_{0}^{t}\iint_{\Omega \times \mathbb{R}^{3}}\mathrm{sgn}(\nabla_v f)\big\{\nabla
_{v}H-\nabla _{x}f-\nabla _{v}\nu f\big\} .\notag \\
 \label{global_v}
\end{eqnarray}}
\end{lemma}

\begin{proof}
We use the Duhamel formula of $f$:
\begin{equation}\label{duhamel_linear}
\begin{split}
f(t,x,v) =& \ \mathbf{1}_{\{t<t_{\mathbf{b}} (x,v)\}} e^{-\int^{t}_{0}\nu(t-\tau, x-\tau v,v) \mathrm{d}\tau }f_0(x-tv,v)\\
&  + \mathbf{1}_{\{t>t_{\mathbf{b}} (x,v) \}} e^{- \int^{t_{\mathbf{b}}(x,v)}_{0}\nu(t-\tau, x-\tau v,v)  \mathrm{d}\tau } g(t- t_{\mathbf{b}} (x,v), x_{\mathbf{b}} (x,v),v)\\
&  + \int_0^{\min\{t,t_{\mathbf{b}} (x,v)\}} e^{- \int_{0}^{s} \nu(t-\tau, x-\tau v,v) \mathrm{d}\tau}H(t-s,x-sv,v) \mathrm{d}s.
\end{split}
\end{equation}

Following Proposition 1 of \cite{GKTT}, we have, on $\{t\neq t_{\mathbf{b}}\}$

\begin{equation}\label{compute_f_x}
\begin{split}
%\partial _{t}f(t,x,v) \1_{\{t\neq \tb\}} &=& -\1_{\{t< t_B\}}e^{-t\nu(v)}[\nu f_{0} + v\cdot\nabla_{x}f_{0}-H_{|t=0}]( x-tv,v)
%  +  \1_{\{t>t_B\}}e^{-\tb\nu}\partial_{t}g(t-\tb,x_{\mathbf{b}},v)\label{compute_f_t}\\
%&&+  \int_{0}^{\min(t,\tb)} e^{-s\nu}\partial_{t}H(t-s,x-vs,v) \ud s ,\notag\\
&\nabla _{x}f(t,x,v)\1_{\{t\neq \tb\}}\\
 & =  \1_{\{t< \tb\}}e^{-\int^{t}_{0}\nu(t-\tau, x-\tau v,v) \mathrm{d}\tau}\left\{\nabla _{x}f_{0}( x-tv,v) -\left(\int^{t}_{0} \nabla_x\nu (t-\tau, x-\tau v,v) \mathrm{d}\tau\right) f_{0}( x-tv,v)\right\} \\
& \ + \1_{\{t>t_B\}}e^{-\int^{\tb}_{0}\nu(t-\tau, x-\tau v,v) \mathrm{d}\tau} \Big\{ \sum_{i=1}^2
 \tau_i \partial_{\tau_i} g - \frac{n(\xb)}{v\cdot n(x_B)} \big\{\partial_t g +  \sum_{i=1}^2 (v\cdot \tau_i) \partial_{\tau_i}g +\nu g -H \big\}
\Big\}(t-\tb,\xb,v) \\
&\ -  \1_{\{t>t_B\}} e^{-\int^{\tb}_{0}\nu(t-\tau, x-\tau v,v) \mathrm{d}\tau} \left(\int^{\tb}_{0}\nabla_x\nu (t-\tau, x-\tau v,v) \mathrm{d}\tau\right)g(t-\tb,\xb,v)\\
& \ + \int_{0}^{\min(t,t_B)} e^{-\int^{s}_{0}\nu(t-\tau, x-\tau v,v) \mathrm{d}\tau} \nabla_{x}H(t-s,x-vs,v)\ud s \\ &-  \int_{0}^{\min(t,t_B)} e^{-\int^{s}_{0}\nu(t-\tau, x-\tau v,v) \mathrm{d}\tau} \left(\int^{s}_{0}\nabla_x\nu (s-\tau, x-\tau v,v) \mathrm{d}\tau\right)H(t-s,x-vs,v) \ud s ,
\end{split}
\end{equation}
\begin{equation}\label{compute_f_v}
\begin{split}
& \nabla _{v}f(t,x,v) \1_{\{t\neq \tb\}} \\
 & =   \1_{\{t< \tb\}}e^{-\int^{t}_{0}\nu(t-\tau, x-\tau v,v) \mathrm{d}\tau}[-t\nabla_{x}f_{0}  + \nabla_{v}f_{0}]( x-tv,v) \\
 &\ - \1_{\{t< \tb\}}e^{-\int^{t}_{0}\nu(t-\tau, x-\tau v,v) \mathrm{d}\tau} \int^{t}_{0} \left\{-\tau\nabla_x\nu +\nabla_v\nu\right\}(t-\tau, x-\tau v,v) \mathrm{d}\tau f_{0}( x-tv,v) \\
& \ -
\1_{\{t>\tb\}} \tb e^{-\int^{\tb}_{0}\nu(t-\tau, x-\tau v,v) \mathrm{d}\tau} \Big\{ \sum_{i=1}^2
 \tau_i \partial_{\tau_i} g - \frac{n(\xb)}{v\cdot n(\xb)} \big\{\partial_t g +  \sum_{i=1}^2 (v\cdot \tau_i) \partial_{\tau_i}g +\nu g -H \big\}
\Big\}(t-\tb,\xb,v)
 \\
& \ + \1_{\{t>\tb\}}e^{-\int^{\tb}_{0}\nu(t-\tau, x-\tau v,v) \mathrm{d}\tau}\Big\{  \nabla_v g(t-\tb,\xb,v) \Big\}\\
&\
-  \1_{\{t>\tb\}}e^{-\int^{\tb}_{0}\nu(t-\tau, x-\tau v,v) \mathrm{d}\tau}\Big\{ \int^{\tb}_{0}\left\{-\tau \nabla_x\nu+\nabla_v\nu \right\}(t-\tau, x-\tau v,v) \mathrm{d}\tau\Big\} g(t-\tb,\xb,v)
 \\
& \ + \:\int_{0}^{\min(t,\tb)} e^{-\int^{s}_{0}\nu(t-\tau, x-\tau v,v) \mathrm{d}\tau}\{\nabla_{v}H -s\nabla_{x}H  \}(t-s,x-vs,v) \ud s \\
&\ -\:\int_{0}^{\min(t,\tb)} e^{-\int^{s}_{0}\nu(t-\tau, x-\tau v,v) \mathrm{d}\tau}\{\int^{s}_{0}\left\{-\tau \nabla_x\nu + \nabla_v\nu\right\}(t-\tau, x-\tau v,v)\mathrm{d}\tau\}H(t-s,x-vs,v) \ud s
\end{split}
\end{equation}

 Therefore, we have{ 
\begin{equation}
\begin{split}
%\|f(t) \1_{\{t\neq t_{\mathbf{b}}\}}\|_1   \leq& \  \|e^{\theta |v|^2}f_{0}\|_\infty  +  \int_{0}^{t}  |g|_{\gamma_{-},1}     + \int_{0}^{t} \|H\|_1   ,\label{estimate_f}\\
%\|\partial_{t}f(t)\1_{\{t\neq t_{\mathbf{b}}\}}\|_1   \leq& \ \|v\cdot  \nabla_{x}f_{0} + \nu f_{0} - H(0,\cdot,\cdot)\|_1
%+   \int_{0}^{t} |\partial_{t}g| _{\gamma_{-},1}+ \int_{0}^{t} \|\partial_{t}H\|_1    ,\\
\|\nabla_{x}f(t)\1_{\{t\neq t_{\mathbf{b}}\}}\|_1  \lesssim& \ \|\nabla_{x}f_{0}\|_1 
+ t \{|| e^{\theta |v|^{2}} f_{0} ||_{\infty}+|| e^{\theta |v|^{2}} g ||_{\infty}  \}
 \\
 &  + \int_{0}^{t}  \Big|
  \sum_{i=1}^2
 \tau_i \partial_{\tau_i} g - \frac{n }{v\cdot n } \big\{\partial_t g +  \sum_{i=1}^2 (v\cdot \tau_i) \partial_{\tau_i}g +\nu g -H \big\}
   \Big|_{\gamma_-,1}\\
   &
+   \int_{0}^{t} \|\nabla_{x}H(s) \|_1 
+ \int_{0}^{t} s   ||e^{\theta|v|^{2}}H(s)||_{\infty} 
\end{split}
\end{equation}

\begin{equation}\notag
\begin{split}
\|\nabla_{v}f(t)\1_{\{t\neq t_{\mathbf{b}}\}}\|_1  \lesssim& \ t \|\nabla_{x}f_{0}\|_1  + \|\nabla_{v}f_{0}\|_1  +  
 t   
 \|e^{\theta |v|^2}f_{0}\|_\infty     \\
+ & t  \int_{0}^{t}   \Big|
\Big\{ \sum_{i=1}^2
 \tau_i \partial_{\tau_i} g - \frac{n }{v\cdot n } \big\{\partial_t g +  \sum_{i=1}^2 (v\cdot \tau_i) \partial_{\tau_i}g +\nu g -H \big\}
\Big\}
  \Big|_{\gamma_{-},1}    \\
+ &  \int_{0}^{t}  |\nabla_{v}g|_{\gamma_{-},1}  +  t^2 \sup_{0\le s\le t}|e^{\theta |v|^2}g(s)|_{\gamma_{-},\infty}\\
   +&   \int_{0}^{t} \|\nabla_{x}H\|_1
 +    \int_{0}^{t} \|\nabla_{v}H\|_1   + C \int_{0}^{t} \|e^{\theta |v|^2}H\|_\infty    .
\end{split}
\end{equation}}
From our assumption, $f_0, \ g,$ and $H$ have compact supports and the RHS are bounded. Therefore
\[
\partial f \mathbf{1}_{\{t\neq t_{\mathbf{b}}\}} = [\partial_t f \mathbf{1}_{\{t\neq t_{\mathbf{b}}\}}, \ \nabla_x f \mathbf{1}_{\{t\neq t_{\mathbf{b}}\}}, \ \nabla_v f \mathbf{1}_{\{t\neq t_{\mathbf{b}}\}}] \ {\in \ L^\infty ([0,T]; L^1(\Omega\times\mathbb{R}^3))}.
\]

Since $\partial f \equiv 0$ around $\{t=t_{\mathbf{b}}\}$ clearly $\partial f \mathbf{1}_{\{t\neq t_{\mathbf{b}}\}}$ is the distributional derivative of $f$.{ Therefore $\nabla_x f$ and $\nabla_v f$ lie in $L^\infty([0,T];L^1(\Omega\times\R^3))$; this allows us to apply Lemma \ref{le:ukai} to compute the traces on the incoming boundary in $L^1([0,T];L^1(\gamma_-,\ud \gamma))$ (by taking limits of the flow along the characteristics: see the proof of Proposition 1 in \cite{GKTT} for details). Then, by Green's identity (Lemma \ref{Green}) we know that $\nabla_x f$ and $\nabla_v f$ lie in $C^0([0,T];L^1(\Omega\times\R^3))$ and we get \eqref{global_x} and \eqref{global_v}.

}

\end{proof}

Now we are ready to prove the main theorem.

\begin{proof}[\textbf{Proof of Theorem 1}]
For $f_{0} \in BV(\Omega\times \mathbb{R}^{3})$ and $|| e^{\theta|v|^{2}}f_{0}||_{\infty}<\infty$ we choose $f_{0}^{\varepsilon}\in BV(\Omega\times\mathbb{R}^{3}) \cap C^{\infty}(\Omega\times\mathbb{R}^{3})$ satisfying $|| e^{\theta|v|^{2} } [ f_{0}^{\varepsilon}- f_{0} ] ||_{\infty}\rightarrow 0$ and $|| \nabla_{x,v} f^{\varepsilon}_{0}||_{1} \rightarrow ||  f_{0}||_{\tilde{BV}}$.  

Consider the sequence {$f^{\varepsilon, m}$ defined by $f^{\varepsilon,0}=$ and for all $m\ge0$,} 
\begin{equation}\label{ep_m}
\begin{split}
\partial_{t} f^{\varepsilon, m+1} + v\cdot \nabla_{x} f^{\varepsilon, m+1} +
\nu(\sqrt{\mu} f^{\varepsilon, m}) f^{\varepsilon, m+1}  = \chi_{\varepsilon}\Gamma_{\mathrm{gain}} (f^{\varepsilon, m}, f^{\varepsilon,m}),& \ \ \text{in} \ \ \Omega \times \mathbb{R}^{3},\\
 f^{\varepsilon, m} (0,x,v)  = \chi_{\varepsilon} f_{0}^{\varepsilon} (x,v) ,& \ \ \text{in} \ \ \Omega \times \mathbb{R}^{3},\\
f^{\varepsilon, m+1}(t,x,v)  = \chi_{\varepsilon}(x,v) c_{\mu } \sqrt{\mu(v)} \int_{n(x)\cdot u>0} f^{\varepsilon,m} (t,x,u) \sqrt{\mu} \{ n\cdot u\} \mathrm{d}u,&  \ \ \text{on} \ \ \gamma_{-},
\end{split}
\end{equation}
where $\chi_{\varepsilon}$ is {defined in} (\ref{cutoff}). Clearly for {a fixed} $0<\varepsilon \ll 1,$ {$(f^{\varepsilon,m})_m$ is Cauchy for the norm $\sup_{0 \leq t \leq T}||e^{\theta|v|^{2}} \cdot||_{\infty}$ for $0 < \theta < \frac{1}{4}$ and some $0<T\ll 1$} (Lemma 6 of \cite{GKTT}). Therefore $f^{\varepsilon, m} \rightarrow f^{\varepsilon}$ {for the norm} $\sup_{0 \leq t \leq T}||e^{\theta|v|^{2}} \cdot||_{\infty} $ and $f^{\varepsilon}$ satisfies (\ref{ep_m}) {with $f^{\varepsilon, m+1}$ and $f^{\varepsilon, m}$ replaced by $f^{\varepsilon}$}.
%\begin{equation}\notag
%\begin{split}
%\partial_{t} f^{\varepsilon} + v\cdot \nabla_{x} f^{\varepsilon } +
%\nu(\sqrt{\mu} f^{\varepsilon }) f^{\varepsilon }  = \chi_{\varepsilon}\Gamma_{\mathrm{gain}} (f^{\varepsilon }, f^{\varepsilon }),& \ \ \text{in} \ \ \Omega \times \mathbb{R}^{3},\\
% f^{\varepsilon } (0,x,v)  = \chi_{\varepsilon} f_{0} (x,v) ,& \ \ \text{in} \ \ \Omega \times \mathbb{R}^{3},\\
%f^{\varepsilon }(t,x,v)  = \chi_{\varepsilon}(x,v) c_{\mu } \sqrt{\mu(v)} \int_{n(x)\cdot u>0} f (t,x,u) \sqrt{\mu} \{ n\cdot u\} \mathrm{d}u,&  \ \ \text{on} \ \ \gamma_{-}.\\
%\end{split}
%\end{equation}
Since $|\chi_{\varepsilon}|\leq 1$ for $0< \varepsilon \ll1 $, $\sup_{0 \leq t \leq T}|| e^{\theta|v|^{2}}f^{\varepsilon}(t)||_{\infty}$ is uniformly {bounded in $\varepsilon$} for $0<\varepsilon \ll 1$ and $0 < T \ll 1.$ Therefore $f^{\varepsilon}\rightarrow f$ weak$-*$ up to {a subsequence} and the limiting function $f$ solves the Boltzmann equation with the diffuse boundary condition in the sense of {distributions}.
%{\color{blue} THIS LAST SENTENCE IS NOT CLEAR TO ME: I THINK THAT BECAUSE OF THE NONLINEARITY WE NEED A STRONGER CONVERGENCE TO CHECK THAT $f$ IS THE SOLUTION OF THE BOLTZMANN EQUATION. MAYBE WE SHOULD DO IT AT THE END ONCE WE KNOW THAT THE CONVERGENCE HOLDS IN $L^1$?}\\ 

Now we consider the derivatives of {the solution $f^{\varepsilon, m}$ of (\ref{ep_m})}. Due to the smooth approximation $f_{0}^{\varepsilon}$ of the initial datum $f_{0}$ and the cut-off $\chi_{\varepsilon}$, {$f^{\varepsilon,m}$ is smooth} by Lemma \ref{inflow_linear}. We take derivatives $\partial \in \{\partial_{x}, \partial_{v}\}$ to have
\begin{equation}\notag
\begin{split}
&\big[\partial_{t} + v\cdot \nabla_{x} + \nu( \sqrt{\mu} f^{\varepsilon,m})\big] \partial f^{\varepsilon, m+1}\\
& =  - \partial v \cdot \nabla_{x} f^{\varepsilon, m+1} - \nu (\partial [\sqrt{\mu} f^{\varepsilon,m}]) f^{\varepsilon, m+1}
 +O(1) e^{-\theta |v|^{2}}\partial \nu || e^{\theta|v|^{2}} f^{\e, m}||_{\infty} || e^{\theta|v|^{2}} f^{\e, m+1}||_{\infty}
 \\
& \ \ \ \ \ \ +\partial \chi_{\varepsilon} \Gamma_{\mathrm{gain}}(f^{\varepsilon,m}, f^{\varepsilon,m}) + \chi_{\varepsilon} \partial [  \Gamma_{\mathrm{gain}} ( f^{\varepsilon,m} ,   f^{\varepsilon,m}) ] ,\\
&\partial f^{ \varepsilon,m+1} (0,x,v) =   \partial \chi_{\varepsilon} f_{0}^{\varepsilon} (x,v) + \chi_{\varepsilon} \partial f_{0}^{\varepsilon} (x,v), 
\end{split}
\end{equation} 
and, {for all $(x,v)\in\gamma_{-},$}
\begin{equation}\notag
\begin{split} |\partial f^{\varepsilon,m+1} (t,x,v)| \lesssim& \ \sqrt{\mu(v)} \left(1+ \frac{\langle
v\rangle }{|n(x)\cdot v|}\right) \int_{n(x)\cdot u>0} |\partial f^{\varepsilon,m }  (t,x,u)|
\mu(u)^{\frac{1}{4}} \{n(x)\cdot u\} \mathrm{d}u\\
& + \frac{\langle
v\rangle^{\kappa } e^{-C_{\theta} |v|^{2}}}{|n(x)\cdot v|} || e^{\theta |v|^{2}} f_0||_\infty 
 + |\partial \chi_{\varepsilon} (x,v) \sqrt{\mu (v)}| P( || e^{\theta |v|^{2}} f_0||_\infty ),
\end{split}
\end{equation} for some polynomial $P$. 
%{\color{blue} I THINK THERE IS A TERM MISSING ABOVE. THE MISSING TERM COMES FROM INITIAL CONDITION IN THE VELOCITY DERIVATIVE.}
Then by Proposition \ref{inflow_linear} and $\sqrt{\mu}f^{\varepsilon,m} \geq 0,$
\begin{equation}\label{energy_est}
\begin{split}
& || \partial f^{\varepsilon, m+1}(t) ||_{1} + \int_{0}^{t} |\partial f^{\varepsilon, m+1}(s) |_{\gamma_{+},1}\\
 \lesssim & \ 
|| e^{-\theta^{\prime} |v|^{2}}\partial \chi_{\varepsilon} ||_{1} || e^{\theta^{\prime}|v|^{2}} f_{0} ||_{\infty} + || \partial f_{0}^{\varepsilon}  ||_{1} + \int_{0}^{t} |\partial f^{\varepsilon, m+1}(s) |_{\gamma_{-}}
  \\
 &+\int_{0}^{t} || \partial f^{\varepsilon, m+1}(s) ||_{1} \mathrm{d}s    +  P ( || e^{\theta |v|^{2}} f^{\varepsilon,m} ||_{\infty})\\
 & \ \ \ \ \ \times \Big\{ t+  t   \iint_{\Omega\times \mathbb{R}^{3}} 
  e^{-C_{\theta }|v|^{2}}|\partial \chi_{\varepsilon}  |      + \int_{0}^{t} || \partial f^{\varepsilon, m }(s) ||_{1} \mathrm{d}s  \Big\},
\end{split}
\end{equation}
where we have used Lemma 5 of \cite{GKTT} {: for any function $g=g(x,v)$}
\begin{eqnarray*}
&&|| \partial  \Gamma_{\mathrm{gain}}(g , g )  ||_{1}\\
 &\lesssim  &
|| e^{\theta|v|^{2}} g||_{\infty} \Big\{ |\partial x| || \nabla_{x} g ||_{1} + |\partial v| || \nabla_{v} g ||_{1}   \Big\} + \langle v\rangle^{\kappa} e^{-\theta|v|^{2}} |\partial v| || e^{\theta|v|^{2}} g||_{\infty}^{2},\\
&&\iint_{\Omega \times \mathbb{R}^{3}} | \nu(  \partial [\sqrt{\mu} g]  )   g | \mathrm{d}v \mathrm{d}x \\& \lesssim &
|| e^{\theta |v|^{2}} g  ||_{\infty} \int_{\Omega} 
\int_{\mathbb{R}^{3}}
\int_{\mathbb{R}^{3}} e^{- \frac{\theta}{4} |v-u|^{2}} |\partial g  (u)| \mathrm{d}u
\mathrm{d}v
\mathrm{d}x  \lesssim  || e^{\theta |v|^{2}} g  ||_{\infty}  || \partial g  ||_{1}.
\end{eqnarray*}
%{\color{blue} I THINK WE ALSO USE A $L^\infty$ ESTIMATE FOR $\Gamma_{gain}$. WE SHOULD CITE THIS ONE. ACTUALLY I AM NOT SURE OF WHICH $L^\infty$ ESTIMATE WE USE.}
%ASSUME
%\begin{equation}\notag
%\begin{split}
% \iint_{\Omega\times \mathbb{R}^{3}} | \partial \chi_{\varepsilon} (x,v) | e^{-\theta |v|^{2}}\mathrm{d}x \mathrm{d}v \lesssim 1, \ \ \iint_{\gamma_{-}} | \partial \chi_{\varepsilon} (x,v) | e^{-\theta |v|^{2}} |n(x)\cdot v|\mathrm{d}S_x \mathrm{d}v \lesssim 1.\\
%\end{split}
%\end{equation}
%NEED LEMMA
%\begin{equation}\notag
%e^{-\theta |v|^{2}} \nu(\sqrt{\mu} g_{1}) (v) \lesssim \int_{\mathbb{R}^{3}} |v-u|^{\kappa} e^{-\theta |v|^{2}} e^{- \frac{|u|^{2}}{4}} |g_{1} (u)| \mathrm{d}u
%\lesssim \int_{\mathbb{R}^{3}} e^{- \frac{\theta}{4} |v-u|^{2}} |g_{1} (u)| \mathrm{d}u,
%\end{equation}
%where we have used $|v|^{2} + |u|^{2} = |v|^{2} + |-u|^{2} \gtrsim |v-u|^{2}.$ Using this,
%\begin{equation}\notag
%\begin{split}
%\iint_{\Omega \times \mathbb{R}^{3}} | \nu( \sqrt{\mu} g_{1}  )| |g_{2}| \mathrm{d}v \mathrm{d}x &\lesssim 
%|| e^{\theta |v|^{2}} g_{2} ||_{\infty} \int_{\Omega} 
%\int_{\mathbb{R}^{3}}
%\int_{\mathbb{R}^{3}} e^{- \frac{\theta}{4} |v-u|^{2}} |g_{1} (u)| \mathrm{d}u
%\mathrm{d}v
%\mathrm{d}x\\
%& \lesssim || e^{\theta |v|^{2}} g_{2} ||_{\infty}  || g_{1} ||_{1}  \int_{v} e^{-\frac{\theta}{4}|v-u|^{2}} \lesssim  || e^{\theta |v|^{2}} g_{2} ||_{\infty}  || g_{1} ||_{1}.
%\end{split}
%\end{equation}
 
Applying Lemma \ref{cut_off} and Lemma \ref{small_boundary_lemma} to (\ref{energy_est}), we obtain
\begin{equation}\label{energy_est_1}
\begin{split}
& || \partial f^{\varepsilon, m+1}(t) ||_{1} + \int_{0}^{t} |\partial f^{\varepsilon, m+1}(s) |_{\gamma_{+},1}\\
 \lesssim & \ 
 || e^{\theta^{\prime}|v|^{2}} f_{0} ||_{\infty} + ||   f_{0}  ||_{BV} + \int_{0}^{t} |\partial f^{\varepsilon, m+1}(s) |_{\gamma_{-},1}\\
& +  t[1+  P ( || e^{\theta^{\prime} |v|^{2}} f_{0} ||_{\infty})]\sup_{0 \leq s \leq t}|| \partial f^{\varepsilon, m+1}(s) ||_{1} \mathrm{d}s    +  tP ( || e^{\theta^{\prime} |v|^{2}} f_{0}||_{\infty}).
\end{split}
\end{equation}
On the other hand, we apply {Proposition \ref{double_iteration} and Lemma \ref{le:ukai}} to bound
 \begin{equation}\label{boundary_nonlin}
 \begin{split}
  \int_{0}^{t} |\partial f^{\varepsilon, m+1} |_{\gamma_{-},1}  
&  \lesssim   O(\delta) \int_{0}^{t} |\partial f^{\varepsilon,m-1 } |_{\gamma_{+},1}  + C_{ \delta}  \{ || f_{0} ||_{BV} + t P(|| e^{\theta^{\prime} |v|^{2}} f_{0} ||_{\infty}  )\}\\
  & \ + C_{ \delta} t [1+ P(|| e^{\theta^{\prime} |v|^{2}} f_{0} ||_{\infty}  )   ] \max_{i=m,m-1} \sup_{0 \leq s \leq t} || \partial f^{\varepsilon, i} (s)||_{1}.
   \end{split}
 \end{equation}
 
 Finally from (\ref{energy_est_1}) and (\ref{boundary_nonlin}), {chosing $\delta \ll 1$ and $T:=T(f_0)$ small enough, we have for all $0\le t\le T$}
 \begin{equation}\notag
 \begin{split}
 &\sup_{0 \leq s \leq t}|| \partial f^{\varepsilon, m+1}(s) ||_{1} + \int_{0}^{t} |\partial f^{\varepsilon, m+1}(s)|_{\gamma ,1} \\
 \leq & \ C  {\{|| f_{0} ||_{BV} +P ( || e^{\theta^{\prime} |v|^{2}} f_{0}||_{\infty}) \}}  \\  + & \frac{1}{8}\max_{i=m,m-1} \Big\{   \sup_{0 \leq s \leq t} ||\partial f^{\varepsilon, i} (s) ||_{1}+ \int_{0}^{t} |\partial f^{\varepsilon, i}|_{\gamma ,1}  \Big\}.
 \end{split}
 \end{equation}
Now we use the fact from \cite{EGKM}: Suppose $a_{i}\geq 0, D \geq 0$ and $A_{i} = \max \{a_{i},\cdots, a_{i-(k-1)}\}$ for fixed $k \in\mathbb{N}$. If $a_{m+1} \leq \frac{1}{8} A_{m} + D$ then $A_{m} \leq \frac{1}{8} A_{0} + \left(\frac{8}{7}\right)^{2} D$ for $m/k \gg 1.$ Hence
\begin{equation}\label{uniform_f}
  \sup_{0 \leq s \leq t} || \partial f ^{\varepsilon, m} (s) ||_{1} + 
\int_{0}^{t} |\partial f^{\varepsilon,m}(s)|_{\gamma,1}
\  \lesssim \ || f_{0}||_{BV}   + P (|| e^{\theta|v|^{2}}  f_{0} ||_{\infty})
 \ \ \ \ \text{for all } m\in\mathbb{N}.
\end{equation}

Now we pass the {to limit in $m$ and then in $\varepsilon$} to conclude the main theorem. It is standard that $BV(\Omega\times\mathbb{R}^{3})$ has  
$i)$ a \textit{compactness} property: 
\begin{equation}\label{compact_bulk}
\begin{split}
&\text{Suppose } f^{k} \in BV  \text{ and } \sup_{k}||f^{k} ||_{{BV}}< \infty \\ & \ \ \ \ \ \ \ \ \ \ \ \ \ \ \ \text{ then } \exists \ f\in BV \text{ with } f^{k } \rightarrow f \text{ in } \ L^{1}   \text{ up to subsequence,}
\end{split}
\end{equation}
and $ii)$ a \textit{lower semicontinuity} property:
\begin{equation}\label{lower_semi_bulk}
\text{Suppose } f^{k} \in   BV  \text{ and }  f^{k} \rightarrow f \text{ in }  L^{1}_{loc}   \  \text{ then }  \ 
|| f  ||_{\tilde{BV}} \leq \liminf_{k \rightarrow \infty} || f^{k} ||_{\tilde{BV}}. 
\end{equation}
Applying (\ref{compact_bulk}) and (\ref{lower_semi_bulk}) we conclude
\begin{equation}\notag
  \sup_{0 \leq s \leq t} ||   f   (s) ||_{BV}  
\  \lesssim \ || f_{0}||_{BV}   + P (|| e^{\theta|v|^{2}}  f_{0} ||_{\infty}) .
\end{equation}

For the boundary term we use the \textit{weak compactness of measures}: If $\sigma^{k}$ is a signed Radon measures on $\partial\Omega \times \mathbb{R}^{3}$ satisfying $\sup_{k} \sigma^{k}(\partial\Omega \times \mathbb{R}^{3})<\infty$ then there exists a Radon measure $\sigma$ such that $\sigma^{k}\rightharpoonup \sigma$ in $\mathcal{M}$.

More precisely we define, for any Lebesgue-measurable set $A\subset \partial\Omega \times \mathbb{R}^{3}$,  
\begin{eqnarray*}
\sigma^{\varepsilon, m} (A) &=&
 \Big(    \sigma_{x^{1}}^{\varepsilon, m} (A) ,
  \sigma_{x^{2}}^{\varepsilon, m} (A),
   \sigma_{x^{3}}^{\varepsilon, m} (A),
    \sigma_{v^{1}}^{\varepsilon, m} (A) ,
  \sigma_{v^{2}}^{\varepsilon, m} (A),
   \sigma_{v^{3}}^{\varepsilon, m} (A) 
 \Big)^{T}\\
  &: =& \int_{A} \nabla_{x,v} f^{\varepsilon,m}  \mathrm{d}\gamma
  \in \mathbb{R}^{6}.
\end{eqnarray*}
Then there exists a Radon measure $\sigma$ such that $\sigma^{\varepsilon,m} \rightharpoonup \sigma$ in $\mathcal{M}$, i.e. 
\begin{equation}\label{weak_measure}
\int_{\partial\Omega \times \mathbb{R}^{3}} g \partial f^{\varepsilon, m} \mathrm{d}\gamma\rightarrow   \int_{\partial\Omega \times \mathbb{R}^{3}} g \mathrm{d}\sigma \ \ \  \text{ for all } \ g \in C^{0}_{c}(\partial\Omega \times \mathbb{R}^{3}).
\end{equation}
It is standard (Hahn's decomposition theorem) to decompose $\sigma= \sigma_{+} -  \sigma_{-}$ with $\sigma_{\pm} \geq 0.$ Denote $|\sigma|_{\mathcal{M}(\gamma)} =  \sigma_{+}(\partial\Omega \times \mathbb{R}^{3}) +  \sigma_{-}(\partial\Omega \times \mathbb{R}^{3}).$ Then by the lower semicontinuity property of measures we have $| \sigma|_{\mathcal{M}(\gamma)} \leq \liminf | \sigma^{\varepsilon, m} |_{\mathcal{M}(\gamma)}= \liminf | \partial f^{\varepsilon,m} |_{L^{1}(\gamma)} \lesssim || f_{0}||_{BV} + P(|| e^{\theta |v|^{2}} f_{0} ||_{\infty}).$ Due to (\ref{weak_measure}), the (distributional) derivatives $\nabla_{x,v} f|_{\gamma}$ equals the Radon measure $\sigma$ on $\partial\Omega\times \mathbb{R}^{3}$ in the sense of distribution.

\end{proof}

\section*{{Appendix. $\mathfrak{S}_{\mathrm{B}}$ is a Co-Dimension 1 subset}}
We prove \textbf{Remark 1}. It suffices to show that $\mathfrak{S}_{\mathrm{B}} \cap \bar{\Omega} \times \mathbb{R}^{3}$ is a co-dimension 1 submanifold of $\bar{\Omega } \times \mathbb{R}^{3}$. More precisely we will show that if $(x_{0},v_{0}) \in \Omega \times \mathbb{R}^{3}$ satisfies $n(x_{\mathbf{b}}(x_{0},v_{0}))\cdot v_{0} =0$ and the boundary is strictly non-convex (\ref{nonconvex}) at $(x_{\mathbf{b}}(x_{0},v_{0}),v_{0})$ then there exists $0<\varepsilon\ll 1$ such that the following set is a 5 dimensional submanifold:
\begin{equation}\label{singularset3}
\big\{ (x,v) \in 
\mathfrak{S}_{\mathrm{B}}\cap B((x_{0},v_{0});\varepsilon):
\xb(x,v) \sim \xb(x_{0},v_{0})
\big\} \subset \bar{\Omega} \times \mathbb{R}^{3}.
\end{equation}

Without loss of generality we may assume $\xb(x_{0},v_{0}) =(0,0,0)= \mathbf{0}$ and $v_{0} = \mathbf{e}_{1}$ and $n(0,0,0) = - \mathbf{e}_{3}$ so that $\partial\Omega$ is locally a graph of a function $\eta : \mathbb{R}^{2} \rightarrow \mathbb{R}$ and $\nabla \eta (0,0)=\mathbf{0}$. Therefore the strictly non-convex condition (\ref{nonconvex}) at $(\xb(x_{0},v_{0}), v_{0})=(\mathbf{0}, \mathbf{e}_{1})$ implies 
\begin{eqnarray} 
\partial_{1}\partial_{1} \eta (0,0) \neq0.\label{nonconvex2} 
\end{eqnarray}

Clearly, $(\ref{singularset3})$ is contained in
\begin{eqnarray} 
\big\{(x+s v,v) \in \bar{\Omega}\times\mathbb{R}^3: x\in \partial\Omega, \ n(x) \cdot v =0 , \  (x,v)\sim ({x}_{0},{v}_{0}), \ s \in [0,\infty)\big\}.\label{singularset2} 
\end{eqnarray}

Consider $(x,v) \sim (x_{0},v_{0})$. We choose two basis for the tangent space:
\begin{eqnarray*}
{\tau}_1 & =&  \frac{1}{\sqrt{1+|\nabla {\eta}|^2}} \left(\begin{array}{ccc} 1 \\ 0 \\ \partial_{1}{\eta}\end{array}\right), \\
{\tau}_2  & =&  \frac{1}{\sqrt{1+|\nabla{\eta}|^2 }\sqrt{1+(\partial_{1}{\eta})^2}} \left(\begin{array}{ccc}
-\partial_{1} {\eta} \partial_{2} {\eta} \\
1+ (\partial_{1}{\eta})^2 \\
\partial_{2} {\eta} \end{array}\right).\end{eqnarray*}

For $( {x}_1, {x}_2, \theta, r_{v},s) \in \mathbb{R}^2 \times [0,2\pi)\times [0,\infty)\times [0,\infty)$ we write $(x+sv,v)$ in (\ref{singularset2}) as
\begin{equation}\nonumber
\begin{split}
 & {X}( {x}_1, {x}_2, \theta, r_{v},s) :=  \left(\begin{array}{ccc}  {x}_1 \\  {x}_2 \\  {\eta}( {x_1}, {x}_2)
\end{array}\right) +s r_{v} \cos\theta \  {\tau}_1 ( {x}_1, {x}_2) +s r_{v} \sin\theta \  {\tau}_2( {x}_1,  {x}_2),\\
&  {V}( {x}_1, {x}_2, \theta, r_{v},s):= r_{v}  \cos\theta \  {\tau}_1 ( {x}_1, {x}_2) + r_{v} \sin\theta \  {\tau}_2 ( {x}_1, {x}_2).
\end{split}
\end{equation}
%It is clear that the set (\ref{singularset_part}) is contained in
%\begin{equation}\label{singularset_part2} 
%\Big\{
%\big( {X}( {x}_1, {x}_2, \theta, r_{v},s) ,  {V}( {x}_1, {x}_2, \theta, r_{v},s)   \big) \in \bar{\Omega}\times\mathbb{R}^3: ( {x}_1, {x}_2)\sim (0,0), \  {\theta}\sim 0, \  r_{v} \in [0,\infty),
%\ s \in [0,\infty)
%\Big\}. 
%\end{equation}
In order to prove Remark 1 it suffices to show that the followings are linearly independent
\begin{equation}\nonumber
\begin{split}
 \binom{ \partial_{ {x}_1} {X}}{  \partial_{ {x}_1} {V}}, \ \binom{ \partial_{ {x}_2} {X}}{ \partial_{ {x}_2} {V}}, \ \binom{ \partial_\theta  {X}}{  \partial_\theta  {V}}
 ,  \ \binom{ \partial_s  {X}}{ \partial_s  {V}}, \ \binom{ \partial_{ r_{v}} {X}}{ \partial_{ r_{v}} {V}} \in \mathbb{R}^{6}.
\end{split}
\end{equation}
It suffices to show that the normal is non-vanishing:
\begin{equation}\nonumber
\begin{split}
\mathcal{N}:=   \det\left(\begin{array}{cccccc}
\mathbf{e}_1 & \mathbf{e}_2 & \mathbf{e}_3 & \mathbf{e}_4 & \mathbf{e}_5 & \mathbf{e}_6\\
\partial_{ {x}_1}  {X}_1  & \partial_{ {x} _1}  {X}_2& \partial_{  {x}_1}  {X}_3 & \partial_{ {x}_1}   {V}_1 & \partial_{  {x}_1}  {V}_2 & \partial_{  {x}_1}  {V}_3\\
\partial_{ {x}_2} {X}_1  & \partial_{  {x}_2}  {X}_2& \partial_{  {x}_2}  {X}_3 & \partial_{  {x}_2}   {V}_1 & \partial_{  {x}_2}  {V}_2 & \partial_{  {x}_2}  {V}_3\\
\partial_\theta {X}_1  & \partial_\theta  {X}_2& \partial_\theta {X}_3 & \partial_{\theta}  {V}_1 & \partial_{\theta}  {V}_2 & \partial_{\theta}  {V}_3\\
\partial_s  {X}_1 & \partial_{s}  {X}_2 & \partial_{s}  {X}_3 & 0 &  0 & 0\\
 \partial_{r_{v}}  {X}_1 & \partial_{r_{v}}  {X}_2 & \partial_{r_{v}}  {X}_3 & \partial_{r_{v}}  {V}_1 & \partial_{r_{v}}  {V}_2 & \partial_{r_{v}}  {V}_3
\end{array}\right).
\end{split}
\end{equation}

To compute the normal we need to know
\begin{equation}\nonumber
\begin{split}
\partial_{1} {\tau}_1( {x}_1, {x}_2) &= \frac{\partial_{1}^2 {\eta}}{[1+(\nabla{\eta})^2]^{3/2}} \left(\begin{array}{ccc} -\partial_{1} {\eta} \\ 0 \\ 1 \end{array}\right)
+ \frac{\partial_{2} \eta}{[1+( \nabla{\eta})^2]^{3/2}} 
\left(\begin{array}{c}
0 \\ 0 \\ \partial_{2} \eta \partial_{1}^{2} \eta - \partial_{1} \eta \partial_{1} \partial_{2} \eta
\end{array}
\right)
,\\
\partial_{2} {\tau}_1({x}_1,{x}_2) &=
\frac{1}{ [1+( \nabla\eta)^2]^{1/2} } \left(\begin{array}{c}0\\ 0 \\ \partial_{1} \partial_{2} \eta \end{array} \right) 
- \frac{1}{ [1+( \nabla\eta)^2]^{3/2}}
\left(\begin{array}{c} 
\nabla \eta \cdot \nabla \partial_{2} \eta \\ 0 \\ \partial_{1} \eta \nabla \eta \cdot \nabla \partial_{2} \eta
\end{array}\right)
,\\
\partial_{1}({\tau}_2)_1 & = \frac{(\partial_{1}{\eta})^2 \partial_{2}{\eta} \partial_{1}^2 {\eta}}{[1+(\partial_{1}{\eta})^2]^{3/2}[1+|\nabla {\eta}|^2]^{1/2}} + \frac{ (\partial_{1}{\eta})^2 \partial_{2}{\eta} \partial_{1}^2 {\eta} + \partial_{1} {\eta} (\partial_{2}{\eta})^2 \partial_{1}\partial_{2}{\eta} }{[1+(\partial_{1}{\eta})^2]^{1/2}[1+|\nabla {\eta}|^2]^{3/2}}\\
& \  - \frac{ \partial_{1}^2 {\eta} \partial_{2} {\eta} + \partial_{1} {\eta} \partial_{1}\partial_{2}{\eta} }{[1+(\partial_{1}{\eta})^2]^{1/2}[1+|\nabla {\eta}|^2]^{1/2}},\\
\partial_{2}({\tau}_2)_1 & = \frac{\partial_{1}{\eta} \partial_{2}{\eta} \partial_{1}\partial_{2}{\eta}}{[1+(\partial_{1}{\eta})^2]^{3/2}[1+|\nabla {\eta}|^2]^{1/2}} + \frac{ (\partial_{1}{\eta})^2 \partial_{2}{\eta} \partial_{1}\partial_{2} {\eta} + \partial_{1} {\eta} (\partial_{2}{\eta})^2 \partial_{2}^2{\eta} }{[1+(\partial_{1}{\eta})^2]^{1/2}[1+|\nabla {\eta}|^2]^{3/2}}\\
& \  - \frac{ \partial_{1}\partial_{2} {\eta} \partial_{2} {\eta} + \partial_{1} {\eta} \partial_{2}^2{\eta} }{[1+(\partial_{1}{\eta})^2]^{1/2}[1+|\nabla {\eta}|^2]^{1/2}},\\
\partial_{1}(\tau_2)_2 & = \frac{\partial_{1}{\eta} \partial_{1}^2 {\eta}}{[1+(\partial_{1}{\eta})^2]^{1/2}[1+|\nabla {\eta}|^2]^{1/2}} - \frac{[1+(\partial_{1}{\eta})^2]^{1/2}[ \partial_{1} {\eta} \partial_{1}^2 {\eta} + \partial_{2}{\eta} \partial_{1}\partial_{2}{\eta} ]}{[1+|\nabla {\eta}|^2]^{3/2}},\\
\partial_{2}(\tau_2)_2 & = \frac{\partial_{1}{\eta} \partial_{1}\partial_{2} {\eta}}{[1+(\partial_{1}{\eta})^2]^{1/2}[1+|\nabla {\eta}|^2]^{1/2}} - \frac{[1+(\partial_{1}{\eta})^2]^{1/2}[ \partial_{1} {\eta} \partial_{1}\partial_{2} {\eta} + \partial_{2}{\eta} \partial_{2}^2{\eta}] }{[1+|\nabla {\eta}|^2]^{3/2}},\\
\partial_{1}(\tau_2)_3 & = -\frac{\partial_{1}{\eta} \partial_{2}{\eta} \partial_{1}^2{\eta}}{[1+(\partial_{1}{\eta})^2]^{3/2}[1+|\nabla {\eta}|^2]^{1/2}} - \frac{ \partial_{1}{\eta}  \partial_{2}{\eta} \partial_{1}^2 {\eta} + (\partial_{2}{\eta})^2 \partial_{1}\partial_{2}{\eta} }{[1+(\partial_{1}{\eta})^2]^{1/2}[1+|\nabla {\eta}|^2]^{3/2}}\\
& \  + \frac{ \partial_{1}\partial_{2} {\eta} }{[1+(\partial_{1}{\eta})^2]^{1/2}[1+|\nabla {\eta}|^2]^{1/2}},\\
\partial_{2}(\tau_2)_3 & = -\frac{\partial_{1}{\eta} \partial_{2}{\eta} \partial_{2}^2{\eta}}{[1+(\partial_{1}{\eta})^2]^{3/2}[1+|\nabla {\eta}|^2]^{1/2}} - \frac{ \partial_{1}{\eta}  \partial_{2}{\eta} \partial_{1}\partial_{2} {\eta} + (\partial_{2}{\eta})^2 \partial_{2}^2{\eta} }{[1+(\partial_{1}{\eta})^2]^{1/2}[1+|\nabla {\eta}|^2]^{3/2}}\\
& \  + \frac{  \partial_{2}^2 {\eta} }{[1+(\partial_{1}{\eta})^2]^{1/2}[1+|\nabla {\eta}|^2]^{1/2}}.
\end{split}
\end{equation}

We evaluate the normal at $({x}_1,{x}_2,\theta,s,r_{v})=(0,0,0, s,r_{v})$. Since $\partial_{1}{\eta}(0,0)=0= {\partial}_{2} \eta(0, 0)$,
\begin{equation}\nonumber
\begin{split}
&n(0, 0) = \mathbf{e}_3, \ \ {\tau}_1(0, 0) = \mathbf{e}_1, \ \ {\tau}_2(0, 0) = \mathbf{e}_2, \\
& \partial_{1} {\tau}_1(0, 0) = \partial_{1} \partial_{1} {\eta}(0, 0) \mathbf{e}_3, \ \ \partial_{2} {\tau}_1(0, 0)= \partial_{1} \partial_{2} {\eta} (0, 0) \mathbf{e}_3,\\
& \partial_{1} {\tau}_2(0,0)= \partial_{1} \partial_{2} {\eta}(0,0) \mathbf{e}_3 , \ \ \partial_{2} {\tau}_2(0,0) = \partial_{2} \partial_{2} \eta(0, 0) \mathbf{e}_3.
\end{split}
\end{equation}
Due to (\ref{nonconvex2}) we have
\begin{equation}\notag
\begin{split}
{\mathcal{N}}(0,0,0,s, r_{v})  =\det \left(\begin{array}{cccccc}
\mathbf{e}_1 & \mathbf{e}_2 & \mathbf{e}_3 & \mathbf{e}_4 & \mathbf{e}_5 & \mathbf{e}_6\\
1 & 0 & -s\partial_{ 1} \partial_{1} {\eta} & 0 & 0 & -r_{v} \partial_{1} \partial_{1} {\eta}\\
0 & 1 & -s \partial_{1}\partial_{2} {\eta} & 0 & 0 &  -r_{v}\partial_{1}\partial_{2}{\eta}
\\
0 & s& 0 & 0 & r_{v} & 0 \\
r_{v} & 0 & 0 & 0 & 0 & 0\\
s & 0 & 0 & 1 & 0 & 0
\end{array}\right) 
 =
 \left(\begin{array}{cccc}
0 \\ 0 \\ r_{v}^2 \partial_{1}  \partial_{1}  {\eta}(0,0) \\ 0 \\ 0 \\ sr_{v} \partial_{1}  \partial_{1}  {\eta}(0,0)
\end{array}\right) \neq 0.
\end{split}
\end{equation}
Therefore $\mathcal{N}(x_{1},x_{2}, \theta, s, r_{v}) \neq 0$ for $(x_{1},x_{2},\theta) \sim (0,0,0)$. This proves the claim.

\bigskip

\noindent \textbf{Acknowledgements}: 
This project was initiated during the Kinetic Program at ICERM, 2011. Y. Guo's research is supported in part by NSFC grant $\#$10828103 and NSF grant $\#$DMS-1209437. C.Kim's research is supported in part by the Herchel Smith fund at the University of Cambridge. He thanks Brown University and the Academia Sinica for the kind hospitality and support during his stay. A. Trescases thanks the Division of Applied Mathematics, Brown University for the kind hospitality during her visit.

\vspace{20pt}
 
\noindent{Division of Applied Mathematics, Brown University,
Providence, RI 02812, U.S.A., Yan$\underline{~}$Guo@brown.edu }\newline

\noindent{Department of Mathematics,
University of Wisconsin, Madison, WI 53706, U.S.A, ckim@math.wisc.edu}\newline

\noindent{{CEREMADE, Universit\'e Paris Dauphine, Place du Mar\'echal De Lattre De Tassigny, 75775,  Paris, France, tonon@ceremade.dauphine.fr} \newline}

\noindent{CMLA, ENS Cachan, 61, avenue du Pr\'{e}sident Wilson, 94235
Cachan, France, trescase@cmla.ens-cachan.fr}

 \end{document}